\documentclass[journal]{IEEEtran}
\ifCLASSINFOpdf
\else
\fi

\usepackage{graphicx}          

\usepackage[dvips]{epsfig}    
\usepackage{latexsym,amssymb}
\usepackage{amsmath, amsbsy, amsthm}
\usepackage{amsopn, amstext}
\usepackage{colortbl}
\usepackage{cite}
\newtheorem{remark}{Remark}
\newtheorem{theorem}{Theorem}
\newtheorem{definition}{Definition}
\newtheorem{lemma}{Lemma}

\newtheorem{problem}{Problem}
\newtheorem{assumption}{Assumption}

\allowdisplaybreaks
\begin{document}
%
\title{A Complete Solution to Optimal Control and Stabilization for Mean-field Systems: Part I, Discrete-time Case }
%
%
%

\author{Huanshui~Zhang$^{\ast}$
        and~Qingyuan~Qi~\IEEEmembership{}
\thanks{This work is supported by the National Science Foundation of China under Grants 61120106011,61573221, 61633014.}
\thanks{H. Zhang and Q. Qi are with School of Control Science and Engineering, Shandong University,
Jinan 250061, P.R. China. H. Zhang is the corresponding author(hszhang@sdu.edu.cn).}
}

\maketitle

\begin{abstract}
Different from most of the previous works, this paper provides a thorough solution to the fundamental problems of linear-quadratic (LQ) control and stabilization for discrete-time mean-field systems under basic assumptions. Firstly, the sufficient and necessary condition for the solvability of mean-field LQ control problem is firstly presented in analytic expression based on the maximum principle developed in this paper, which is compared with the results obtained in literatures where only operator type solvability conditions were given. The optimal controller is given in terms of a coupled Riccati equation which is derived from the solution to forward and backward stochastic difference equation (FBSDE). Secondly, the sufficient and necessary stabilization conditions are explored. It is shown that, under exactly observability assumption, the mean-field system is stabilizable in mean square sense if and only if a coupled algebraic Riccati equation (ARE) has a unique solution $P$ and $\bar{P}$ satisfying $P>0$ and $P+\bar{P}>0$. Furthermore, under the exactly detectability assumption, which is a weaker assumption than exactly observability, we show that the mean-field system is stabilizable in mean square sense if and only if the coupled ARE has a unique solution $P$ and $\bar{P}$ satisfying $P\geq 0$ and $P+\bar{P}\geq 0$. The key techniques adopted in this paper are the maximum principle and the solution to the FBSDE obtained in this paper. The derived results in this paper forms the basis to solve the mean-field control problem for continuous-time systems \cite{con} and other related problems.
\end{abstract}


%
\IEEEpeerreviewmaketitle

\section{Introduction}

In this paper, the mean-field linear quadratic optimal control and stabilization problems are considered for discrete-time case. Different from the classical stochastic control problem, mean-field terms appear in system dynamics and cost function, which combines mean-field theory with stochastic control problems. Mean-field stochastic control problem has been a hot research topic since 1950s. System state is described by a controlled mean-field stochastic differential equation (MF-SDE), which was firstly proposed in \cite{kac}, and the initial study of MF-SDEs was given by reference \cite{Mckean}. Since then, many contributions have been made in studying MF-SDEs and related topics by many researchers. See, for example, \cite{dawson}-\cite{gartner} and the references cited therein. The recent development for  mean-field control problems  can be found in \cite{buck1}, \cite{buck2}, \cite{huangjh}, \cite{yong}, \cite{ni1}, \cite{ni2} and references therein.

Reference \cite{yong} dealt with the continuous-time finite horizon mean-field LQ control problem, a sufficient and necessary solvability condition of the problem was presented in terms of operator criteria.  By using decoupling technique, the optimal controller was designed via two Riccati equations. Furthermore, the continuous-time mean-field LQ control and stabilization problem for infinite horizon was investigated in \cite{huangjh}, the equivalence of several notions of stability for mean-field system was established. It was shown that the optimal mean-field LQ controller for infinite horizon case can be presented via AREs.

For discrete-time mean-field LQ control problem, \cite{ni1} and \cite{ni2} studied the finite horizon case and infinite horizon case respectively. In \cite{ni1}, a necessary and sufficient solvability condition for finite horizon discrete-time mean-field LQ control problem was presented in operator type. Furthermore, under stronger conditions, the explicit optimal controller was derived using matrix dynamical optimization method, which is in fact a sufficient solvability solution to the discrete-time mean-field LQ control problem \cite{ni1}. Besides, for the infinite time case, the equivalence of $L^{2}$ open-loop stabilizability and $L^{2}$ closed-loop stabilizability was studied. Also the stabilizing condition was investigate in \cite{ni2}.

However, it should be highlighted that the LQ control and stabilization problems for mean-field systems remain to be further investigated although major progresses have been obtained in the above works \cite{ni1}, \cite{huangjh}, \cite{ni2}, \cite{yong} and references therein. The basic reasons are twofold: Firstly, the solvability for the LQ control was given in terms of operator type condition \cite{ni1}, which is difficult to be verified in practice; Secondly, the stabilization control problem of the mean-field system has not been essentially solved  as only sufficient conditions of stabilization were given in the previous works.

In this paper, we aim to provide a complete solution to the problems of optimal LQ control and stabilization for discrete-time mean-field systems. Different from previous works, we will derive the maximum principle (MP) for discrete-time mean-field LQ control problem which is new to the best of our knowledge. Then, by solving the coupled state equation (forward) and the costate equation (backward), the optimal LQ controller is obtained from the equilibrium condition naturally, and accordingly the sufficient and necessary solvability condition is explored in explicit expression. The controller is designed via a coupled Riccati equation which is derived from the solution to the FBSDE, and posses the similarity with the case of standard  LQ control. Finally, with convergence analysis on the coupled Riccati equation, the infinite horizon LQ controller and the stabilization condition (sufficient and necessary) is explored by defining the Lyapunov function with the optimal cost function. Two stabilization results are obtained under two different assumptions. One is under the standard assumption of exactly observability, it is shown that the mean-field system is stabilizable in mean square sense if and only if a coupled ARE has a unique solution $P$ and $\bar{P}$ satisfying $P>0$ and $P+\bar{P}>0$. The other one is under a weaker  assumption of exactly detectability, it is shown that the mean-field system is stabilizable in mean square sense if and only if the coupled ARE admits a unique solution $P$ and $\bar{P}$ satisfying $P\geq 0$ and $P+\bar{P}\geq 0$.

It should be pointed out that the presented results are parallel to the solution of the standard stochastic LQ with similar results such as controller design and stabilization conditions under the same assumptions on system and weighting matrices. In particular, the weighting matrices $R_{k}$ and $R_{k}+\bar{R}_{k}$ are only required to be positive semi-definite for optimal controller designed in this paper. It is more standard than the previous works \cite{ni1}, where the matrices are assumed to be positive definite.

A preliminary version of this paper was submitted as in \cite{cdc}, in which the finite horizon optimal control for mean-field system was considered. In this paper, both the finite horizon control problem and infinite horizon optimal control and stabilization problems are investigated. The remainder of this paper is organized as follows. Section II presents the maximum principle and the solution to finite horizon mean-field LQ control. In Section III, the infinite horizon optimal control and stabilization problems are investigated. Numerical examples are given in Section IV to illustrate main results of this paper. Some concluding remarks are given in Section V. Finally, relevant proofs are detailed in Appendices.

Throughout this paper, the following notations and definitions are used.

\textbf{Notations and definitions}: $I_{n}$ means the unit matrix with rank $n$; Superscript $'$ denotes the transpose of a matrix. Real symmetric matrix $A>0$ (or $\geq 0$) implies that $A$ is strictly positive definite (or positive semi-definite). $\mathcal{R}^{n}$ signifies the $n$-dimensional Euclidean space. $B^{-1}$ is used to indicate the inverse of real matrix $B$. $\{\Omega,\mathcal{F},\mathcal{P},\{\mathcal{F}_{k}\}_{k\geq 0}\}$ represents a complete probability space, with natural filtration $\{\mathcal{F}_{k}\}_{k\geq 0}$ generated by $\{x_{0},w_{0},\cdots,w_{k}\}$ augmented by all the $\mathcal{P}$-null sets. $E[\cdot|\mathcal{F}_{k}]$ means the conditional expectation with respect to $\mathcal{F}_{k}$ and $\mathcal{F}_{-1}$ is understood as $\{\emptyset, \Omega\}$.

\begin{definition}\label{def}
For random vector $x$, if $E(x'x)=0$, we call it zero random vector, i.e., $x=0$.
\end{definition}

\section{Finite Horizon Mean-field LQ Control Problem}

\subsection{Problem Formulation and Preliminaries }
\subsubsection{Problem Formulation}

Consider the following discrete-time mean-field system
\begin{equation}\label{ps1}
\left\{ \begin{array}{ll}
x_{k+1}=(A_{k}x_{k}+\bar{A}_{k}Ex_{k}+B_{k}u_{k}+\bar{B}_{k}Eu_{k})\\
~~~~~~~+(C_{k}x_{k}+\bar{C}_{k}Ex_{k}+D_{k}u_{k}+\bar{D}_{k}Eu_{k})w_{k},\\
x_{0}=\xi,\\
\end{array} \right.
\end{equation}
where $A_{k},\bar{A}_{k},C_{k},\bar{C}_{k}\in \mathcal{R}^{n\times n}$, and  $B_{k},\bar{B}_{k},D_{k},\bar{D}_{k}\in
\mathcal{R}^{n\times m}$, all the coefficient matrices are given deterministic. $x_{k}\in\mathcal{R}^{n}$ is the state process and
$u_{k}\in \mathcal{R}^{m}$ is the control process. The system noise $\{w_{k}\}_{k=0}^{N}$ is scalar valued random white noise with zero
mean and variance $\sigma^{2}$. $E$ is the expectation taken over the noise $\{w_{k}\}_{k=0}^{N}$ and initial state $\xi$. Denote
$\mathcal{F}_{k}$ as the natural filtration generated by $\{\xi,w_{0},\cdots,w_{k}\}$ augmented by all the $\mathcal{P}$-null sets.

By taking expectations on both sides of \eqref{ps1}, we obtain
\begin{equation}\label{ps20}
  Ex_{k+1}=(A_{k}+\bar{A}_{k})Ex_{k}+(B_{k}+\bar{B}_{k})Eu_{k}.
\end{equation}

The cost function associated with system equation \eqref{ps1} is given by:
\begin{align}\label{ps2}
  J_{N}&=\sum_{k=0}^{N}E\Big[x_{k}'Q_{k}x_{k}+(Ex_{k})'\bar{Q}_{k}Ex_{k}\notag\\
  &~~+u_{k}'R_{k}u_{k}+(Eu_{k})'\bar{R}_{k}Eu_{k}\Big]\notag\\
  &~~+E(x_{N\hspace{-0.5mm}+\hspace{-0.5mm}1}'P_{N\hspace{-0.5mm}+\hspace{-0.5mm}1}x_{N\hspace{-0.5mm}+\hspace{-0.5mm}1})
  \hspace{-1mm}+\hspace{-1mm}(Ex_{N\hspace{-0.5mm}+\hspace{-0.5mm}1})'\bar{P}_{N\hspace{-0.5mm}+\hspace{-0.5mm}1}Ex_{N\hspace{-0.5mm}+\hspace{-0.5mm}1},
\end{align}
where $Q_{k},\bar{Q}_{k},R_{k},\bar{R}_{k}$, $P_{N+1},\bar{P}_{N+1}$ are deterministic symmetric matrices with compatible dimensions.

The finite horizon mean-field LQ optimal control problem is stated as follows:
\begin{problem}\label{prob1}
For system \eqref{ps1} associated with cost function \eqref{ps2}, find $\mathcal{F}_{k-1}$-measurable controller $u_{k}$ such that
\eqref{ps2} is minimized.
\end{problem}

To guarantee the solvability of \emph{Problem \ref{prob1}}, the following standard assumption is made as follows.
\begin{assumption}\label{ass1}
The weighting matrices in \eqref{ps2} satisfy $Q_{k}\geq 0$, $Q_{k}+\bar{Q}_{k}\geq 0$, $R_{k}\geq 0$, $R_{k}+\bar{R}_{k}\geq 0$ for
$0\leq k \leq N$ and $P_{N+1}\geq 0$, $P_{N+1}+\bar{P}_{N+1}\geq 0$.
\end{assumption}
\subsubsection{Preliminaries}
In order to solve the above problem, a basic result is firstly presented as below.
\begin{lemma}\label{lemma01}
For any random vector $x\not=0$, i.e., $E(x'x)\neq 0$ as defined in Definition \ref{def}, $E(x'Mx)\geq0$, if and only if~$M\geq 0$,
where $M$ is a real symmetric matrix.
\end{lemma}

\begin{proof} The proof is straightforward and is omitted here.
\end{proof}

\begin{remark}\label{rem1}
From Lemma \ref{lemma01}, immediately we have

1) For any $x$ satisfying $x=Ex\neq 0$, i.e., $x$ is deterministic, $x'Mx\geq 0$ if and only if $M\geq 0$.

2) For any random vector $x$ satisfying $Ex=0$ and $x\neq 0$, $E(x'Mx)\geq 0$ if and only if $M\geq 0$.
\end{remark}

\begin{remark}\label{rem2}
Note that Lemma \ref{lemma01} and Remark \ref{rem1} also hold if ``$\geq$" in the conclusion is replaced by ``$\leq$", ``$<$", ``$>$" or ``$=$".
\end{remark}

\subsection{Maximum Principle}
In this subsection, we will present a general result for the maximum principle of general mean-field stochastic control problem which is the base to solve the problems studied in this paper.

Consider the general discrete-time mean-field stochastic systems
\begin{equation}\label{ps3}
x_{k+1}=f^{k}(x_{k},u_{k},Ex_{k},Eu_{k},w_{k}),
\end{equation}
where $x_{k}$ and $u_{k}$ are the system state and control input, respectively. $Ex_{k}$, $Eu_{k}$ are expectation of $x_{k}$ and $u_{k}$. Scalar-valued $w_{k}$ is the random white noise with zero mean and variance $\sigma^{2}$. $f^{k}(x_{k},u_{k},Ex_{k},Eu_{k},w_{k})$, in general, is a nonlinear function.

The corresponding scalar performance index is given in the general form
\begin{equation}\label{ps04}
  J_{N}\hspace{-1mm}=\hspace{-1mm}E\Big\{\phi(x_{N+1},Ex_{N+1})\hspace{-1mm}+\hspace{-1mm}\sum_{k=0}^{N}L^{k}(x_{k},u_{k},Ex_{k},Eu_{k})\Big\},
\end{equation}
where $\phi(x_{N+1},Ex_{N+1})$ is a function of the final time $N+1$, $x_{N+1}$ and $Ex_{N+1}$. $L^{k}(x_{k},u_{k},Ex_{k},Eu_{k})$ is a function of $x_{k},Ex_{k},u_{k},Eu_{k}$ at each time $k$.

From system \eqref{ps3}, we have that
\begin{align}\label{mp0002}
  Ex_{k+1}&=E[f^{k}(x_{k},u_{k},Ex_{k},Eu_{k},w_{k})]\notag\\
  &=g^{k}(x_{k},u_{k},Ex_{k},Eu_{k}),
\end{align}
where $g^{k}(x_{k},u_{k},Ex_{k},Eu_{k})$ is deterministic function.

The general maximum principle (necessary condition) to minimize \eqref{ps04} is given in the following theorem.

\begin{theorem}\label{maximum}
The necessary conditions for the minimizing $J_{N}$ is given as,
\begin{align}
0\hspace{-1mm}=\hspace{-1mm}E\hspace{-1mm}\left\{\hspace{-1mm}
(L^k_{u_k})'\hspace{-1mm}+\hspace{-1mm}E(L_{Eu_{k}}^{k})'\hspace{-1mm}+\hspace{-1mm}\left[\hspace{-2mm}
  \begin{array}{cc}
   f_{u_{k}}^{k} \hspace{-2mm}\\
   g_{u_{k}}^{k}     \hspace{-2mm}           \\
  \end{array}
\hspace{-2mm}\right]'\hspace{-1mm}\lambda_{k}\hspace{-1mm}+\hspace{-1mm}E\left\{\hspace{-1mm}\left[\hspace{-2mm}
  \begin{array}{cc}
   f_{Eu_{k}}^{k} \hspace{-1mm}\\
    g_{Eu_{k}}^{k}     \hspace{-1mm}           \\
  \end{array}
\hspace{-2mm}\right]'\hspace{-1mm}\lambda_{k}\right\}\hspace{-1mm}\Bigg| {\mathcal F}_{k-1}\hspace{-1mm}\right\},\label{ps43}
\end{align}
where costate $\lambda_{k}$ obeys
\begin{align}
&\lambda_{k-1}\hspace{-1mm}=\hspace{-1mm}E\Big\{\left[\hspace{-2mm}
  \begin{array}{cc}
    I_{n}\hspace{-1mm}\\
    0     \hspace{-1mm}           \\
  \end{array}
\hspace{-2mm}\right][L^k_{x_k}\hspace{-1mm}+\hspace{-1mm}E(L_{Ex_{k}}^{k})]'\hspace{-1mm}+\hspace{-1mm}[\tilde{f}^k_{x_k}]'\lambda_{k}\Big|\mathcal{F}_{k-1}\Big\},\label{ps41}
\end{align}
with final condition
\begin{equation}\label{ps42}
\lambda_{N}\hspace{-1mm}=\hspace{-1mm}\left[\hspace{-2mm}
  \begin{array}{cc}
   (\phi_{x_{N+1}})'+E(\phi_{Ex_{N+1}})'\hspace{-1mm} \\
    0            \hspace{-1mm}   \\
  \end{array}
\hspace{-2mm}\right],\\
\end{equation}
where
\[f_{x_{k}}^{k}=\frac{\partial f_{k}}{\partial x_{k}},~~f_{u_{k}}^{k}=\frac{\partial f_{k}}{\partial
u_{k}},f_{Ex_{k}}^{k}=\frac{\partial f_{k}}{\partial Ex_{k}},
f_{Eu_{k}}^{k}=\frac{\partial f_{k}}{\partial Eu_{k}},\]
\[g_{x_{k}}^{k}=\frac{\partial g_{k}}{\partial x_{k}},~~g_{u_{k}}^{k}=\frac{\partial g_{k}}{\partial
u_{k}},g_{Ex_{k}}^{k}=\frac{\partial g_{k}}{\partial Ex_{k}},
g_{Eu_{k}}^{k}=\frac{\partial g_{k}}{\partial Eu_{k}},\]
and
\begin{align*}
&\phi_{Ex_{N+1}}\hspace{-1.1mm}=\hspace{-1.1mm}\frac{\partial \phi(x_{N+1},\hspace{-0.3mm}Ex_{N+1})}{\partial Ex_{N+1}}, \hspace{-0.5mm} \phi_{x_{N+1}}\hspace{-1.1mm}=\hspace{-1.1mm}\frac{\partial \phi(x_{N+1},\hspace{-0.3mm}Ex_{N+1})}{\partial x_{N+1}},\\
&L_{x_{k}}^{k}\hspace{-0.3mm}=\hspace{-0.3mm}\frac{\partial L_{k}}{\partial x_{k}},L_{u_{k}}^{k}\hspace{-0.3mm}=\hspace{-0.3mm}\frac{\partial L_{k}}{\partial u_{k}},L_{Ex_{k}}^{k}\hspace{-0.3mm}=\hspace{-0.3mm}\frac{\partial L_{k}}{\partial Ex_{k}},L_{Eu_{k}}^{k}\hspace{-0.3mm}=\hspace{-0.3mm}\frac{\partial L_{k}}{\partial Eu_{k}},\\
&\tilde{f}^{k}_{x_{k}}=\left[
  \begin{array}{cc}
    f_{x_{k}}^{k} & f_{Ex_{k}}^{k} \\
    g_{x_{k}}^{k} & g_{Ex_{k}}^{k}\\
  \end{array}
\right],~k=0,\cdots,N.
\end{align*}
\end{theorem}

\begin{proof}
See Appendix A.
\end{proof}

\subsection{Solution to Problem \ref{prob1}}

Following Theorem \ref{maximum}, it is easy to obtain the following maximum principle for system \eqref{ps1} associated with the cost function \eqref{ps2}.
\begin{lemma}\label{lemma1}
The necessary condition of minimizing  \eqref{ps2} for system \eqref{ps1} can be stated as:
\begin{align}
 0&=E\Big\{
R_{k}u_{k}\hspace{-1mm}+\hspace{-1mm}\bar{R}_{k}Eu_{k}\hspace{-1mm}+\hspace{-1mm}\left[\hspace{-2mm}
  \begin{array}{cc}
   B_{k}+w_{k}D_{k} \hspace{-1mm}\\
    0     \hspace{-1mm}           \\
  \end{array}
\hspace{-2mm}\right]'\lambda_{k}\hspace{-1mm}\notag\\
&+E\Big\{\left[\hspace{-2mm}
  \begin{array}{cc}
   \bar{B}_{k}+w_{k}\bar{D}_{k} \hspace{-1mm}\\
    B_{k}+\bar{B}_{k}     \hspace{-1mm}           \\
  \end{array}
\hspace{-2mm}\right]'\lambda_{k}\Big\}\Big| {\mathcal F}_{k-1}\Big\},\label{th33}
\end{align}
where costate $\lambda_{k}$ satisfies the following iteration
\begin{align}
\lambda_{k-1}&=E\Big\{\left[\hspace{-2mm}
  \begin{array}{cc}
    Q_{k}x_{k}+\bar{Q}_{k}Ex_{k}\hspace{-1mm}\\
     0     \hspace{-1mm}           \\
  \end{array}
\hspace{-2mm}\right]\hspace{-1mm}\notag\\
&+\left[
  \begin{array}{cc}
   \hspace{-1mm} A_{k}+w_{k}C_{k} &  \bar{A}_{k}+w_{k}\bar{C}_{k}\hspace{-1mm}\\
   \hspace{-1mm}  0     & A_{k}+\bar{A}_{k}\hspace{-1mm}      \\
  \end{array}
\right]'\lambda_{k}\Big|\mathcal{F}_{k-1}\Big\},\label{th32}
\end{align}
with final condition
\begin{equation}\label{th31}
  \lambda_{N}=\left[
  \begin{array}{cc}
    P_{N+1}& \bar{P}_{N+1}^{(1)}\\
     \bar{P}_{N+1}^{(2)} & \bar{P}_{N+1}^{(3)}     \\
  \end{array}
\right]\left[
  \begin{array}{cc}
   x_{N+1}\\
     Ex_{N+1}\\
  \end{array}
\right],
\end{equation}
where $\bar{P}_{N+1}^{(1)}=\bar{P}_{N+1}$, $\bar{P}_{N+1}^{(2)}=\bar{P}_{N+1}^{(3)}=0$, $P_{N+1}$ and $\bar{P}_{N+1}$ are given by the cost function \eqref{ps2}.

\end{lemma}

In Lemma \ref{lemma1}, $\lambda_{k}$ ($1\leq k\leq N$) is costate and (\ref{th32}) is costate-state equation. \eqref{th32} and state equation \eqref{ps1} form the FBSDE system.  \eqref{th33} is termed as equilibrium equation (condition).

The main result of this section is stated as below.
\begin{theorem}\label{main}
Under Assumption \ref{ass1}, \emph{ Problem 1} has a unique solution if and only if $\Upsilon_{k}^{(1)}$ and $\Upsilon_{k}^{(2)}$ for $k=0,\cdots,N$, as given in the below,  are all positive definite. In this case, the optimal controller $\{u_{k}\}_{k=0}^{N}$ is given as:
\begin{equation}\label{th43}\begin{split}
  u_{k}&=K_kx_{k}+\bar{K}_k Ex_{k},
\end{split}\end{equation}
where
\begin{align}
 K_{k}&=-[\Upsilon_{k}^{(1)}]^{-1}M_{k}^{(1)},\label{kk}\\
 \bar{K}_{k}&=-\left\{[\Upsilon_{k}^{(2)}]^{-1}M_{k}^{(2)}-[\Upsilon_{k}^{(1)}]^{-1}M_{k}^{(1)}\right\}, \label{kkbar}
 \end{align}
and $\Upsilon_{k}^{(1)}$, $M_{k}^{(1)}$, $\Upsilon_{k}^{(2)}$, $M_{k}^{(2)}$ are given as
\begin{align}
\Upsilon_{k}^{(1)}&=R_{k}+B_{k}'P_{k+1}B_{k}+\sigma^{2}D_{k}'P_{k+1}D_{N},\label{upsi1}\\
M_{k}^{(1)}&=B_{k}'P_{k+1}A_{k}+\sigma^{2}D_{k}'P_{k+1}C_{k},\label{h1}\\
\Upsilon_{k}^{(2)}&=R_{k}+\bar{R}_{k}+(B_{k}+\bar{B}_{k})'(P_{k+1}+\bar{P}_{k+1})(B_{k}+\bar{B}_{k})\notag\\
&+\sigma^{2}(D_{k}+\bar{D}_{k})'P_{k+1}(D_{k}+\bar{D}_{k}),\label{upsi2}\\
M_{k}^{(2)}&=(B_{k}+\bar{B}_{k})'(P_{k+1}+\bar{P}_{k+1})(A_{k}\hspace{-1mm}+\hspace{-1mm}\bar{A}_{k})\notag\\
&+\sigma^{2}(D_{k}+\bar{D}_{k})'P_{k+1}(C_{k}+\bar{C}_{k}),\label{h2}
\end{align}
while $P_{k}$ and $\bar{P}_{k}$ in the above obey the following coupled Riccati equation for $k=0,\cdots,N$.
\begin{align}
P_{k}&=Q_{k}+A_{k}'P_{k+1}A_{k}\hspace{-1mm}+\hspace{-1mm}\sigma^{2}C_{k}'P_{k+1}C_{k}\notag\\
&-[M_{k}^{(1)}]'[\Upsilon_{k}^{(1)}]^{-1}M_{k}^{(1)},\label{th41}\\
   \bar{P}_{k}&=\bar{Q}_{k}+A_{k}'P_{k+1}\bar{A}_{k}+\sigma^{2}C_{k}'P_{k+1}\bar{C}_{k}\notag\\
 &+\bar{A}_{k}'P_{k+1}A_{k}+\sigma^{2}\bar{C}_{k}'P_{k+1}C_{k}\notag\\
  &+\bar{A}_{k}'P_{k+1}\bar{A}_{k}+\sigma^{2}\bar{C}_{k}'P_{k+1}\bar{C}_{k}\notag\\
  &+(A_{k}+\bar{A}_{k})'\bar{P}_{k+1}(A_{k}+\bar{A}_{k})\notag\\
  &+[M_{k}^{(1)}]'[\Upsilon_{k}^{(1)}]^{-1}M_{k}^{(1)}-[M_{k}^{(2)}]'[\Upsilon_{k}^{(2)}]^{-1}M_{k}^{(2)},\label{th42}
\end{align}
with final condition $P_{N+1}$ and $\bar{P}_{N+1}$ given by \eqref{ps2}.

The associated optimal cost function is given by
\begin{equation}\label{jnst}
  J_{N}^{*}=E(x_{0}'P_{0}x_{0})+(Ex_{0})'\bar{P}_{0}(Ex_{0}).
\end{equation}

Moreover,  the costate $\lambda_{k-1}$ in \eqref{th32} and the state $x_{k},~Ex_{k}$ admit the following relationship,
\begin{equation}\label{th4}
\lambda_{k-1}\hspace{-1mm}=\left[
  \begin{array}{cc}
    P_{k}&\bar{P}_{k}^{(1)}\\
     \bar{P}_{k}^{(2)}&\bar{P}_{k}^{(3)}     \\
  \end{array}
\right]\left[
  \begin{array}{cc}
   x_{k}\\
   Ex_{k}\\
  \end{array}
\right],
\end{equation}
where $P_k$ obeys Riccati equation \eqref{th41}, $\bar{P}_{k}^{(1)}+\bar{P}_{k}^{(2)}+\bar{P}_{k}^{(3)}=\bar{P}_k$, and $\bar{P}_k$ satisfies Riccati equation \eqref{th42}.

\end{theorem}

\begin{proof}
 See Appendix B.\end{proof}

\begin{remark}
We show that the necessary and sufficient solvability conditions for the mean-field LQ optimal control are that the matrices $\Upsilon_{k}^{(1)} $, $\Upsilon_{k}^{(2)}$ are positive definite which are parallel to the solvability condition of standard LQ control. It should be noted the solvability conditions in previous works \cite{yong} and \cite{ni1} for the mean-field LQ optimal control are given with operator type which is not easy to be verified in practice.
\end{remark}

\begin{remark}
It should be noted that the weighting matrices $R_{k}$ and $R_{k}+\bar{R}_{k}$ in cost function \eqref{ps2} are only required to be positive semi-definite in this paper which is more standard than the assumptions made in most of previous works where the matrices are required to be positive definite \cite{ni1}, \cite{yong}.
\end{remark}

\begin{remark}\label{rem3}
The presented results in Theorem \ref{main} contain the standard stochastic LQ control problem as a special case. Actually, when coefficient matrices $\bar{A}_{k}$, $\bar{B}_{k}$, $\bar{C}_{k}$, $\bar{D}_{k}$ in \eqref{ps1} and weighting matrices $\bar{Q}_{k}$, $\bar{R}_{k}$, $\bar{P}_{N+1}$ in \eqref{ps2} are zero for $0\leq k\leq N$,  by \eqref{upsi1}-\eqref{h2} and  induction method, it is easy to know that $\Upsilon_{k}^{(1)}=\Upsilon_{k}^{(2)}$, $M_{k}^{(1)}=M_{k}^{(2)}$ and thus $\bar{K}_{k}=0$. Furthermore, notice \eqref{pk1}-\eqref{pk3} and \eqref{th42}, we have  $\bar{P}_{k}^{(1)}=\bar{P}_{k}^{(2)}=\bar{P}_{k}^{(3)}=\bar{P}_{k}=0$,  \eqref{th4} becomes $\lambda_{k-1}\hspace{-1mm}=\left[
  \begin{array}{cc}
   \hspace{-1mm}P_{k}x_{k}  \hspace{-1mm}\\
   \hspace{-1mm}0\hspace{-1mm}\\
  \end{array}
\right]$. Refer to reference\cite{xyz}, \cite{det} and \cite{huang}, we know \eqref{th43}, \eqref{jnst} and \eqref{th4} are exactly the solution to standard stochastic LQ control problem.
\end{remark}


\section{Infinite Horizon Mean-field LQ Control and Stabilization}

\subsection{Problem Formulation}
In this section, the infinite horizon mean-field stochastic LQ control problem is solved. Besides, the necessary and sufficient
stabilization condition for mean-field systems is investigated.

To study the stabilization problem for infinite horizon case, we consider the following time invariant system,
\begin{equation}\label{ps10}
\left\{ \begin{array}{ll}
x_{k+1}=(Ax_{k}+\bar{A}Ex_{k}+Bu_{k}+\bar{B}Eu_{k})\\
~~~~~~~+(Cx_{k}+\bar{C}Ex_{k}+Du_{k}+\bar{D}Eu_{k})w_{k},\\
x_{0}=\xi,\\
\end{array} \right.
\end{equation}
where $A,~\bar{A},~B,~\bar{B},~C,~\bar{C},~D,~\bar{D}$ are all constant coefficient matrices with compatible dimensions. The system noise $w_{k}$ is defined as in \eqref{ps1}.

The associated cost function is given by

\begin{equation}\label{ps200}\begin{split}
  J\hspace{-1mm}=\hspace{-1mm}\sum_{k=0}^{\infty}E[x_{k}'Qx_{k}\hspace{-1mm}+\hspace{-1mm}
  (Ex_{k})'\bar{Q}Ex_{k}\hspace{-1mm}+\hspace{-1mm}u_{k}'Ru_{k}\hspace{-1mm}+\hspace{-1mm}(Eu_{k})'\bar{R}Eu_{k}],
\end{split}\end{equation}
where $Q$, $\bar{Q}$, $R$, $\bar{R}$ are deterministic symmetric weighting matrices with appropriate dimensions.

Throughout this section, the following assumption is made on the weighting matrices in \eqref{ps200}.

\begin{assumption}\label{ass2}
$R>0$, $R+\bar{R}>0$, and $Q\geq 0$, $Q+\bar{Q}\geq 0$.
\end{assumption}

\begin{remark}
It should be pointed out that Assumption \ref{ass2} is a basic condition in order to investigate the stabilization for stochastic systems, see \cite{huang}, \cite{yongj}, and so forth.
\end{remark}

The following notions of stability and stabilization are introduced.
\begin{definition}
System \eqref{ps10} with $u_{k}=0$ is called asymptotically mean square stable if for any initial values $x_{0}$, there holds
\begin{equation*}
  \lim_{k\rightarrow \infty}E(x_{k}'x_{k})=0.
\end{equation*}
\end{definition}

\begin{definition}
System \eqref{ps10} is  stabilizable in mean square sense if there exists $\mathcal{F}_{k-1}$-measurable linear controller
$u_{k}$ in terms of $x_{k}$ and $Ex_{k}$, such that for any random vector $x_{0}$, the closed loop of system \eqref{ps10} is
asymptotically mean square stable.
\end{definition}

Following from references \cite{huang},\cite{zhangw} and \cite{zhangw2}, the definitions of exactly observability and exactly detectability are respectively given in the below.
\begin{definition}\label{ob1}
Consider the following mean-field system
\begin{equation}\label{mf}
 \left\{ \begin{array}{ll}
 x_{k+1}=(Ax_{k}+\bar{A}Ex_{k})+(Cx_{k}+\bar{C}Ex_{k})w_{k},\\
Y_{k}=\mathcal{Q}^{1/2}\mathbb{X}_{k}.
\end{array} \right.
\end{equation}
where $\mathcal{Q}=\left[
  \begin{array}{cc}
   Q& 0\\
   0        & Q+\bar{Q}      \\
  \end{array}
\right]$ and $\mathbb{X}_{k}=\left[
  \begin{array}{cc}
   \hspace{-1mm} x_{k}-Ex_{k}\hspace{-1mm}\\
   \hspace{-1mm} Ex_{k}     \hspace{-1mm}           \\
  \end{array}
\right]$.

System \eqref{mf} is said to be exactly observable, if for any $N\geq 0$,
\begin{equation*}
  Y_{k}= 0, ~\forall~ 0\leq k\leq N~\Rightarrow~x_{0}=0,
\end{equation*}
where the meaning of $Y_{k}=0$ and $x_{0}=0$ are given by Definition \ref{def}. For simplicity, we rewrite system \eqref{mf} as $(A,\bar{A},C,\bar{C},\mathcal{Q}^{1/2})$.
\end{definition}

\begin{definition}\label{det1}
System $(A,\bar{A},C,\bar{C},\mathcal{Q}^{1/2})$ in \eqref{mf} is said to be exactly detectable, if for any $N\geq 0$,
\begin{equation*}
  Y_{k}= 0, ~\forall~ 0\leq k\leq N~\Rightarrow~\lim_{k\rightarrow+\infty}E(x_{k}'x_{k})=0.
\end{equation*}

\end{definition}

Now we make the following two assumptions.
\begin{assumption}\label{ass3}
$(A,\bar{A},C,\bar{C},\mathcal{Q}^{1/2})$ is exactly observable.
\end{assumption}

\begin{assumption}\label{ass4}
$(A,\bar{A},C,\bar{C},\mathcal{Q}^{1/2})$ is exactly detectable.
\end{assumption}


\begin{remark}\label{rem13}
\begin{itemize}
\item It is noted that Definition \ref{det1} gives a different definition of `exactly detectability' from the one given in previous work \cite{ni2}. In fact, \cite{ni2} considers the mean-field system with different observation $y_k$,
\begin{equation}\label{mf2}
 \left\{ \begin{array}{ll}
 x_{k+1}=(Ax_{k}+\bar{A}Ex_{k})\hspace{-1mm}+\hspace{-1mm}(Cx_{k}+\bar{C}Ex_{k})w_{k},\\
y_{k}=Qx_{k}+\bar{Q}Ex_{k}.
\end{array} \right.
\end{equation}
As sated in  \cite{ni2}, system \eqref{mf2} is `exactly detectable', if for any $N\geq 0$,
\begin{equation*}
  y_{k}= 0, ~\forall~ 0\leq k\leq N~\Rightarrow~\lim_{k\rightarrow +\infty}E(x_{k}'x_{k})=0.
\end{equation*}
Obviously, it is different from the definition given in this paper.
\item It should be highlighted that the exactly detectability made in Assumption \ref{ass4} is weaker \eqref{mf} than the exactly detectability made in \cite{ni2}. In fact, if the system is exactly detectable as made in Assumption \ref{ass4} of the paper, then we have that
 $$Y_{k}=\mathcal{Q}^{1/2}\mathbb{X}_{k}=0\Rightarrow \lim_{k\rightarrow+\infty} E(x_{k}'x_{k})=0.$$

Note that $Y_{k}=\mathcal{Q}^{1/2}\mathbb{X}_{k}=0$ implies
 \begin{equation}\label{equ}\left[
  \begin{array}{cc}
   Q& 0\\
   0        & Q+\bar{Q}      \\
  \end{array}
\right]^{1/2}\left[
  \begin{array}{cc}
   \hspace{-1mm} x_{k}-Ex_{k}\hspace{-1mm}\\
   \hspace{-1mm} Ex_{k}     \hspace{-1mm}           \\
  \end{array}
\right]=0.\end{equation}
Equation \eqref{equ} indicates that
\begin{align}\label{mf3}Q(x_{k}-Ex_{k})=0,~\text{and}~(Q+\bar{Q})Ex_{k}=0,\end{align}
and thus, $Qx_{k}+\bar{Q}Ex_{k}=0. $

Hence, if $(A,\bar{A},C,\bar{C},Q,\bar{Q})$ is `exactly detectable' as defined in \cite{ni2}, then
$(A,\bar{A},C,\bar{C},\mathcal{Q}^{1/2})$ would be exactly detectable as defined in Definition \ref{det1}.
\end{itemize}
\end{remark}

\begin{remark}\label{ob}
Definition \ref{ob1} and Definition \ref{det1} can be reduced to the standard exactly observability and exactly detectability for standard stochastic systems, respectively. Actually, with $\bar{A}=0, \bar{C}=0, \bar{Q}=0$ in system \eqref{mf}, Definition \ref{ob1} becomes $Q^{1/2}x_{k}=0\Rightarrow x_{0}=0$, which is exactly the observability definition for standard stochastic linear systems. Similarly, we can show that the exactly detectability given in Definition \ref{det1} can also be reduced to the standard exactly detectability definition for standard stochastic system. One can refer to reference \cite{abou}, \cite{huang}, \cite{zhw}, and so forth.
\end{remark}

The problems of infinite horizon LQ control and stabilization for discrete-time mean-field systems are stated as the following.

\begin{problem}\label{prob2}
Find  $\mathcal{F}_{k-1}$ measurable linear controller $u_{k}$ in terms of $x_{k}$ and $Ex_{k}$ to minimize the cost function \eqref{ps200} and stabilize system \eqref{ps10} in the mean square sense.
\end{problem}

\subsection{Solution to Problem \ref{prob2}}
For the convenience of discussion, to make the time horizon $N$ explicit for finite horizon mean-field LQ control problem, we re-denote  $\Upsilon_{k}^{(1)}$, $\Upsilon_{k}^{(2)}$, $M_{k}^{(1)}$, $M_{k}^{(2)}$ in \eqref{upsi1}-\eqref{h2} as $\Upsilon_{k}^{(1)}(N)$, $\Upsilon_{k}^{(2)}(N)$, $M_{k}^{(1)}(N)$ and $M_{k}^{(2)}(N)$ respectively. Accordingly, $K_{k}$, $\bar{K}_{k}$, $P_{k}$ and $\bar{P}_{k}$ in \eqref{kk}, \eqref{kkbar}, \eqref{th41} and \eqref{th42} are respectively rewritten as $K_{k}(N)$, $\bar{K}_{k}(N)$, $P_{k}(N)$ and $\bar{P}_{k}(N)$. Moreover, the coefficient matrices $A_{k}, \bar{A}_{k}, B_{k}, \bar{B}_{k}, C_{k}, \bar{C}_{k}, D_{k}, \bar{D}_{k}$ in \eqref{th43}-\eqref{th4} are time invariant as in \eqref{ps10}. The terminal weighting matrix $P_{N+1}$ and $\bar{P}_{N+1}$ in \eqref{ps2} are set to be zero.

Before presenting the solution to \emph{Problem \ref{prob2}}, the following lemmas will be given at first.
\begin{lemma}\label{111}
For any $N\geq 0$, $P_{k}(N)$ and $\bar{P}_{k}(N)$ in \eqref{th41}-\eqref{th42} satisfy $P_{k}(N)\geq 0$ and
$P_{k}(N)+\bar{P}_{k}(N)\geq 0$.
\end{lemma}
\begin{proof}
 See Appendix C.\end{proof}

\begin{lemma}\label{lemma3}
With the assumption $R>0$ and $R+\bar{R}>0$, \emph{Problem 1} admits a unique solution.
\end{lemma}

\begin{proof} From Lemma \ref{111}, we know that $P_{k}(N)\geq0$ and $P_{k}(N)+\bar{P}_{k}(N)\geq 0$. Besides, as $R>0$ and $R+\bar{R}>0$, from \eqref{upsi1} and \eqref{upsi2}, we know that $\Upsilon_{k}^{(1)}(N)>0$ and $\Upsilon_{k}^{(2)}(N)>0$ for $0\leq k\leq N$.  Apparently from Theorem \ref{main}, we can conclude that \emph{Problem 1} admits a unique solution for any $N>0$. This completes the proof.\end{proof}

\begin{lemma}\label{lemma2}
Under Assumptions \ref{ass2} and \ref{ass3}, for any $k\geq 0$, there exists a positive integer $N_{0}\geq 0$ such that $P_{k}(N_{0})>0$ and $P_{k}(N_{0})+\bar{P}_{k}(N_{0})>0$.
\end{lemma}
\begin{proof}
 See Appendix D.\end{proof}

\begin{theorem}\label{theorem2}
Under Assumptions \ref{ass2} and \ref{ass3}, if system \eqref{ps10} is stabilizable in the mean square sense, the following assertions hold:

1) For any $k\geq 0$, $P_{k}(N)$ and $\bar{P}_{k}(N)$ are convergent, i.e.,
$$\lim_{N\rightarrow +\infty}P_{k}(N)=P,~\lim_{N\rightarrow +\infty}\bar{P}_{k}(N)=\bar{P},$$
where $P$ and $\bar{P}$ satisfy the following coupled ARE:
\begin{align}
  P&=Q+A'PA+\sigma^{2}C'PC\hspace{-1mm}-\hspace{-1mm}[M^{(1)}]'[\Upsilon^{(1)}]^{-1}M^{(1)},\label{are1}\\
  \bar{P}&=\bar{Q}+A'P\bar{A}+\sigma^{2}C'P\bar{C}+\bar{A}'PA+\sigma^{2}\bar{C}'PC\notag\\
  &+\bar{A}'P\bar{A}+\sigma^{2}\bar{C}'P\bar{C}+(A+\bar{A})'\bar{P}(A+\bar{A})\notag\\
  &+[M^{(1)}]'[\Upsilon^{(1)}]^{-1}M^{(1)}-[M^{(2)}]'[\Upsilon^{(2)}]^{-1}M^{(2)},\label{are2}
\end{align}
while
\begin{align}
\Upsilon^{(1)}&=R+B'PB+\sigma^{2}D'PD\geq R>0,\label{up1}\\
M^{(1)}&=B'PA+\sigma^{2}D'PC,\label{hh1}\\
\Upsilon^{(2)}&=R+\bar{R}+(B+\bar{B})'(P+\bar{P})(B+\bar{B})\notag\\
&~~~~~~+\sigma^{2}(D+\bar{D})'P(D+\bar{D})\geq R+\bar{R}>0,\label{up2}\\
M^{(2)}&=(B+\bar{B})'(P+\bar{P})(A+\bar{A})\notag\\
&~~~~~~+\sigma^{2}(D+\bar{D})'P(C+\bar{C}).\label{hh2}
\end{align}

2) $P$ and $P+\bar{P}$ are positive definite.
\end{theorem}
\begin{proof}
 See Appendix E.\end{proof}

We are now in the position to present the main results of this section. Two results are to be given, one is based on the assumption of exactly observability (Assumption \ref{ass3}), and the other is based on a weaker assumption of exactly detectability (Assumption \ref{ass4}).
\begin{theorem}\label{succeed}
Under Assumption \ref{ass2} and \ref{ass3},  mean-field system \eqref{ps10} is stabilizable in the mean square sense if and only if there exists a unique solution to coupled ARE \eqref{are1}-\eqref{are2} $P$ and $\bar{P}$ satisfying $P>0$ and $P+\bar{P}>0$.

In this case, the stabilizable controller is given by
\begin{align}\label{control}
  u_{k}=Kx_{k}+\bar{K}Ex_{k},
\end{align}
where
\begin{align}
K&=-[\Upsilon^{(1)}]^{-1}M^{(1)},\label{K}\\
\bar{K}&=-\{[\Upsilon^{(2)}]^{-1}M^{(2)}-[\Upsilon^{(1)}]^{-1}M^{(1)}\},\label{KK}
\end{align}
$\Upsilon^{(1)}$, $M^{(1)}$, $\Upsilon^{(2)}$ and $M^{(2)}$ are given by \eqref{up1}-\eqref{hh2}.

Moreover, the stabilizable controller $u_{k}$ minimizes the cost function \eqref{ps200}, and the optimal cost function is given by
\begin{equation}\label{cost}
  J^{*}=E(x_{0}'Px_{0})+Ex_{0}'\bar{P}Ex_{0}.
\end{equation}
\end{theorem}
\begin{proof}
 See Appendix F.\end{proof}

\begin{theorem}\label{succeed2}
Under Assumption \ref{ass2} and \ref{ass4}, mean-field system \eqref{ps10} is stabilizable in the mean square sense if and only if there exists a unique solution to coupled ARE \eqref{are1}-\eqref{are2} $P$ and $\bar{P}$ satisfying $P\geq 0$ and $P+\bar{P}\geq 0$.

In this case, the stabilizable controller is given by \eqref{control}. Moreover, the stabilizable controller $u_{k}$ minimizes the cost function \eqref{ps200}, and the optimal cost function is as \eqref{cost}.
\end{theorem}
\begin{proof}
 See Appendix G.\end{proof}

\begin{remark}
Theorem \ref{succeed} and \ref{succeed2} propose a new approach to stochastic control problems  based on the maximum principle and solution to FBSDE developed in this paper, and thus essentially solve the optimal control and stabilization for mean-field stochastic systems under more standard assumptions which is compared with previous works \cite{ni1} and \cite{ni2}.
\end{remark}


\section{Numerical Examples}
\subsection{The Finite Horizon Case}
Consider system \eqref{ps1} and the cost function \eqref{ps2} with $N=4$ and $\sigma^{2}=1$, we choose the coefficient matrices and weighting matrices in \eqref{ps1} and \eqref{ps2} to be time-invariant for $k=1,2,3$ as:
\begin{align*}
A_{k}&\hspace{-1mm}=\hspace{-1mm}\left[\hspace{-2mm}
  \begin{array}{ccc}
   1.1\hspace{-2mm}&\hspace{-2mm} 0.9\hspace{-2mm}&\hspace{-2mm} 0.8\\
    0 \hspace{-2mm} &\hspace{-2mm}0.6 \hspace{-2mm}&\hspace{-2mm}  1.2\\
    0.4  \hspace{-2mm}&\hspace{-2mm}0.9 \hspace{-2mm}&\hspace{-2mm}  1\\
  \end{array}
\hspace{-2mm}\right],\bar{A}_{k}\hspace{-1mm}=\hspace{-1mm}\left[\hspace{-2mm}
  \begin{array}{ccc}
    0.5\hspace{-1mm}&\hspace{-1mm} 1\hspace{-1mm}&\hspace{-1mm} 0.9\\
    0.8 \hspace{-1mm}&\hspace{-1mm}0.7 \hspace{-1mm}&\hspace{-1mm}  1.2\\
     1.1  \hspace{-1mm}&\hspace{-1mm}2 \hspace{-1mm}&\hspace{-1mm}  1.9\\
  \end{array}
\hspace{-2mm}\right],B_{k}\hspace{-1mm}=\hspace{-1mm}\left[\hspace{-2mm}
  \begin{array}{ccc}
    2\hspace{-1mm}&\hspace{-1mm} 0.3\\
    1.1 \hspace{-1mm} &\hspace{-1mm}0.6    \\
    0.9  \hspace{-1mm}&\hspace{-1mm}1.3\\
  \end{array}
\hspace{-2mm}\right],\\
\bar{B}_{k}&\hspace{-1mm}=\hspace{-1mm}\left[\hspace{-2mm}
  \begin{array}{ccc}
    1.2\hspace{-2mm}&\hspace{-2mm} 0.6\\
    0.9  \hspace{-2mm}&\hspace{-2mm}1    \\
    0  \hspace{-2mm}&\hspace{-2mm}0.8\\
  \end{array}
\hspace{-2mm}\right],C_{k}\hspace{-1mm}=\hspace{-1mm}\left[\hspace{-2mm}
  \begin{array}{ccc}
   0.8\hspace{-2mm}&\hspace{-2mm}0.9\hspace{-2mm}&\hspace{-2mm} 1.5\\
   1.2  \hspace{-2mm}&\hspace{-2mm}1 \hspace{-2mm}&\hspace{-2mm}  0.8\\
    0  \hspace{-2mm}&\hspace{-2mm}0.6\hspace{-2mm}&\hspace{-2mm}  0.4\\
  \end{array}
\hspace{-2mm}\right],\bar{C}_{k}\hspace{-1mm}=\hspace{-1mm}\left[\hspace{-2mm}
  \begin{array}{ccc}
   1\hspace{-2mm}&\hspace{-2mm}0\hspace{-2mm}&\hspace{-2mm}0.3\\
    0.5  \hspace{-2mm}&\hspace{-2mm}0.6 \hspace{-2mm}&\hspace{-2mm}  0.9\\
     0.7  \hspace{-2mm}&\hspace{-2mm}1.2 \hspace{-2mm}&\hspace{-2mm} 0.8\\
  \end{array}
\hspace{-1mm}\right],\\
D_{k}&=\left[\hspace{-2mm}
  \begin{array}{ccc}
    0.5\hspace{-2mm}&\hspace{-2mm}0.4\\
    2  \hspace{-2mm}&\hspace{-2mm}0.9  \\
    1\hspace{-2mm}&\hspace{-2mm}0\\
  \end{array}
\right],\bar{D}_{k}=\left[\hspace{-2mm}
  \begin{array}{ccc}
    2\hspace{-2mm}&\hspace{-2mm} 1\\
   0.5 \hspace{-2mm}&\hspace{-2mm}0.8    \\
   0 \hspace{-2mm}&\hspace{-2mm}0.5\\
  \end{array}
\right],\\
Q_{k}&=diag([0,~2,~1]),\bar{Q}_{k}=diag([1,~-1,~0]),\\
R_{k}&=diag([0,~2]),\bar{R}_{k}=diag([1,~-2]),\\
P_{4}&=diag([1,~2,~0]),\bar{P}_{4}=diag([1,~-1,~1]).
\end{align*}
It is noted that $R_{k}$ and $R_{k}+\bar{R}_{k}$ are semi-positive definite, while not positive definite  for $k=1,2,3$.

Based on \eqref{th43}-\eqref{th4} of Theorem \ref{main}, the solution to coupled Riccati equation \eqref{th41}-\eqref{th42} can be given as:
\begin{align*}
P_{3}&\hspace{-1mm}=\hspace{-1mm}\left[\hspace{-2mm}
  \begin{array}{ccc}
    0.995\hspace{-2mm}&\hspace{-2mm} 0.298\hspace{-2mm}&\hspace{-2mm} -0.115\\
   0.298\hspace{-2mm}&\hspace{-2mm}2.417\hspace{-2mm}& \hspace{-2mm}0.840\\
   -0.115\hspace{-2mm}&\hspace{-2mm}0.840\hspace{-2mm}&\hspace{-2mm}  3.360\\
  \end{array}
\hspace{-2mm}\right]\hspace{-1mm},\bar{P}_{3}\hspace{-1mm}=\hspace{-1mm}\left[\hspace{-2mm}
  \begin{array}{ccc}
    0.667\hspace{-2mm}& \hspace{-2mm}0.074\hspace{-2mm}&\hspace{-2mm} -0.006\hspace{-2mm}\\
    0.074 \hspace{-2mm}&\hspace{-2mm}1.033\hspace{-2mm} &\hspace{-2mm}  0.133\hspace{-2mm}\\
    -0.006 \hspace{-2mm} &\hspace{-2mm}0.133\hspace{-2mm} &\hspace{-2mm} -1.319\hspace{-2mm}\\
  \end{array}
\hspace{-2mm}\right]\hspace{-1mm},\\
P_{2}&\hspace{-1mm}=\hspace{-1mm}\left[\hspace{-2mm}
  \begin{array}{ccc}
   1.658\hspace{-2mm}&\hspace{-2mm} 0.161\hspace{-2mm}&\hspace{-2mm} 0.024\\
   0.161 \hspace{-2mm}&\hspace{-2mm}2.547 \hspace{-2mm}&\hspace{-2mm}  0.839\\
   0.024 \hspace{-2mm}&\hspace{-2mm}0.839 \hspace{-2mm}&\hspace{-2mm}  3.379\\
  \end{array}
\hspace{-2mm}\right],\bar{P}_{2}\hspace{-1mm}=\hspace{-1mm}\left[\hspace{-2mm}
  \begin{array}{ccc}
   0.630\hspace{-2mm}&\hspace{-2mm} 0.919\hspace{-2mm}&\hspace{-2mm} -0.457\\
   0.919 \hspace{-2mm}&\hspace{-2mm}5.282\hspace{-2mm} &\hspace{-2mm} 1.439\\
    -0.457 \hspace{-2mm} &\hspace{-2mm}1.439\hspace{-2mm} &\hspace{-2mm}  -0.520\\
  \end{array}
\hspace{-2mm}\right],\\
P_{1}&\hspace{-1mm}=\hspace{-1mm}\left[\hspace{-2mm}
  \begin{array}{ccc}
   1.907\hspace{-2mm}&\hspace{-2mm} 0.315\hspace{-2mm}&\hspace{-2mm}0.268 \\
   0.315\hspace{-2mm}&\hspace{-2mm}2.812\hspace{-2mm}&\hspace{-2mm}1.352\\
   0.268\hspace{-2mm}&\hspace{-2mm}1.352\hspace{-2mm}&\hspace{-2mm} 4.408\\
  \end{array}
\hspace{-2mm}\right],\bar{P}_{1}\hspace{-1mm}=\hspace{-1mm}\left[\hspace{-2mm}
  \begin{array}{ccc}
   0.982\hspace{-2mm}&\hspace{-2mm}0.924\hspace{-2mm}&\hspace{-2mm} -1.027\\
    0.924 \hspace{-2mm}&\hspace{-2mm}5.306\hspace{-2mm} &\hspace{-2mm}  0.711\\
   -1.027 \hspace{-2mm} &\hspace{-2mm}0.711\hspace{-2mm} & \hspace{-2mm} -1.327\\
  \end{array}
\hspace{-2mm}\right],\\
P_{0}&\hspace{-1mm}=\hspace{-1mm}\left[\hspace{-2mm}
  \begin{array}{ccc}
   2.026\hspace{-2mm}&\hspace{-2mm}0.353\hspace{-2mm}&\hspace{-2mm}0.364\\
   0.353\hspace{-2mm}&\hspace{-2mm}2.896\hspace{-2mm}&\hspace{-2mm}1.472\\
   0.364\hspace{-2mm}&\hspace{-2mm}1.472\hspace{-2mm}&\hspace{-2mm} 4.641\\
  \end{array}
\hspace{-2mm}\right],\bar{P}_{0}\hspace{-1mm}=\hspace{-1mm}\left[\hspace{-2mm}
  \begin{array}{ccc}
    1.217\hspace{-2mm}&\hspace{-2mm} 1.294\hspace{-2mm}&\hspace{-2mm}-1.198\\
   1.294 \hspace{-2mm}&\hspace{-2mm}6.232\hspace{-2mm} &\hspace{-2mm} 0.644\\
    -1.198 \hspace{-2mm} &\hspace{-2mm}0.644\hspace{-2mm} & \hspace{-2mm} -1.498\\
  \end{array}
\hspace{-2mm}\right],
\end{align*}
and $\Upsilon_{k}^{(1)}$, $\Upsilon_{k}^{(2)}$ for $k=0,1,2,3$ in \eqref{upsi1} and \eqref{upsi2} can be calculated as
\begin{align*}
\Upsilon_{3}^{(1)}&\hspace{-1mm}=\hspace{-1mm}\left[\hspace{-1mm}
  \begin{array}{cc}
    14.670& 5.720\\
    5.720 &4.590 \\
  \end{array}
\hspace{-1mm}\right],\Upsilon_{3}^{(2)}\hspace{-1mm}=\hspace{-1mm}\left[
  \begin{array}{cc}
    45.040& 22.850\\
     22.850 &16.330\\
  \end{array}
\hspace{-1mm}\right],\\
\Upsilon_{2}^{(1)}&\hspace{-1mm}=\hspace{-1mm}\left[\hspace{-1mm}
  \begin{array}{cc}
    29.302& 13.536\\
   13.536  &12.297\\
  \end{array}
\hspace{-1mm}\right],\Upsilon_{2}^{(2)}\hspace{-1mm}=\hspace{-1mm}\left[\hspace{-1mm}
  \begin{array}{cc}
   73.069&46.750\\
    46.750 &38.789\\
  \end{array}
\hspace{-1mm}\right],\\
\Upsilon_{1}^{(1)}&\hspace{-1mm}=\hspace{-1mm}\left[\hspace{-1mm}
  \begin{array}{cc}
    32.593& 14.480\\
    14.480&12.607\\
  \end{array}
\hspace{-1mm}\right],\Upsilon_{1}^{(2)}\hspace{-1mm}=\hspace{-1mm}\left[\hspace{-1mm}
  \begin{array}{cc}
   113.585& 76.217\\
    76.217 &64.973\\
  \end{array}
\hspace{-1mm}\right],\\
\Upsilon_{0}^{(1)}&\hspace{-1mm}=\hspace{-1mm}\left[\hspace{-1mm}
  \begin{array}{ccc}
  42.070&19.232\\
   19.232 &15.875 \\
  \end{array}
\hspace{-1mm}\right],\Upsilon_{0}^{(2)}\hspace{-1mm}=\hspace{-1mm}\left[\hspace{-1mm}
  \begin{array}{ccc}
    130.398& 82.580\\
    82.580 &68.411 \\
  \end{array}
\hspace{-1mm}\right],\\
\det[\Upsilon_{3}^{(1)}]&=34.617>0,\det[\Upsilon_{3}^{(2)}]=213.381>0,\\
 \det[\Upsilon_{2}^{(1)}]&=117.112>0 ,\det[\Upsilon_{2}^{(2)}]=648.698>0,\\
 \det[\Upsilon_{1}^{(1)}]&=201.228>0,  \det[\Upsilon_{1}^{(2)}]=1570.987>0 ,\\
\det[\Upsilon_{0}^{(1)}]&=297.946>0, \det[\Upsilon_{0}^{(2)}]=2101.236>0.
\end{align*}

Since $\Upsilon_{k}^{(1)}>0$ and $\Upsilon_{k}^{(2)}>0$, thus by Theorem \ref{main}, the unique optimal controller can be given as:
\begin{equation*}
  u_{k}=K_{k}x_{k}+\bar{K}_{k}Ex_{k}, k=0,1,2,3,
\end{equation*}
where
\begin{align*}
K_{3}&=\left[\hspace{-1mm}
  \begin{array}{ccc}
   -0.517& -0.483& -0.471 \\
  0.032 &-0.084 & -0.223\\
  \end{array}
\right],\\
\bar{K}_{3}&=\left[\hspace{-1mm}
  \begin{array}{ccc}
      0.184& 0.328&0.357 \\
   -0.522 &-0.819 & -0.920\\
  \end{array}
\right]\hspace{-1mm},\\
K_{2}&=\left[\hspace{-1mm}
  \begin{array}{ccc}
    -0.385& -0.481& -0.410 \\
   -0.030 &-0.247& -0.474\\
  \end{array}
\right],\\
\bar{K}_{2}&=\left[\hspace{-1mm}
  \begin{array}{ccc}
      0.036& 0.327& 0.336 \\
   -0.364 &-0.734& -0.828\\
  \end{array}
\right],\\
K_{1}&=\left[\hspace{-1mm}
  \begin{array}{ccc}
    -0.413& -0.476& -0.394 \\
    -0.011&-0.256&-0.500\\
  \end{array}
\right],\\
\bar{K}_{1}&=\left[\hspace{-1mm}
  \begin{array}{ccc}
     0.071& 0.345& 0.305 \\
    -0.334 &-0.699 &-0.820\\
  \end{array}
\right],\\
K_{0}&=\left[
  \begin{array}{ccc}
   -0.411& -0.487& -0.398 \\
   0.001&-0.259&-0.525\\
  \end{array}
\right],\\
\bar{K}_{0}&=\left[
  \begin{array}{ccc}
    0.070& 0.339&0.297 \\
   -0.358&-0.692&-0.780\\
  \end{array}
\right].
\end{align*}

\subsection{The Infinite Horizon Case}
Consider system \eqref{ps10} and the cost function \eqref{ps200} with the following coefficient matrices and weighting matrices:
\begin{align*}
A&=1.1,\bar{A}=0.2,B=0.4,\bar{B}=0.1,C=0.9,\bar{C}=0.5,\\
D&=0.8,\bar{D}=0.2,Q=2,\bar{Q}=1,R=1,\bar{R}=1,\sigma^{2}=1.
\end{align*}
the initial state $x_{0}\sim N(1,2)$, i.e., $x_{0}$ obeys the normal distribution with mean 1 and covariance 2.

Note that $Q=2$, $Q+\bar{Q}=3$, $R=1$, $R+\bar{R}=2$ are all positive, then Assumption 3 and Assumption 4 are satisfied. By using coupled ARE \eqref{are1}-\eqref{are2}, we have $P=5.6191$ and $\bar{P}=5.1652$. From \eqref{up1}-\eqref{hh2}, we can obtain $ \Upsilon^{(1)}=5.4953,M^{(1)}=6.5182, \Upsilon^{(2)}=10.3152$, and $M^{(2)}=14.8765$.

Notice that $P>0$ and $P+\bar{P}>0$, according to Theorem \ref{succeed}, there exists a unique optimal controller to stabilize mean-field system \eqref{ps10} as well as minimize cost function \eqref{ps200}, the controller in \eqref{control} is presented as
\begin{align*}
u_{k}=Kx_{k}+\bar{K}Ex_{k}=-1.1861x_{k}-0.2561Ex_{k},~k\geq 0.
\end{align*}

Using the designed controller, the simulation of system state is shown in Fig. 1. With the optimal controller, the regulated system state is stabilizable in mean square sense as shown in Fig. 1.
\begin{figure}[htbp]
  \centering
  \includegraphics[width=0.45\textwidth]{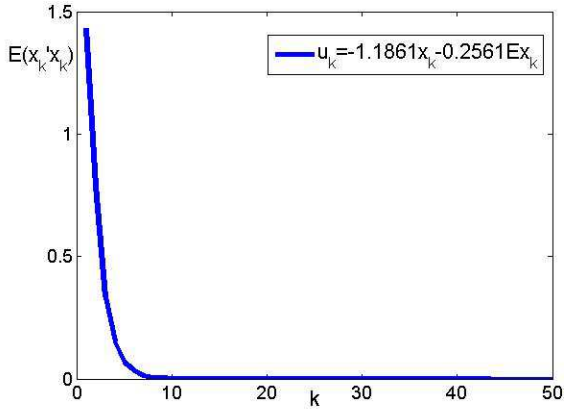}
  \caption{The mean square stabilization of mean-field system.}\label{fig:1}
\end{figure}

To explore the effectiveness of the main results presented in this paper, we consider mean-field system \eqref{ps10} and cost function \eqref{ps200} with
\begin{align*}
A&=2,\bar{A}=0.8,B=0.5,\bar{B}=1,C=1,\bar{C}=1,\\
D&=-0.8,\bar{D}=0.6,Q=1,\bar{Q}=1,R=1,\bar{R}=1,\sigma^{2}=1.
\end{align*}
The initial state are assumed to be the same as that given above.

By solving the coupled ARE \eqref{are1}, it can be found that $P$ has two negative roots as $P=-1.1400$ and $P=-0.2492$. Thus, according to Theorem \ref{succeed} and Theorem \ref{succeed2}, we know that  system \eqref{ps10} is not stabilizable in mean square sense.

Actually, when $P=-1.1400$, it is easily known that equation \eqref{are2} has no real roots for $\bar{P}$. While in the case of $P=-0.2492$, $\bar{P}$ has two real roots  which can be solved from \eqref{are2} as $\bar{P}=7.0597$ and $\bar{P}=-0.6476$, respectively.

In the latter case, with $P=-0.2492$ and $\bar{P}=7.0597$, we can calculate $K$ and $\bar{K}$ from \eqref{K} and \eqref{KK} as $K=0.0640$,  $\bar{K}=1.5939$. Similarly, with $P=-0.2492$ and $\bar{P}=-0.6476$, $K$ and $\bar{K}$ can be computed as  $K=0.0640$,  $\bar{K}=131.8389$. Accordingly, the controllers are designed as $u_{k}=0.0640x_{k}+1.5939Ex_{k}$, $u_{k}=0.0640x_{k}+131.8389Ex_{k}$, respectively.

Simulation results of the corresponding state trajectories with the designed controllers are respectively shown as in Fig. 2 and Fig. 3.  As expected, the state trajectories are not convergent.

\begin{figure}[htbp]
  \centering
  \includegraphics[width=0.45\textwidth]{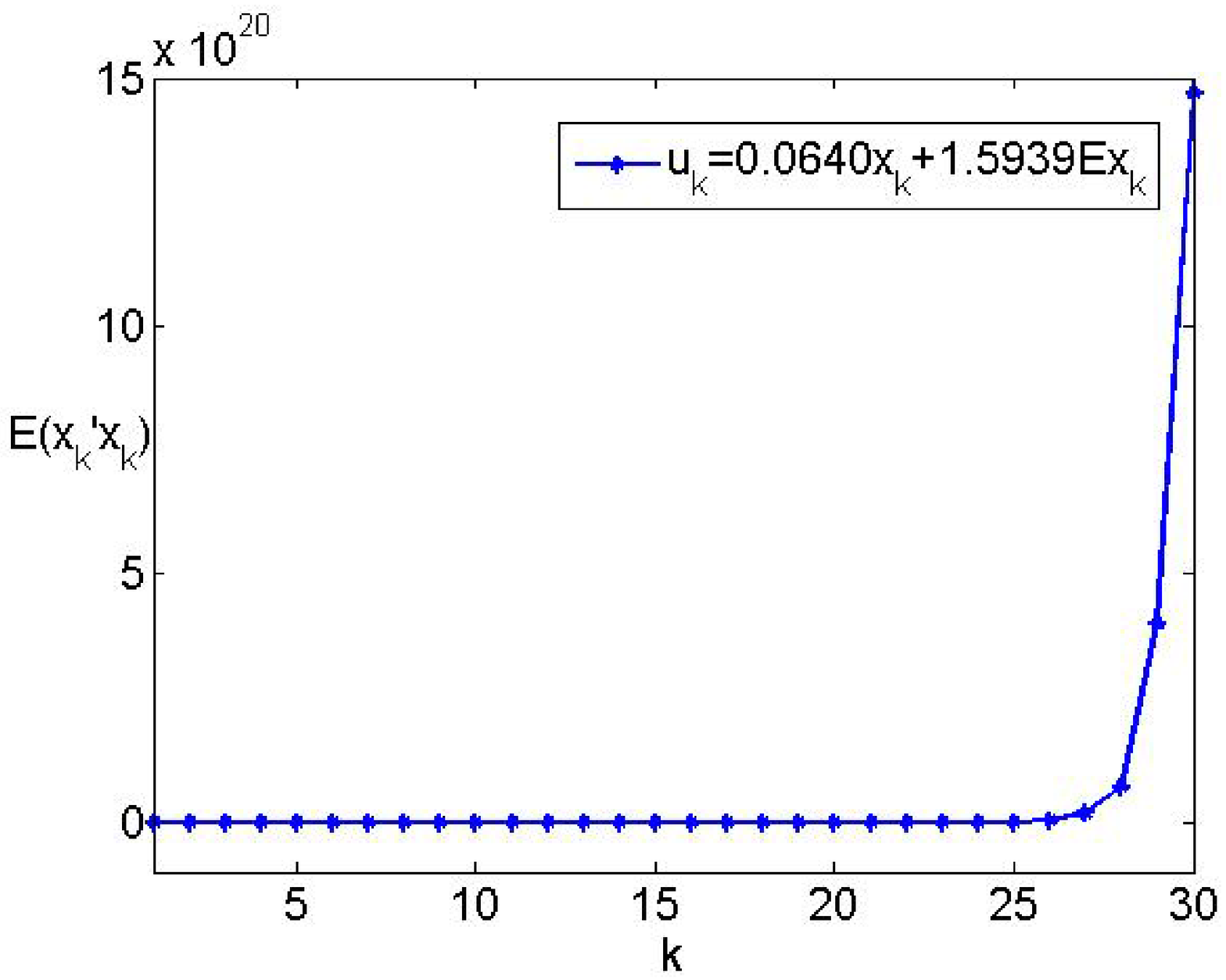}
  \caption{Simulation for the state trajectory $E(x_{k}'x_{k})$.}\label{fig:2}
\end{figure}

\begin{figure}[htbp]
  \centering
  \includegraphics[width=0.45\textwidth]{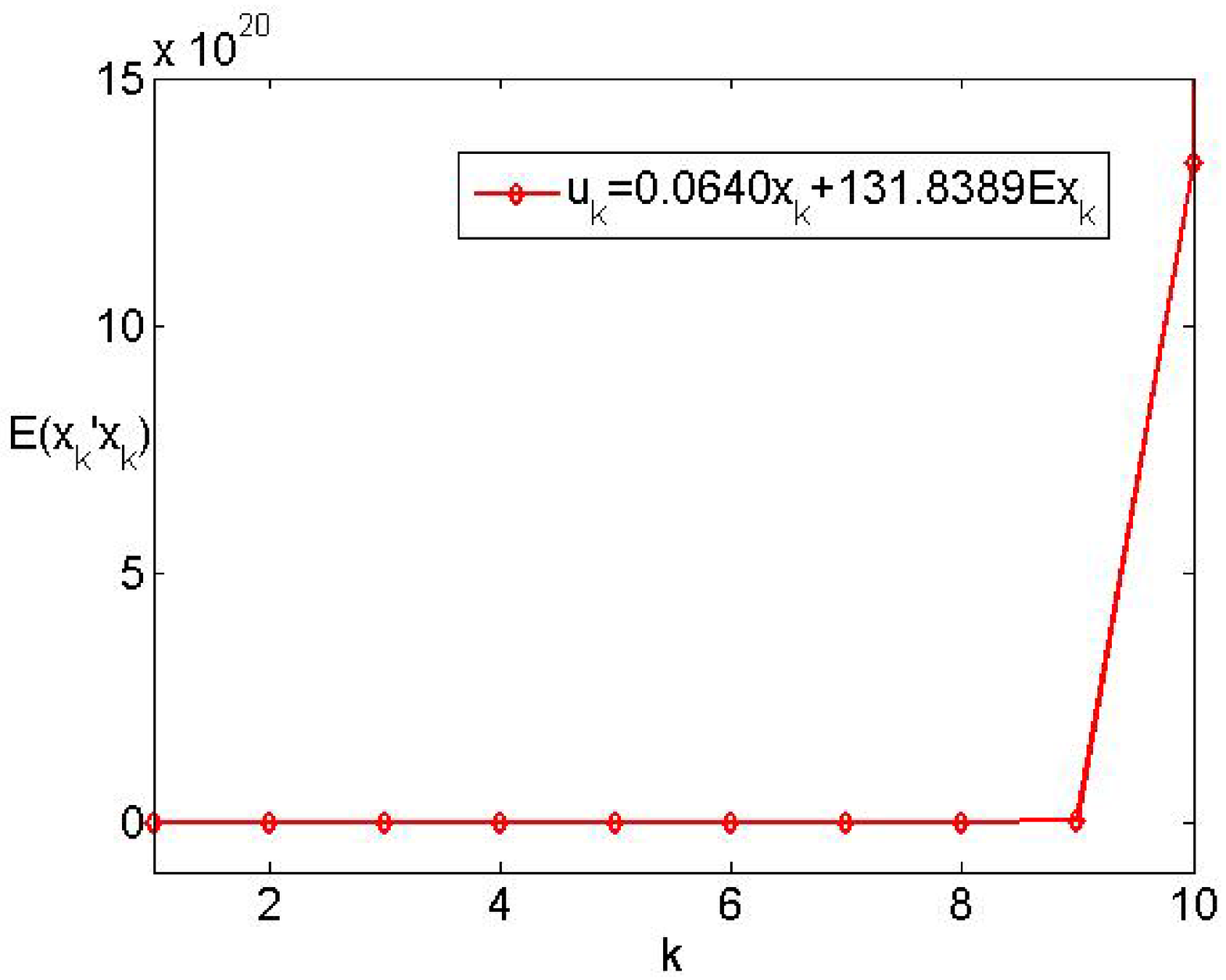}
  \caption{Simulation for the state trajectory $E(x_{k}'x_{k})$.}\label{fig:3}
\end{figure}

\section{Conclusion}
This paper proposes a new approach to stochastic optimal control with the key tools of maximum principle and solution to FBSDE explored in this paper. Accordingly, with the approach, the optimal control and stabilization problems for discrete-time mean-field systems have been essentially solved. The main results include:  1) The sufficient and necessary solvability condition of finite horizon optimal control problem has been obtained in analytical form via a coupled Riccati equation; 2) The sufficient and necessary conditions for the stabilization of mean-field systems has been obtained. It is shown that, under exactly observability assumption, the mean-field system is stabilizable in the mean square sense if and only if a coupled ARE has a unique solution $P$ and $\bar{P}$ satisfying $P>0$ and $P+\bar{P}>0$. Furthermore, under exactly detectability assumption which is weaker than exactly observability, we show that the mean-field system is stabilizable in the mean square sense if and only if the coupled ARE admits a unique solution $P$ and $\bar{P}$ satisfying $P\geq 0$ and $P+\bar{P}\geq 0$.

\appendices

\section{Proof of Theorem \ref{maximum}}

\begin{proof}
For the general stochastic mean-field optimal control problem, the control domain for system \eqref{ps3} to minimize \eqref{ps04} is given by
\begin{equation*}
  \mathcal{U}=\left\{u_{k}\in\mathcal{R}^{m}|~u_{k}~\text{is}~\mathcal{F}_{k-1}~\text{measurable},~E|u_{k}|^{2}<\infty  \right\}.
\end{equation*}

We assume that the control domain $\mathcal{U}$ to be convex. Any $u_{k}\in \mathcal{U}$ is called admissible control. Besides, for arbitrary $u_{k},~\delta u_{k}\in \mathcal{U}$ and $\varepsilon\in(0,1)$, we can obtain $u_{k}^{\varepsilon}=u_{k}+\varepsilon \delta u_{k}\in \mathcal{U}$.

Let $x_{k}^{\varepsilon}$, $J_{N}^{\varepsilon}$ be the corresponding state and cost function with $u_{k}^{\varepsilon}$, and $x_{k}$, $J_{N}$ represent the corresponding state and cost function with $u_{k}$.

We examine the increment in $J_N$ due to increment in the controller $u_k$. Assume that final time
$N+1$ is fixed, by using Taylor's expansion and following cost function \eqref{ps04}, the increment $\delta
J_{N}=J_{N}^{\varepsilon}-J_{N}$ can be calculated as follows,
\begin{align}\label{mp0001}
&\delta J_N=E\Big\{\phi_{x_{N+1}} \delta x_{N+1}+\phi_{Ex_{N+1}} \delta Ex_{N+1}\notag\\
&\hspace{-1mm}+\hspace{-1mm}\sum_{k=0}^{N}\hspace{-1mm}\big[L^k_{x_k}\delta x_{k}\hspace{-1mm}+\hspace{-1mm}L^k_{Ex_k}\delta Ex_{k}\hspace{-1mm}+\hspace{-1mm}L^k_{u_k}\varepsilon\delta u_{k}\hspace{-1mm}+\hspace{-1mm}L^k_{Eu_k}\varepsilon\delta Eu_{k}\big]\Big\}\hspace{-1mm}+\hspace{-1mm}O(\varepsilon^{2})\notag\\
&\hspace{-1mm}=\hspace{-1mm}E\big\{[\phi_{x_{N+1}}\hspace{-1mm}+\hspace{-1mm}E(\phi_{Ex_{N+1}})]\delta x_{N+1}\hspace{-1mm}+\hspace{-1mm}\sum_{k=0}^{N}[L^k_{u_k}\hspace{-1mm}+\hspace{-1mm}E(L_{Eu_{k}}^{k})]\varepsilon \delta u_{k}\notag\\
&+\sum_{k=0}^{N}[L^k_{x_k}+E(L^k_{Ex_k})] \delta x_{k}\big\}+O(\varepsilon^{2}).
\end{align}
where $O(\varepsilon^{2})$ means infinitesimal of the same order with $\varepsilon^{2}$.

Another thing to note is the variation of the initial state $\delta x_{0}=\delta Ex_{0}=0$.

By \eqref{ps1} and \eqref{mp0002}, for $\delta x_{k}=x_{k}^{\varepsilon}-x_{k}$, the following assertion holds,
\begin{align}\label{asser}
  \left[\hspace{-2mm}
  \begin{array}{cc}
   \delta x_{k+1} \\
    \delta Ex_{k+1}
  \end{array}
\hspace{-3mm}\right]&\hspace{-1mm}=\hspace{-1mm}\left[\hspace{-2mm}
  \begin{array}{cc}
    f_{x_{k}}^{k}\hspace{-2mm} &\hspace{-2mm} f_{Ex_{k}}^{k} \\
    g_{x_{k}}^{k}               \hspace{-2mm} &\hspace{-2mm} g_{Ex_{k}}^{k}\\
  \end{array}
\hspace{-2mm}\right]\hspace{-1mm}\left[\hspace{-2mm}
  \begin{array}{cc}
    \delta x_{k} \\
    \delta Ex_{k}
  \end{array}
\hspace{-2mm}\right]\hspace{-1mm}+\hspace{-1mm}\left[\hspace{-2mm}
  \begin{array}{cc}
    f_{u_{k}}^{k} \hspace{-2mm}&\hspace{-2mm} f_{Eu_{k}}^{k} \\
    g_{u_{k}}^{k}\hspace{-2mm}&\hspace{-2mm} g_{Eu_{k}}^{k}\\
  \end{array}
\hspace{-2mm}\right]\hspace{-1mm}\left[\hspace{-2mm}
  \begin{array}{cc}
    \varepsilon\delta u_{k} \\
    \varepsilon\delta Eu_{k}
  \end{array}
\hspace{-2mm}\right],
\end{align}

Thus the variation of $\delta x_{k+1}$ can be presented as
\begin{align}\label{mp0005}
&\delta x_{k+1} =f_{x_{k}}^{k}\delta x_{k}+f_{u_{k}}^{k}\varepsilon\delta u_{k}+f_{Ex_{k}}^{k}\delta Ex_{k}+f_{Eu_{k}}^{k}\varepsilon\delta Eu_{k}\notag\\
&=\left[\hspace{-2mm}
  \begin{array}{cc}
    f_{x_{k}}^{k} \hspace{-2mm}&\hspace{-2mm} \hspace{-1mm}f_{Ex_{k}}^{k} \\
  \end{array}
\hspace{-2mm}\right] \hspace{-1mm}\left[\hspace{-2mm}
  \begin{array}{cc}
   \delta x_{k} \\
   \delta Ex_{k}
  \end{array}
\hspace{-2mm}\right] +f_{u_{k}}^{k}\varepsilon\delta u_{k}+f_{Eu_{k}}^{k}\varepsilon\delta Eu_{k}\notag \\
&=\left[\hspace{-2mm}
  \begin{array}{cc}
   f_{x_{k}}^{k} \hspace{-2mm}&\hspace{-2mm} f_{Ex_{k}}^{k}\\
  \end{array}
\hspace{-2mm}\right]\hspace{-1mm}\left[\hspace{-2mm}
  \begin{array}{cc}
   f_{x_{k-1}}^{k-1}\hspace{-2mm} &\hspace{-2mm} f_{Ex_{k-1}}^{k-1} \\
    g_{x_{k-1}}^{k-1} \hspace{-2mm} &\hspace{-2mm} g_{Ex_{k-1}}^{k-1}\\
  \end{array}
\hspace{-2mm}\right]\hspace{-1mm}\left[\hspace{-2mm}
  \begin{array}{cc}
    \delta x_{k-1} \\
    \delta Ex_{k-1}
  \end{array}
\hspace{-2mm}\right]\notag\\
&+\left[\hspace{-2mm}
  \begin{array}{cc}
   f_{x_{k}}^{k} \hspace{-3mm}& \hspace{-3mm}f_{Ex_{k}}^{k} \\
  \end{array}
\hspace{-2mm}\right]\hspace{-1mm}\left[\hspace{-2mm}
  \begin{array}{cc}
    f_{u_{k-1}}^{k-1}\hspace{-3mm} &\hspace{-3mm} f_{Eu_{k-1}}^{k-1} \\
    g_{u_{k-1}}^{k-1}     \hspace{-3mm}&\hspace{-3mm} g_{Eu_{k-1}}^{k-1}\\
  \end{array}
\hspace{-2mm}\right]\hspace{-1mm}\left[\hspace{-2mm}
  \begin{array}{cc}
    \varepsilon\delta u_{k-1}\hspace{-1mm} \\
    \varepsilon\delta Eu_{k-1}\hspace{-1mm}
  \end{array}
\hspace{-2mm}\right]\hspace{-1mm}+\hspace{-1mm}f_{u_{k}}^{k}\delta u_{k}
\hspace{-1mm}+\hspace{-1mm}f_{Eu_{k}}^{k}\delta Eu_{k}\notag \\
&=\tilde{F}_{x}(k,0)\hspace{-1mm}\left[\hspace{-2mm}
  \begin{array}{cc}
   \delta x_{0} \\
    \delta Ex_{0}
  \end{array}
\hspace{-2mm}\right]\hspace{-1mm}+\hspace{-1mm}\sum_{l=0}^{k}\tilde{F}_{x}(k,l+1)\hspace{-1mm}\left[\hspace{-2mm}
  \begin{array}{cc}
    f_{u_{l}}^{l}& f_{Eu_{l}}^{l} \\
    g_{u_{l}}^{l} & g_{Eu_{l}}^{l}\\
  \end{array}
\hspace{-2mm}\right]\hspace{-1mm}\left[\hspace{-2mm}
  \begin{array}{cc}
    \varepsilon\delta u_{l} \\
    \varepsilon\delta Eu_{l}
  \end{array}
\hspace{-2mm}\right]\notag\\
&\hspace{-1mm}=\hspace{-1mm}\sum_{l=0}^{k}\hspace{-1mm}\tilde{F}_{x}(k,l\hspace{-1mm}+\hspace{-1mm}1)\hspace{-1mm}\left[\hspace{-2mm}
  \begin{array}{cc}
    f_{u_{l}}^{l}\\
  g_{u_{l}}^{l} \\
  \end{array}
\hspace{-2mm}\right]\hspace{-1mm}\varepsilon\delta u_{l}\hspace{-1mm}+\hspace{-1mm}\sum_{l=0}^{k}\hspace{-1mm}\tilde{F}_{x}(k,l\hspace{-1mm}+\hspace{-1mm}1)\left[\hspace{-2mm}
  \begin{array}{cc}
    f_{Eu_{l}}^{l} \\
   g_{Eu_{l}}^{l}\\
  \end{array}
\hspace{-2mm}\right]\hspace{-1mm}\varepsilon\delta Eu_{l},
\end{align}
where
\begin{equation}\label{fx}
 \tilde{F}_x(k,l)=\left[\hspace{-2mm}
  \begin{array}{cc}
    f_{x_{k}}^{k} \hspace{-2mm}&\hspace{-2mm} f_{Ex_{k}}^{k}  \\
  \end{array}
\hspace{-2mm}\right]\tilde{f}^{k-1}_{x_{k-1}} \cdots \tilde{f}^l_{x_l},l=0,\cdots,k;
\end{equation}
$\tilde{F}_x(k,k+1)=[I_{n}~0]$,
and
$\tilde{f}^{l}_{x_{l}}=\left[\hspace{-2mm}
  \begin{array}{cc}
   f_{x_{l}}^{l} \hspace{-2mm}&\hspace{-2mm} f_{Ex_{l}}^{l} \\
    g_{x_{l}}^{l} \hspace{-2mm}&\hspace{-2mm} g_{Ex_{l}}^{l}\\
  \end{array}
\hspace{-2mm}\right].$

Substituting \eqref{mp0005} into \eqref{mp0001} yields
\begin{align}\label{mp02}
  &\delta J_N\hspace{-1mm}=\hspace{-1mm}E\Big\{[\phi_{x_{N+1}}\hspace{-1mm}+\hspace{-1mm}E(\phi_{Ex_{N+1}})]\sum_{l=0}^{N}\tilde{F}_{x}(N,l\hspace{-1mm}+\hspace{-1mm}1)\hspace{-1mm}\left[\hspace{-2mm}
  \begin{array}{cc}
    f_{u_{l}}^{l}\\
    g_{u_{l}}^{l}\\
  \end{array}
\hspace{-2mm}\right]\hspace{-1mm}\varepsilon\delta u_{l}\notag\\
&+[\phi_{x_{N+1}}\hspace{-1mm}+\hspace{-1mm}E(\phi_{Ex_{N+1}})]\sum_{l=0}^{N}\tilde{F}_{x}(N,l\hspace{-1mm}+\hspace{-1mm}1)\hspace{-1mm}\left[\hspace{-2mm}
  \begin{array}{cc}
    f_{Eu_{l}}^{l} \\
   g_{Eu_{l}}^{l}\\
  \end{array}
\hspace{-2mm}\right]\hspace{-1mm}\varepsilon\delta Eu_{l}\notag\\
&+\sum_{k=0}^{N}[L^k_{x_k}\hspace{-1mm}+\hspace{-1mm}E(L^k_{Ex_k})]\sum_{l=0}^{k-1}\tilde{F}_{x}(k\hspace{-1mm}-\hspace{-1mm}1,l\hspace{-1mm}+\hspace{-1mm}1)\hspace{-1mm}\left[\hspace{-2mm}
  \begin{array}{cc}
   f_{u_{l}}^{l}\\
   g_{u_{l}}^{l}\\
  \end{array}
\hspace{-2mm}\right]\hspace{-1mm}\varepsilon\delta u_{l}\notag\\
&+\sum_{k=0}^{N}[L^k_{x_k}\hspace{-1mm}+\hspace{-1mm}E(L^k_{Ex_k})]\sum_{l=0}^{k-1}\tilde{F}_{x}(k\hspace{-1mm}-\hspace{-1mm}1,l\hspace{-1mm}+\hspace{-1mm}1)\hspace{-1mm}\left[\hspace{-1mm}
  \begin{array}{cc}
   f_{Eu_{l}}^{l}\\
  g_{Eu_{l}}^{l}\\
  \end{array}
\hspace{-2mm}\right]\hspace{-1mm}\varepsilon\delta Eu_{l}\notag\\
&+\sum_{k=0}^{N}[L^k_{u_k}+E(L_{Eu_{k}}^{k})] \varepsilon\delta u_{k}\Big\}+O(\varepsilon^{2}).
\end{align}

Note the facts that
\begin{align}
 &
 ~~E\Big\{[L^k_{x_k}\hspace{-1mm}+\hspace{-1mm}E(L^k_{Ex_k})]\sum_{l=0}^{k-1}\tilde{F}_{x}(k\hspace{-1mm}-\hspace{-1mm}1,l\hspace{-1mm}+\hspace{-1mm}1)\hspace{-1mm}\left[\hspace{-2mm}
  \begin{array}{cc}
    f_{Eu_{l}}^{l}\\
    g_{Eu_{l}}^{l}\\
  \end{array}
\hspace{-2mm}\right]\hspace{-1mm}\varepsilon\delta Eu_{l}
  \Big\}\notag\\
  &\hspace{-1mm}=\hspace{-1mm}
  E\hspace{-1mm}\Big\{E\Big\{[L^k_{x_k}\hspace{-1mm}+\hspace{-1mm}E(L^k_{Ex_k})]\sum_{l=0}^{k-1}\tilde{F}_{x}(k\hspace{-1mm}-\hspace{-1mm}1,l\hspace{-1mm}+\hspace{-1mm}1)\hspace{-1mm}\left[\hspace{-2mm}
  \begin{array}{cc}
    f_{Eu_{l}}^{l}\\
    g_{Eu_{l}}^{l}\\
  \end{array}
\hspace{-2mm}\right]\hspace{-1mm}\Big\}\varepsilon\delta u_{l}
 \hspace{-1mm} \Big\},\label{exp}\\
 & ~~ E\hspace{-1mm}\Big\{[\phi_{x_{N+1}}\hspace{-1mm}+\hspace{-1mm}E(\phi_{Ex_{N+1}})]\sum_{l=0}^{N}\tilde{F}_{x}(N,l\hspace{-1mm}+\hspace{-1mm}1)\hspace{-1mm}\left[\hspace{-2mm}
  \begin{array}{cc}
    f_{Eu_{l}}^{l}\\
    g_{Eu_{l}}^{l}\\
  \end{array}
\hspace{-2mm}\right]\hspace{-1mm}\varepsilon\delta Eu_{l}\Big\}\notag\\
  &\hspace{-1mm}=
  \hspace{-1mm}E\Big\{E\Big\{\hspace{-1mm}[\phi_{x_{N+1}}\hspace{-1mm}+\hspace{-1mm}E(\phi_{Ex_{N+1}})]\sum_{l=0}^{N}\hspace{-1mm}\tilde{F}_{x}(N,l\hspace{-1mm}+\hspace{-1mm}1)\hspace{-1mm}\left[\hspace{-2mm}
  \begin{array}{cc}
   f_{Eu_{l}}^{l}\\
  g_{Eu_{l}}^{l}\\
  \end{array}
\hspace{-2mm}\right]\hspace{-1mm}\Big\}\varepsilon\delta u_{l}
  \hspace{-1mm}\Big\}.\label{exp1}
\end{align}

Also,we have
\begin{align}\label{mp3}
&\sum_{k=0}^{N}[L^k_{x_k}+E(L^k_{Ex_k})]\sum_{l=0}^{k-1} \tilde{F}_{x}(k-\hspace{-1mm}1,l+\hspace{-1mm}1)\hspace{-1mm}\left[\hspace{-2mm}
  \begin{array}{cc}
    f_{u_{l}}^{l}\\
    g_{u_{l}}^{l}\\
  \end{array}
\hspace{-2mm}\right]\hspace{-1mm}\varepsilon\delta u_{l}\\
&\hspace{-1mm}=\hspace{-1mm}\sum_{l=0}^{N-1}\left\{\sum_{k=l+1}^{N}[L^k_{x_k}\hspace{-1mm}+\hspace{-1mm}E(L^k_{Ex_k})]\tilde{F}_{x}(k\hspace{-1mm}-\hspace{-1mm}1,l\hspace{-1mm}+\hspace{-1mm}1)\hspace{-1mm}\left[\hspace{-2mm}
  \begin{array}{cc}
    f_{u_{l}}^{l}\\
    g_{u_{l}}^{l}\\
  \end{array}
\hspace{-2mm}\right]\hspace{-1mm}\right\}\varepsilon\delta u_{l},\notag\\
&\notag\\
&\sum_{k=0}^{N}E\hspace{-1mm}\left\{[L^k_{x_k}\hspace{-1mm}+\hspace{-1mm}E(L^k_{Ex_k})]\sum_{l=0}^{k-1} \hspace{-1mm} \tilde{F}_{x}(k-1,l+1)\hspace{-1mm}\left[\hspace{-2mm}
  \begin{array}{cc}
    f_{Eu_{l}}^{l}\\
   g_{Eu_{l}}^{l}\\
  \end{array}
\hspace{-2mm}\right]\right\}\hspace{-1mm}\varepsilon\delta u_{l}\notag\\
&\hspace{-1.3mm}=\hspace{-1.6mm}\sum_{l=0}^{N\hspace{-0.5mm}-\hspace{-0.5mm}1}\hspace{-1mm}\left\{\hspace{-1mm}E\hspace{-1mm}\left\{
\sum_{k\hspace{-0.3mm}=\hspace{-0.3mm}l\hspace{-0.3mm}+\hspace{-0.3mm}1}^{N}\hspace{-1mm}[L^k_{x_k}\hspace{-1mm}+\hspace{-1mm}E(L^k_{Ex_k})]'\tilde{F}_{x}(k\hspace{-1mm}-\hspace{-1mm}1,l+\hspace{-1mm}1)\hspace{-1mm}\left[\hspace{-2mm}
  \begin{array}{cc}
   f_{Eu_{l}}^{l}\\
  g_{Eu_{l}}^{l} \\
  \end{array}
\hspace{-2mm}\right]\hspace{-1mm}\right\}\hspace{-1mm}\right\}\hspace{-1mm}\varepsilon\delta u_{l}.\label{mp3000}
\end{align}

Therefore, \eqref{mp02} becomes
\begin{align}\label{mp4}
\delta J_N &\hspace{-1mm}=\hspace{-1mm}E\Big\{\mathcal{G}(N+1,N)\varepsilon\delta u_{N}\hspace{-1mm}+\hspace{-1mm}\sum_{l=0}^{N-1}[\mathcal{G}(l\hspace{-1mm}+\hspace{-1mm}1,N)]\varepsilon\delta u_{l}\Big\}\hspace{-1mm}+\hspace{-1mm}O(\varepsilon^{2}),
\end{align}
where
\begin{align}
  &\mathcal{G}(N+1,N)=[\phi_{x_{N+1}}+E(\phi_{Ex_{N+1}})]f_{u_{N}}^{N}\notag\\
  &\hspace{-1mm}+\hspace{-1mm}E\Big\{[\phi_{x_{N+1}}\hspace{-1mm}+\hspace{-1mm}E(\phi_{Ex_{N+1}})]f_{Eu_{N}}^{N}\Big\}\hspace{-1mm}+\hspace{-1mm}[L_{u_{N}}^{N}\hspace{-1mm}+\hspace{-1mm}E(L_{Eu_{N}}^{N})],\label{mp04}\\
  &\mathcal{G}(l+1,N)\notag\\
  &=[\phi_{x_{N+1}}+E(\phi_{Ex_{N+1}})]\tilde{F}_{x}(N,l+1)\left[\hspace{-2mm}
  \begin{array}{cc}
    f_{u_{l}}^{l}\\
    g_{u_{l}}^{l}\\
  \end{array}
\hspace{-2mm}\right]\notag\\
&+E\left\{[\phi_{x_{N+1}}+E(\phi_{Ex_{N+1}})]\tilde{F}_{x}(N,l+1)\left[\hspace{-2mm}
  \begin{array}{cc}
   f_{Eu_{l}}^{l}\\
    g_{Eu_{l}}^{l} \\
  \end{array}
\hspace{-2mm}\right]\right\}\notag\\
  &+\sum_{k=l+1}^{N}[L_{x_{k}}^{k}+E(L_{Ex_{k}}^{k})]\tilde{F}_{x}(k-1,l+1)\left[\hspace{-2mm}
  \begin{array}{cc}
    f_{u_{l}}^{l}\\
   g_{u_{l}}^{l}\\
  \end{array}
\hspace{-2mm}\right]\notag\\
&+E\Big\{\sum_{k=l+1}^{N}[L_{x_{k}}^{k}+E(L_{Ex_{k}}^{k})]\tilde{F}_{x}(k-1,l+1)\left[\hspace{-2mm}
  \begin{array}{cc}
   f_{Eu_{l}}^{l}\\
   g_{Eu_{l}}^{l} \\
  \end{array}
\hspace{-2mm}\right]\Big\}\notag\\
  &+[L_{u_{l}}^{l}+E(L_{Eu_{l}}^{l})].\label{mp004}
\end{align}

Furthermore, \eqref{mp4} can be rewritten as
\begin{align}\label{mp9}
\delta J_{N}
&=E\Big\{ E \left[{\mathcal G}(N+1,N)\mid {\mathcal F}_{N-1}\right]\varepsilon\delta u_{N}\Big\}\notag\\
&+E\Big\{\sum_{l=0}^{N-1} E \left[{\mathcal G}(l+1, N) \mid {\mathcal
F}_{l-1}\right]\varepsilon\delta u_l\Big\}+O(\varepsilon^{2})\notag\\
&+E\left\lbrace  \left\lbrace {\mathcal G}(N\hspace{-1mm}+\hspace{-1mm}1,N)-E\left[{\mathcal
G}(N\hspace{-1mm}+\hspace{-1mm}1,N)\mid {\mathcal F}_{N-1}\right]\right\rbrace \varepsilon\delta u_{N}
\right\rbrace \notag\\
 &+E\Big\{\sum_{l=0}^{N-1}\left\lbrace  {\mathcal G}(l\hspace{-1mm}+\hspace{-1mm}1, N)\hspace{-1mm}-\hspace{-1mm}E
\left[{\mathcal G}(l\hspace{-1mm}+\hspace{-1mm}1, N) \mid {\mathcal F}_{l-1}\right] \right\rbrace
\varepsilon\delta u_l\Big\}\notag\\
&=E\Big\{ E \left[{\mathcal G}(N+1,N)\mid {\mathcal F}_{N-1}\right]\varepsilon\delta u_{N}\notag\\
&+\sum_{l=0}^{N-1} E \left[{\mathcal G}(l+1, N) \mid {\mathcal
F}_{l-1}\right]\varepsilon\delta u_l\Big\}\hspace{-1mm}+\hspace{-1mm}O(\varepsilon^{2}),
\end{align}
where the following facts are applied in the last equality,
\begin{equation*}\begin{split}
E\left\lbrace  \left\lbrace {\mathcal G}(N+1,N)\hspace{-1mm}-\hspace{-1mm}E\left[{\mathcal
G}(N+1,N)\mid {\mathcal F}_{N-1}\right]\right\rbrace \varepsilon\delta u_{N}
\right\rbrace
&=0, \notag\\
E\Big\{\sum_{l=0}^{N-1}\left\lbrace  {\mathcal G}(l\hspace{-1mm}+\hspace{-1mm}1, N)\hspace{-1mm}-\hspace{-1mm}E
\left[{\mathcal G}(l\hspace{-1mm}+\hspace{-1mm}1, N) \mid {\mathcal F}_{l-1}\right] \right\rbrace
\hspace{-1mm}\varepsilon\delta u_l\hspace{-1mm}\Big\}&=0.
\end{split}\end{equation*}

Since $\delta u_{l}$ is arbitrary for $0\leq l\leq N$, thus the necessary condition for the minimum can be given from (\ref{mp9}) as
\begin{align}
0&=E \left\lbrace {\mathcal G}(N+1,N)\mid {\mathcal F}_{N-1}\right\rbrace,\label{mp11}\\
0&=E\left\lbrace {\mathcal G}(l+1, N) \mid {\mathcal F}_{l-1} \right\rbrace,~l=0,\cdots,N-1.\label{mp011}
\end{align}

Now we will show that the equation \eqref{ps43}-\eqref{ps42} is a restatement of the necessary conditions \eqref{mp11}-\eqref{mp011}.

In fact, substituting \eqref{ps42} into
\eqref{ps43} and letting $k=N$, we have
\begin{align}\label{yz1}
&E\Big\lbrace (L^{N}_{u_{N}})'\hspace{-1mm}+\hspace{-1mm}E(L_{Eu_{N}}^{N})'\hspace{-1mm}+\hspace{-1mm}(f^{N}_{u_{N}})'
[\phi_{x_{N+1}}+E(\phi_{Ex_{N+1}})]'\notag\\
&+E\big\{(f_{Eu_{N}}^{N})'[\phi_{x_{N+1}}\hspace{-1mm}+\hspace{-1mm}E(\phi_{Ex_{N+1}})]'\big\}\Big|
{\mathcal F}_{N-1}\Big\rbrace \hspace{-1mm}=\hspace{-1mm}0,
\end{align}
which means that \eqref{yz1} is exactly \eqref{mp11}.

Furthermore, noting \eqref{ps41}, we have that
\begin{align}\label{mp14}
&\lambda_{k-1}\hspace{-1mm}=\hspace{-1mm}E\Big\{\left[\hspace{-2mm}
  \begin{array}{cc}
    I_{n}\\
    0   \\
  \end{array}
\hspace{-2mm}\right][(L^k_{x_k})'\hspace{-1mm}+\hspace{-1mm}E(L_{Ex_{k}}^{k})]'\hspace{-1mm}+\hspace{-1mm}[\tilde{f}^k_{x_k}]'\lambda_{k}\Big|\mathcal{F}_{k-1}\Big\}\notag\\
&=E\Big\{\hspace{-1.5mm}\left[\hspace{-2mm}
  \begin{array}{cc}
   I_{n} \\
    0  \\
  \end{array}
\hspace{-2mm}\right]\hspace{-1.5mm}[L^k_{x_k}\hspace{-1mm}+\hspace{-1mm}E(L_{Ex_{k}}^{k})]'\hspace{-1mm}+\hspace{-1mm}(\tilde{f}^k_{x_k})'\left[\hspace{-2mm}
  \begin{array}{cc}
    I_{n} \\
    0   \\
  \end{array}
\hspace{-2mm}\right][L^{k+1}_{x_k+1}\hspace{-1mm}+\hspace{-1mm}E(L_{Ex_{k}}^{k+1})]'\notag\\
&~~~+(\tilde{f}^k_{x_k})'(\tilde{f}^k_{x_k})'\lambda_{k+1}\Big|\mathcal{F}_{k-1}\Big\}\notag\\
&=E\Big\{\hspace{-1.5mm}\left[\hspace{-2mm}
  \begin{array}{cc}
  I_{n} \\
    0   \\
  \end{array}
\hspace{-2mm}\right]\hspace{-1.5mm}[L^k_{x_k}\hspace{-1mm}+\hspace{-1mm}E(L_{Ex_{k}}^{k})]'\hspace{-1mm}+\hspace{-1mm}(\tilde{f}^k_{x_k})'\left[\hspace{-2mm}
  \begin{array}{cc}
    I_{n} \\
    0    \\
  \end{array}
\hspace{-2mm}\right][L^k_{x_k+1}\hspace{-1mm}+\hspace{-1mm}E(L_{Ex_{k}}^{k+1})]'\notag\\
&~~+(\tilde{f}^k_{x_k})'(\tilde{f}^{k+1}_{x_k+1})'\left[\hspace{-2mm}
  \begin{array}{cc}
    I_{n} \\
    0      \\
  \end{array}
\hspace{-2mm}\right][L^{k+2}_{x_k+2}+E(L_{Ex_{k+2}}^{k+2})]'+\cdots\notag\\
&~~+(\tilde{f}^k_{x_k})'(\tilde{f}^{k+1}_{x_k+1})'\cdots(\tilde{f}^{N-1}_{x_N-1})'\left[\hspace{-2mm}
  \begin{array}{cc}
   I_{n} \\
    0   \\
  \end{array}
\hspace{-2mm}\right][L^{N}_{x_N}\hspace{-1mm}+\hspace{-1mm}E(L_{Ex_{N}}^{N})]'\notag\\
&~~+(\tilde{f}^k_{x_k})'(\tilde{f}^{k+1}_{x_k+1})'\cdots(\tilde{f}^{N}_{x_N})'\lambda_{N}\Big|\mathcal{F}_{k-1}\Big\}\notag\\
&=E\Big\{\sum_{j=k}^{N}\tilde{F}_{x}'(j-1,k)[L_{x_{j}}^{j}+E(L_{Ex_{j}}^{j})]\notag\\
&~~~+\tilde{F}_{x}'(N,k)[\phi_{x_{N+1}}+E(\phi_{Ex_{N+1}})]'\Big|\mathcal{F}_{k-1}\Big\}.
\end{align}

Substituting \eqref{mp14} into \eqref{ps43}, one has
\begin{align}\label{mp15}
0&\hspace{-1mm}=\hspace{-1mm}E\Bigg\{ [L^k_{u_k}+E(L_{Eu_{k}}^{k})]'\hspace{-1mm}\notag\\
&+\hspace{-1mm}   \sum_{j=k+1}^{N} \left[\hspace{-2mm}
  \begin{array}{cc}
     f_{u_{k}}^{k} \\
    g_{u_{k}}^{k}           \\
  \end{array}
\hspace{-2mm}\right]' \left\{\tilde{F}'_{x}(j-1,k+1)[L_{x_{j}}^{j}+E(L_{Ex_{j}}^{j})]'\right\}\notag\\
&+\hspace{-1mm}\left[\hspace{-2mm}
  \begin{array}{cc}
     f_{u_{k}}^{k} \\
    g_{u_{k}}^{k}  \\
  \end{array}
\hspace{-2mm}\right]' \left\{\tilde{F}_{x}'(N,k+1)[\phi_{x_{N+1}}+E(\phi_{Ex_{N+1}})]'\right\}\notag\\
&\hspace{-1mm}+\hspace{-1mm}E\hspace{-1mm}\bigg\{\hspace{-1mm}\sum_{j=k+1}^{N}\hspace{-1mm}\left[\hspace{-2mm}
  \begin{array}{cc}
    f_{Eu_{k}}^{k} \\
    g_{Eu_{k}}^{k} \\
  \end{array}
\hspace{-2mm}\right]'\hspace{-1mm}\left\{\tilde{F}'_{x}(j\hspace{-1mm}-\hspace{-1mm}1,k\hspace{-1mm}+\hspace{-1mm}1)[L_{x_{j}}^{j}\hspace{-1mm}+\hspace{-1mm}E(L_{Ex_{j}}^{j}]'\right\}
\bigg\}\notag\\
&\hspace{-1mm}+\hspace{-1mm}E\bigg\{\hspace{-1mm}\left[\hspace{-2mm}
  \begin{array}{cc}
    f_{Eu_{k}}^{k} \\
   g_{Eu_{k}}^{k}  \\
  \end{array}
\hspace{-2mm}\right]'\hspace{-1mm}\left\{\tilde{F}'_{x}(N,k\hspace{-1mm}+\hspace{-1mm}1)[\phi_{x_{N+1}}\hspace{-1mm}+\hspace{-1mm}E(\phi_{Ex_{N+1}})]'\hspace{-1mm}\right\}\hspace{-1mm}
\bigg\}\hspace{-1mm}\Bigg|\mathcal{F}_{k-1} \hspace{-1mm}\Bigg\},\notag\\
& ~~~k=0,\cdots, N,
\end{align}
which is \eqref{mp011}. It has been proved that \eqref{ps43}-\eqref{ps42} are exactly the necessary conditions for the minimum of $J_{N}$. The proof is complete.
\end{proof}

\section{Proof of Theorem \ref{main}}

\begin{proof}
``Necessity": Under Assumption \ref{ass1}, if \emph{Problem 1} has a unique solution, we will show by induction that
$\Upsilon_{k}^{(1)},~\Upsilon_{k}^{(2)}$ are all strictly positive definite and the optimal controller is given by \eqref{th43}.

Firstly, we denote $J(k)$ as below
\begin{align}
  J(k)&\triangleq \sum_{j=k}^{N}E\Big[x_{j}'Q_{j}x_{j}+(Ex_{j})'\bar{Q}_{j}Ex_{j}\notag\\
  &+u_{j}'R_{j}u_{j}+(Eu_{j})'\bar{R}_{j}Eu_{j}\Big]\notag\\
  &+E[x_{N+1}'P_{N+1}x_{N\hspace{-0.5mm}+\hspace{-0.5mm}1}]
  \hspace{-1mm}+\hspace{-1mm}(Ex_{N\hspace{-0.5mm}+\hspace{-0.5mm}1})'\bar{P}_{N\hspace{-0.5mm}+\hspace{-0.5mm}1}Ex_{N\hspace{-0.5mm}+\hspace{-0.5mm}1}.\label{barjk}
\end{align}

For $k=N$, equation \eqref{barjk} becomes
\begin{align}
 & J(N)=E\Big[x_{N}'Q_{N}x_{N}+(Ex_{N})'\bar{Q}_{N}Ex_{N}\notag\\
  &~~+u_{N}'R_{N}u_{N}+(Eu_{N})'\bar{R}_{N}Eu_{N}\Big]\notag\\
  &~~+E[x_{N+1}'P_{N+1}x_{N+1}]\hspace{-1mm}+\hspace{-1mm}(Ex_{N+1})'\bar{P}_{N+1}Ex_{N+1}.\label{barjn}
\end{align}

Using system dynamics \eqref{ps1}, $J(N)$ can be calculated as a quadratic form of $x_{N}$, $Ex_{N}$, $u_{N}$ and $Eu_{N}$. By Assumption \ref{ass1}, we know that the minimum of \eqref{barjn} must satisfy $J^{*}(N)\geq 0$.

Let $x_{N}=0$, since it is assumed \emph{Problem \ref{prob1}} admits a unique solution, thus it is clear that $u_{N}=0$ is the optimal controller and optimal cost function is $J^{*}(N)=0$.

Hence, $J(N)$ must be strictly positive for any nonzero $u_{N}$, i.e., for $u_{N}\neq 0$, we can obtain
\begin{align}
J(N)&=E[(u_{N}\hspace{-1mm}-\hspace{-1mm}Eu_{N})'\Upsilon_{N}^{(1)}(u_{N}\hspace{-1mm}-\hspace{-1mm}Eu_{N})]\hspace{-1mm}
  +\hspace{-1mm}Eu_{N}'\Upsilon_{N}^{(2)}Eu_{N}\notag\\
  &>0.\label{barjn1}
\end{align}

Following Lemma \ref{lemma01}, clearly we have $\Upsilon_{N}^{(1)}>0$ and  $\Upsilon_{N}^{(2)}>0$ from \eqref{barjn1}. In fact, in the case $Eu_{N}=0$ and $u_{N}\neq 0$, equation \eqref{barjn1} becomes
\begin{equation*}
  J(N)=E[u_{N}'\Upsilon_{N}^{(1)}u_{N}]>0.
\end{equation*}
Thus $\Upsilon_{N}^{(1)}>0$ can be obtained by using Lemma \ref{lemma01} and Remark \ref{rem1}.

On the other hand, if $u_{N}=Eu_{N}\neq 0$, i.e., $u_{N}$ is deterministic controller, then \eqref{barjn1} can be reduced to
\begin{equation*}
  J(N)=u_{N}'\Upsilon_{N}^{(2)}u_{N}>0.
\end{equation*}
Similarly, it holds from Lemma \ref{lemma01} and Remark \ref{rem1} that $\Upsilon_{N}^{(2)}>0$.

Further the optimal controller $u_{N}$ is to be calculated as follows.

Using \eqref{ps1} and \eqref{th31}, from \eqref{th33} with $k$ replaced by $N$, we have that
\begin{align}\label{nc2}
&0\hspace{-1mm}=\hspace{-1mm}E\Big\{
R_{N}u_{N}\hspace{-1mm}+\hspace{-1mm}\bar{R}_{N}Eu_{N}\hspace{-1mm}+\hspace{-1mm}\left[\hspace{-1mm}
  \begin{array}{cc}
   B_{N}+w_{N}D_{N} \\
    0     \\
  \end{array}
\hspace{-1mm}\right]'\lambda_{N}\hspace{-1mm}\notag\\
&~~~~~~~~~+\hspace{-1mm}E\left[\hspace{-1mm}\left[\hspace{-1mm}
  \begin{array}{cc}
   \bar{B}_{N}+w_{N}\bar{D}_{N} \\
    B_{N}+\bar{B}_{N}          \\
  \end{array}
\hspace{-1mm}\right]'\lambda_{N}\right]\hspace{-1mm}\Big| {\mathcal F}_{N-1}\hspace{-1mm}\Big\}\notag\\
&=E\Big\{R_{N}u_{N}+\bar{R}_{N}Eu_{N}\notag\\
&~~~+(B_{N}+w_{N}D_{N})'(P_{N+1}x_{N+1}+\bar{P}_{N+1}^{(1)}Ex_{N+1})\notag\\
&~~~+E[(\bar{B}_{N}\hspace{-1mm}+\hspace{-1mm}w_{N}\bar{D}_{N})'(P_{N+1}x_{N+1}\hspace{-1mm}+\hspace{-1mm}\bar{P}_{N+1}^{(1)}Ex_{N+1})]\notag\\
&~~~+E[(B_{N}\hspace{-1mm}+\hspace{-1mm}\bar{B}_{N})'(\bar{P}_{N+1}^{(2)}x_{N+1}\hspace{-1mm}+\hspace{-1mm}\bar{P}_{N+1}^{(3)}Ex_{N+1})]\Big|\mathcal{F}_{N-1}\Big\}\notag\\
&=(R_{N}+B_{N}'P_{N+1}B_{N}+\sigma^{2}D_{N}'P_{N+1}D_{N})u_{N}\notag\\
&~~~+\Big[\bar{R}_{N}+B_{N}'P_{N+1}\bar{B}_{N}+\sigma^{2}D_{N}'P_{N+1}\bar{D}_{N}\notag\\
&~~~+B_{N}'\bar{P}_{N+1}^{(1)}(B_{N}+\bar{B}_{N})+\bar{B}_{N}'\bar{P}_{N+1}^{(1)}(B_{N}\hspace{-1mm}+\hspace{-1mm}\bar{B}_{N})\notag\\
&~~~+\bar{B}_{N}'P_{N+1}B_{N}+\sigma^{2}\bar{D}_{N}'P_{N+1}D_{N}\notag\\
&~~~+\bar{B}_{N}'P_{N+1}\bar{B}_{N}+\sigma^{2}\bar{D}_{N}'P_{N+1}\bar{D}_{N}\notag\\
&~~~+(B_{N}+\bar{B}_{N})'(\bar{P}_{N+1}^{(2)}+\bar{P}_{N+1}^{(3)})(B_{N}+\bar{B}_{N})\Big]Eu_{N}\notag\\
&~~~+(B_{N}'P_{N+1}A_{N}+\sigma^{2}D_{N}'P_{N+1}C_{N})x_{N}\notag\\
&~~~+\Big[B_{N}'P_{N+1}\bar{A}_{N}+\sigma^{2}D_{N}'P_{N+1}\bar{C}_{N}\notag\\
&~~~+B_{N}'\bar{P}_{N+1}^{(1)}(A_{N}+\bar{A}_{N})+\bar{B}_{N}'\bar{P}_{N+1}^{(1)}(A_{N}\hspace{-1mm}+\hspace{-1mm}\bar{A}_{N})\notag\\
&~~~+\bar{B}_{N}'P_{N+1}A_{N}+\sigma^{2}\bar{D}_{N}'P_{N+1}C_{N}\notag\\
&~~~+\bar{B}_{N}'P_{N+1}\bar{A}_{N}+\sigma^{2}\bar{D}_{N}'P_{N+1}\bar{C}_{N}\notag\\
&~~~+(B_{N}+\bar{B}_{N})'(\bar{P}_{N+1}^{(2)}\hspace{-1mm}+\hspace{-1mm}\bar{P}_{N+1}^{(3)})(A_{N}\hspace{-1mm}+\hspace{-1mm}\bar{A}_{N})\Big]Ex_{N}.
\end{align}

Note that $\bar{P}_{N+1}^{(1)}+\bar{P}_{N+1}^{(2)}+\bar{P}_{N+1}^{(3)}=\bar{P}_{N+1}$, it follows from \eqref{nc2} that
\begin{align}
  0&=\Upsilon_{N}^{(1)}u_{N}+[\Upsilon_{N}^{(2)}-\Upsilon_{N}^{(1)}]Eu_{N}\notag\\
  &~~+M_{N}^{(1)}x_{N}+[M_{N}^{(2)}-M_{N}^{(1)}]Ex_{N},\label{nc3}
\end{align}
where $\Upsilon_{N}^{(1)}, \Upsilon_{N}^{(2)}, M_{N}^{(1)},M_{N}^{(2)}$ are given by \eqref{upsi1}-\eqref{h2} for $k=N$.

Therefore, taking expectations on both sides of \eqref{nc3}, we have
\begin{equation}\label{nc03}
  \Upsilon_{N}^{(2)}Eu_{N}+M_{N}^{(2)}Ex_{N}=0.
\end{equation}

Since $\Upsilon_{N}^{(1)}$, and $\Upsilon_{N}^{(2)}$ has been proved to be strictly positive, thus $Eu_{N}$ can be presented as
\begin{equation}\label{nc003}
  Eu_{N}=-[\Upsilon_{N}^{(2)}]^{-1} M_{N}^{(2)}Ex_{N}.
\end{equation}

By plugging \eqref{nc003} into \eqref{nc3}, the optimal controller $u_{N}$ given by \eqref{th43} with $k=N$ can be verified.

Next we will show $\lambda_{N-1}$ has the form of \eqref{th4} associated with \eqref{th41}-\eqref{th42} for $k=N$.

Notice \eqref{th31} and \eqref{th32}, we have that
\begin{align}\label{nc8}
  \lambda_{N-1}&=E\Big\{\left[\hspace{-1mm}
  \begin{array}{cc}
    Q_{N}x_{N}+\bar{Q}_{N}Ex_{N}\\
    0               \\
  \end{array}
\hspace{-1mm}\right]\hspace{-1mm}\notag\\
&~+\hspace{-1mm}\left[\hspace{-1mm}
  \begin{array}{cc}
    A_{N}+w_{N}C_{N}&\bar{A}_{N}+w_{N}\bar{C}_{N}\\
    0    & A_{N}+\bar{A}_{N}     \\
  \end{array}
\hspace{-1mm}\right]'\lambda_{N}\Big|\mathcal{F}_{N-1}\Big\}\notag\\
&=\hspace{-1mm}E\Big\{\left[\hspace{-1mm}
  \begin{array}{cc}
    Q_{N}x_{N}+\bar{Q}_{N}Ex_{N}\\
    0             \\
  \end{array}
\hspace{-1mm}\right]\hspace{-1mm}\notag\\
&~+\hspace{-1mm}\left[\hspace{-1mm}
  \begin{array}{cc}
    A_{N}\hspace{-1mm}+\hspace{-1mm}w_{N}C_{N}&
   \bar{A}_{N}\hspace{-1mm}+\hspace{-1mm}w_{N}\bar{C}_{N}\\
    0       & A_{N}\hspace{-1mm}+\hspace{-1mm}\bar{A}_{N}      \\
  \end{array}
\hspace{-1mm}\right]'\hspace{-1mm}\left[\hspace{-1mm}
  \begin{array}{cc}
   P_{N+1}& \bar{P}_{N+1}^{(1)}\\
    \bar{P}_{N+1}^{(2)}& \bar{P}_{N+1}^{(3)}     \\
  \end{array}
\hspace{-1mm}\right]\notag\\
&~~~\times\left[
  \begin{array}{cc}
    x_{N+1}\\
    Ex_{N+1}              \\
  \end{array}
\right]\hspace{-1mm}\Big|\mathcal{F}_{N-1}\Big\}.
\end{align}

By using the optimal controller \eqref{th43} and the system dynamics \eqref{ps1}, each element of $\lambda_{N-1}$ can be calculated as follows,
\begin{align}
  &E[(A_{N}+w_{N}C_{N})'P_{N+1}x_{N+1}|\mathcal{F}_{N-1}]\notag\\
  &=\Big(A_{N}'P_{N+1}A_{N}+\sigma^{2}C_{N}'P_{N+1}C_{N}\notag\\
  &~~~+A_{N}'P_{N+1}B_{N}K_{N}+\sigma^{2}C_{N}'P_{N+1}D_{N}K_{N}\Big)x_{N}\notag\\
  &~~~+\Big[A_{N}'P_{N+1}\bar{A}_{N}+\sigma^{2}C_{N}'P_{N+1}\bar{C}_{N}\notag\\
  &~~~+A_{N}'P_{N+1}B_{N}\bar{K}_{N}+\sigma^{2}C_{N}'P_{N+1}D_{N}\bar{K}_{N}\notag\\
  &~~~+A_{N}'P_{N+1}\bar{B}_{N}(K_{N}+\bar{K}_{N})\notag\\
  &~~~+\sigma^{2}C_{N}'P_{N+1}\bar{D}_{N}(K_{N}+\bar{K}_{N})\Big ]Ex_{N}, \label{nc9}\\
  &E[(A_{N}+w_{N}C_{N})'\bar{P}_{N+1}^{(1)}Ex_{N+1}|\mathcal{F}_{N-1}]\notag\\
  &=\Big\{A_{N}'\bar{P}_{N+1}^{(1)}(A_{N}+\bar{A}_{N})\notag\\
  &~~~+A_{N}'\bar{P}_{N+1}^{(1)}(B_{N}+\bar{B}_{N})(K_{N}+\bar{K}_{N})\Big\}Ex_{N},\label{nc90}\\
&E\Big\{\big[(\bar{A}_{N}\hspace{-1mm}+\hspace{-1mm}w_{N}\bar{C}_{N})'P_{N+1}\notag\\
&~~~+(A_{N}+\bar{A}_{N})'\bar{P}_{N+1}^{(2)}\big]x_{N+1}\Big|\mathcal{F}_{N-1}\Big\}\notag\\
  &=\Big\{\bar{A}_{N}'P_{N+1}A_{N}+\sigma^{2}\bar{C}_{N}'P_{N+1}C_{N}\notag\\
  &~~~+\bar{A}_{N}'P_{N+1}B_{N}K_{N}+\sigma^{2}\bar{C}_{N}'P_{N+1}D_{N}K_{N}\notag\\
  &~~~+(A_{N}+\bar{A}_{N})'\bar{P}_{N+1}^{(2)}A_{N}\notag\\
  &~~~+(A_{N}+\bar{A}_{N})'\bar{P}_{N+1}^{(2)}B_{N}K_{N}\Big\}x_{N}\notag\\
  &~~~+\Big\{\bar{A}_{N}'P_{N+1}\bar{A}_{N}+\sigma^{2}\bar{C}_{N}'P_{N+1}\bar{C}_{N}\notag\\
  &~~~+\bar{A}_{N}'P_{N+1}B_{N}\bar{K}_{N}+\sigma^{2}\bar{C}_{N}'P_{N+1}D_{N}\bar{K}_{N}\notag\\
  &~~~+\bar{A}_{N}'P_{N+1}\bar{B}_{N}(K_{N}+\bar{K}_{N})\notag\\
  &~~~+\sigma^{2}\bar{C}_{N}'P_{N+1}\bar{D}_{N}(K_{N}+\bar{K}_{N})\notag\\
  &~~~+(A_{N}+\bar{A}_{N})'\bar{P}_{N+1}^{(2)}B_{N}\bar{K}_{N}\notag\\
  &~~~+(A_{N}+\bar{A}_{N})'\bar{P}_{N+1}^{(2)}\bar{B}_{N}(K_{N}+\bar{K}_{N})\notag\\
  &~~~+(A_{N}+\bar{A}_{N})'\bar{P}_{N+1}^{(2)}\bar{A}_{N}\Big\}Ex_{N},\label{nc10}
\end{align}
and
\begin{align}\label{nc101}
&E\{[(\bar{A}_{N}\hspace{-1mm}+\hspace{-1mm}w_{N}\bar{C}_{N})'\bar{P}_{N+1}^{(1)}\notag\\
&~~~+(A_{N}+\bar{A}_{N})'\bar{P}_{N+1}^{(3)}]Ex_{N+1}|\mathcal{F}_{N-1}\}\notag\\
&=\Big\{\bar{A}_{N}'\bar{P}_{N+1}^{(1)}(A_{N}+\bar{A}_{N})\notag\\
&~~~+\bar{A}_{N}'\bar{P}_{N+1}^{(1)}(B_{N}+\bar{B}_{N})(K_{N}+\bar{K}_{N})\notag\\
&~~~+(A_{N}+\bar{A}_{N})'\bar{P}_{N+1}^{(3)}(B_{N}+\bar{B}_{N})(K_{N}+\bar{K}_{N})\notag\\
&~~~+(A_{N}+\bar{A}_{N})'\bar{P}_{N+1}^{(3)}(A_{N}+\bar{A}_{N})\Big\}Ex_{N}.
\end{align}

By plugging \eqref{nc9}-\eqref{nc101} into \eqref{nc8}, we know that $\lambda_{N-1}$ is given as,
\begin{equation}\label{nc08}
\lambda_{N-1}=\left[
  \begin{array}{cc}
    P_{N}&\bar{P}_{N}^{(1)}\\
    \bar{P}_{N}^{(2)}  &\bar{P}_{N}^{(3)}     \\
  \end{array}
\right]\left[
  \begin{array}{cc}
    x_{N}\\
    Ex_{N}             \\
  \end{array}
\right],
\end{equation}
where $\bar{P}_{N}^{(1)}$, $\bar{P}_{N}^{(2)}$, $\bar{P}_{N}^{(3)}$ are respectively calculated in the following,
\begin{align}
  \bar{P}_{N}^{(1)}&=\bar{Q}_{N}+A_{N}'P_{N+1}\bar{A}_{N}+\sigma^{2}C_{N}'P_{N+1}\bar{C}_{N}\notag\\
  &~~~+A_{N}'P_{N+1}B_{N}\bar{K}_{N}+\sigma^{2}C_{N}'P_{N+1}D_{N}\bar{K}_{N}\notag\\
  &~~~+A_{N}'P_{N+1}\bar{B}_{N}(K_{N}+\bar{K}_{N})\notag\\
  &~~~+\sigma^{2}C_{N}'P_{N+1}\bar{D}_{N}(K_{N}+\bar{K}_{N})\notag\\
  &~~~+A_{N}'\bar{P}_{N+1}^{(1)}(A_{N}+\bar{A}_{N}) \notag\\
  &~~~+A_{N}'\bar{P}_{N+1}^{(1)}(B_{N}+\bar{B}_{N})(K_{N}+\bar{K}_{N}),\label{pk1}\\
  \bar{P}_{N}^{(2)}&=\bar{A}_{N}'P_{N+1}A_{N}+\sigma^{2}\bar{C}_{N}'P_{N+1}C_{N}\notag\\
  &~~~+\bar{A}_{N}'P_{N+1}B_{N}K_{N}+\sigma^{2}\bar{C}_{N}'P_{N+1}D_{N}K_{N}\notag\\
  &~~~+(A_{N}\hspace{-1mm}+\hspace{-1mm}\bar{A}_{N})'\bar{P}_{N+1}^{(2)}A_{N}\hspace{-1mm}+\hspace{-1mm}
  (A_{N}\hspace{-1mm}+\hspace{-1mm}\bar{A}_{N})'\bar{P}_{N+1}^{(2)}B_{N}K_{N},\label{pk2}\\
  \bar{P}_{N}^{(3)}&=\bar{A}_{N}'P_{N+1}\bar{A}_{N}+\sigma^{2}\bar{C}_{N}'P_{N+1}\bar{C}_{N}\notag\\
  &~~~+\bar{A}_{N}'P_{N+1}B_{N}\bar{K}_{N}+\sigma^{2}\bar{C}_{N}'P_{N+1}D_{N}\bar{K}_{N}\notag\\
  &~~~+\bar{A}_{N}'P_{N+1}\bar{B}_{N}(K_{N}+\bar{K}_{N})\notag\\
  &~~~+\sigma^{2}\bar{C}_{N}'P_{N+1}\bar{D}_{N}(K_{N}+\bar{K}_{N})\notag\\
  &~~~+(A_{N}+\bar{A}_{N})'\bar{P}_{N+1}^{(2)}\bar{A}_{N}\notag\\
  &~~~+(A_{N}+\bar{A}_{N})'\bar{P}_{N+1}^{(2)}B_{N}\bar{K}_{N}\notag\\
  &~~~+(A_{N}+\bar{A}_{N})'\bar{P}_{N+1}^{(2)}\bar{B}_{N}(K_{N}+\bar{K}_{N})\notag\\
  &~~~+\bar{A}_{N}'\bar{P}_{N+1}^{(1)}(A_{N}+\bar{A}_{N})\notag\\
  &~~~+\bar{A}_{N}'\bar{P}_{N+1}^{(1)}(B_{N}+\bar{B}_{N})(K_{N}+\bar{K}_{N})\notag\\
  &~~~+(A_{N}+\bar{A}_{N})'\bar{P}_{N+1}^{(3)}(B_{N}+\bar{B}_{N})(K_{N}+\bar{K}_{N})\notag\\
  &~~~+(A_{N}+\bar{A}_{N})'\bar{P}_{N+1}^{(3)}(A_{N}+\bar{A}_{N}),\label{pk3}
\end{align}
with $\bar{P}_{N+1}^{(1)}=\bar{P}_{N+1}$, $\bar{P}_{N+1}^{(2)}=\bar{P}_{N+1}^{(3)}=0$.

Similarly, $P_N$ is given as
\begin{align}\label{pn}
  P_{N}&=Q_{N}+A_{N}'P_{N+1}A_{N}+\sigma^{2}C_{N}'P_{N+1}C_{N}\notag\\
  &~~~+A_{N}'P_{N+1}B_{N}K_{N}+\sigma^{2}C_{N}'P_{N+1}D_{N}K_{N}\notag\\
  &=Q_{N}+A_{N}'P_{N+1}A_{N}+\sigma^{2}C_{N}'P_{N+1}C_{N}\notag\\
  &~~~-(A_{N}'P_{N+1}B_{N}+\sigma^{2}C_{N}'P_{N+1}D_{N})[\Upsilon_{N}^{(1)}]^{-1}M_{N}^{(1)}\notag\\
  &=Q_{N}+A_{N}'P_{N+1}A_{N}+\sigma^{2}C_{N}'P_{N+1}C_{N}\notag\\
  &~~~-[M_{N}^{(1)}]'[\Upsilon_{N}^{(1)}]^{-1}M_{N}^{(1)},
\end{align}
which is exactly \eqref{th41} for $k=N$. Now we show $\bar{P}_{N}=\bar{P}_{N}^{(1)}+\bar{P}_{N}^{(2)}+\bar{P}_{N}^{(3)}$ obeys \eqref{th42}. In fact, it holds from \eqref{pk1}-\eqref{pn} that
\begin{align}\label{barpn}
  &\bar{P}_{N}=\bar{P}_{N}^{(1)}+\bar{P}_{N}^{(2)}+\bar{P}_{N}^{(3)}\notag\\
  &=\bar{Q}_{N}+A_{N}'P_{N+1}\bar{A}_{N}+\sigma^{2}C_{N}'P_{N+1}\bar{C}_{N}\notag\\
  &~~~+\bar{A}_{N}'P_{N+1}\bar{A}_{N}+\sigma^{2}\bar{C}_{N}'P_{N+1}\bar{C}_{N}\notag\\
  &~~~+\bar{A}_{N}'P_{N+1}A_{N}+\sigma^{2}\bar{C}_{N}'P_{N+1}C_{N}\notag\\
  &~~~+\Big[A_{N}'P_{N+1}B_{N}+\sigma^{2}C_{N}'P_{N+1}D_{N}\notag\\
  &~~~+\bar{A}_{N}'P_{N+1}B_{N}+\sigma^{2}\bar{C}_{N}'P_{N+1}D_{N}\notag\\
  &~~~+A_{N}'P_{N+1}\bar{B}_{N}+\sigma^{2}C_{N}'P_{N+1}\bar{D}_{N}\notag\\
  &~~~+\bar{A}_{N}'P_{N+1}\bar{B}_{N}+\sigma^{2}\bar{C}_{N}'P_{N+1}\bar{D}_{N}\notag\\
  &~~~+(A_{N}+\bar{A}_{N})'\bar{P}_{N+1}(A_{N}+\bar{A}_{N})\Big](K_{N}+\bar{K}_{N})\notag\\
  &~~~-(A_{N}'P_{N+1}B_{N}+\sigma^{2}C_{N}'P_{N+1}D_{N})K_{N}\notag\\
  &=\bar{Q}_{N}+A_{N}'P_{N+1}\bar{A}_{N}+\sigma^{2}C_{N}'P_{N+1}\bar{C}_{N}\notag\\
  &~~~+\bar{A}_{N}'P_{N+1}\bar{A}_{N}+\sigma^{2}\bar{C}_{N}'P_{N+1}\bar{C}_{N}\notag\\
  &~~~+\bar{A}_{N}'P_{N+1}A_{N}+\sigma^{2}\bar{C}_{N}'P_{N+1}C_{N}\notag\\
   &~~~+(A_{N}+\bar{A}_{N})'\bar{P}_{N+1}(A_{N}+\bar{A}_{N})\notag\\
  &~~~+[M_{N}^{(1)}]'[\Upsilon_{N}^{(1)}]^{-1}M_{N}^{(1)}-[M_{N}^{(2)}]'[\Upsilon_{N}^{(2)}]^{-1}M_{N}^{(2)}.
\end{align}
where $\bar{P}_{N+1}^{(1)}+\bar{P}_{N+1}^{(2)}+\bar{P}_{N+1}^{(3)}=\bar{P}_{N+1}$ has been inserted to the second equality of \eqref{barpn}.

Thus, \eqref{th4} associated with \eqref{th41}-\eqref{th42} have been verified for $k=N$.

Therefore we have shown the necessity for $k=N$ in the above. To complete the induction,  take $0\leq l\leq N$, for any $k\geq l+1$, we assume that:
\begin{itemize}
\item  $\Upsilon_{k}^{(1)}$ and $\Upsilon_{k}^{(2)}$ in \eqref{upsi1} and \eqref{upsi2} are all strictly positive;
\item  The costate $\lambda_{k-1}$ is given by \eqref{th4}, $P_{k}$ satisfies \eqref{th41} and $\bar{P}_{k}^{(1)}$, $\bar{P}_{k}^{(2)}$, $\bar{P}_{k}^{(3)}$ satisfy \eqref{pk1}-\eqref{pk3} with $N$ replaced by $k$, respectively. Furthermore, $\bar{P}_{k}^{(1)}+\bar{P}_{k}^{(2)}+\bar{P}_{k}^{(3)}=\bar{P}_{k}$ and $\bar{P}_{k}$ obeys \eqref{th42};
\item  The optimal controller $u_{k}$ is as in \eqref{th43}.
\end{itemize}

We will show the above statements  are also true for $k=l$.

Firstly, we show $\Upsilon_{l}^{(1)}$ and $\Upsilon_{l}^{(2)}$ are positive definite if {\em Problem 1} has a unique solution.

By applying the maximum principle \eqref{th33}-\eqref{th32} and \eqref{ps1}, we can obtain
\begin{align*}
  &~~E\Big\{\left[\hspace{-1mm}
  \begin{array}{cc}
    x_{k}\\
    Ex_{k}\\
  \end{array}
\hspace{-1mm}\right]'\lambda_{k-1}-\left[\hspace{-1mm}
  \begin{array}{cc}
    x_{k+1}\\
    Ex_{k+1}\\
  \end{array}
\hspace{-1mm}\right]'\lambda_{k}\Big\}\\
&\hspace{-1mm}=\hspace{-1mm} E\Big\{\hspace{-1mm}\left[\hspace{-1mm}
  \begin{array}{cc}
     x_{k} \\
    Ex_{k} \\
  \end{array}
\hspace{-1mm}\right]'E\Big\{\hspace{-1mm}\left[\hspace{-1mm}
  \begin{array}{cc}
     A_{k}+w_{k}C_{k} &  \bar{A}_{k}+w_{k}\bar{C}_{k} \\
    0      &  A_{k}+\bar{A}_{k}       \\
  \end{array}
\hspace{-1mm}\right]'\hspace{-1mm}\lambda_{k}\Big|\mathcal{F}_{k-1}\Big\}\\
&+\left[\hspace{-1mm}
  \begin{array}{cc}
    x_{k} \\
    Ex_{k}  \\
  \end{array}
\hspace{-1mm}\right]'\left[\hspace{-1mm}
  \begin{array}{cc}
    Q_{k}x_{k}+\bar{Q}_{k}Ex_{k} \\
    0     \\
  \end{array}
\hspace{-1mm}\right]\\
&-\left[\hspace{-1mm}
  \begin{array}{cc}
     x_{k} \\
    Ex_{k}  \\
  \end{array}
\hspace{-1mm}\right]'\left[\hspace{-1mm}
  \begin{array}{cc}
    A_{k}+w_{k}C_{k} & \bar{A}_{k}+w_{k}\bar{C}_{k} \\
    0     &  A_{k}+\bar{A}_{k}       \\
  \end{array}
\hspace{-1mm}\right]'\lambda_{k}\\
&-\left[\hspace{-1mm}
  \begin{array}{cc}
    u_{k} \\
    Eu_{k} \\
  \end{array}
\hspace{-1mm}\right]'\left[\hspace{-1mm}
  \begin{array}{cc}
    B_{k}+w_{k}D_{k} &  \bar{B}_{k}+w_{k}\bar{D}_{k} \\
    0     & B_{k}+\bar{B}_{k}       \\
  \end{array}
\hspace{-1mm}\right]'\lambda_{k}
\hspace{-1mm}\Big\}\\
&\hspace{-1mm}=\hspace{-1mm} E\Big\{\hspace{-1mm}\left[\hspace{-1mm}
  \begin{array}{cc}
    x_{k} \\
    Ex_{k} \\
  \end{array}
\hspace{-1mm}\right]'E\Big\{\hspace{-1mm}\left[\hspace{-1mm}
  \begin{array}{cc}
    A_{k}+w_{k}C_{k} &  \bar{A}_{k}+w_{k}\bar{C}_{k} \\
    0   &  A_{k}+\bar{A}_{k}       \\
  \end{array}
\hspace{-1mm}\right]'\hspace{-1mm}\lambda_{k}\Big|\mathcal{F}_{k-1}\Big\}\\
&+\left[\hspace{-1mm}
  \begin{array}{cc}
    x_{k} \\
    Ex_{k}  \\
  \end{array}
\hspace{-1mm}\right]'\left[
  \begin{array}{cc}
   Q_{k}x_{k}+\bar{Q}_{k}Ex_{k} \\
    0            \\
  \end{array}
\hspace{-1mm}\right]\\
&-\left[
  \begin{array}{cc}
     x_{k} \\
    Ex_{k}  \\
  \end{array}
\hspace{-1mm}\right]'\left[\hspace{-1mm}
  \begin{array}{cc}
    A_{k}+w_{k}C_{k} &  \bar{A}_{k}+w_{k}\bar{C}_{k} \\
    0       &  A_{k}+\bar{A}_{k}      \\
  \end{array}
\hspace{-1mm}\right]'\lambda_{k}\\
&\hspace{-1mm}-\hspace{-1mm}u_{k}'\left[
  \begin{array}{cc}
     B_{k}+w_{k}D_{k} \\
    0     \\
  \end{array}
\hspace{-1mm}\right]'\lambda_{k}\hspace{-1mm}-\hspace{-1mm}u_{k}'E\Big\{\left[\hspace{-1mm}
  \begin{array}{cc}
     \bar{B}_{k}+w_{k}\bar{D}_{k}\\
    B_{k}+\bar{B}_{k}       \\
  \end{array}
\hspace{-1mm}\right]'\lambda_{k}\Big\}\Big\}\\
&\hspace{-1mm}=\hspace{-1mm}E\hspace{-1mm}\Big\{\hspace{-1mm}\left[\hspace{-1mm}
  \begin{array}{cc}
     x_{k} \\
    Ex_{k}  \\
  \end{array}
\hspace{-1mm}\right]'\hspace{-1mm}\left[\hspace{-1mm}
  \begin{array}{cc}
     Q_{k}x_{k}\hspace{-1mm}+\hspace{-1mm}\bar{Q}_{k}Ex_{k} \\
    0            \\
  \end{array}
\hspace{-1mm}\right]
\hspace{-1mm}\Big\}\hspace{-1mm}+\hspace{-1mm}E(u_{k}'R_{k}u_{k}\hspace{-1mm}+\hspace{-1mm}Eu_{k}'\bar{R}_{k}Eu_{k})\\
&=E(x_{k}'Q_{k}x_{k}\hspace{-1mm}+\hspace{-1mm}Ex_{k}'\bar{Q}_{k}Ex_{k}\hspace{-1mm}+\hspace{-1mm}u_{k}'R_{k}u_{k}\hspace{-1mm}+\hspace{-1mm}Eu_{k}'\bar{R}_{k}Eu_{k}).
    \end{align*}

Adding from $k=l+1$ to $k=N$ on both sides of the above equation, we have
\begin{align*}
  &E\Big\{\hspace{-1mm} \left[\hspace{-1mm}
  \begin{array}{cc}
     x_{l+1} \\
    Ex_{l+1} \\
  \end{array}
\hspace{-1mm}\right]'\hspace{-1mm}\lambda_{l}\hspace{-1mm}
-\hspace{-1mm}x_{N+1}'P_{N+1}x_{N+1}\hspace{-1mm}-\hspace{-1mm}Ex_{N+1}'P_{N+1}Ex_{N+1}\Big\}\\
&\hspace{-1mm}=\hspace{-1mm}\sum_{k=l+1}^{N}\hspace{-1mm}E(x_{k}'Q_{k}x_{k}\hspace{-1mm}+\hspace{-1mm}Ex_{k}'\bar{Q}_{k}Ex_{k}\hspace{-1mm}+\hspace{-1mm}u_{k}'R_{k}u_{k}\hspace{-1mm}+\hspace{-1mm}Eu_{k}'\bar{R}_{k}Eu_{k}).
\end{align*}

Thus, it follows from \eqref{barjk} that
\begin{align}\label{jnnn}
 & J(l)= E(x_{l}'Q_{l}x_{l}+Ex_{l}'\bar{Q}_{l}Ex_{l}+u_{l}'R_{l}u_{l}+Eu_{l}'\bar{R}_{l}Eu_{l})\notag\\
  &+\sum_{k=l+1}^{N}\hspace{-1mm}E(x_{k}'Q_{k}x_{k}\hspace{-1mm}+\hspace{-1mm}Ex_{k}'\bar{Q}_{k}Ex_{k}\hspace{-1mm}+\hspace{-1mm}u_{k}'R_{k}u_{k}\hspace{-1mm}+\hspace{-1mm}Eu_{k}'\bar{R}_{k}Eu_{k})\notag\\
  &+E(x_{N+1}'P_{N+1}x_{N+1}\hspace{-1mm}+\hspace{-1mm}Ex_{N+1}'P_{N+1}Ex_{N+1})\notag\\
  &=E\Big\{x_{l}'Q_{l}x_{l}+Ex_{l}'\bar{Q}_{l}Ex_{l}+u_{l}'R_{l}u_{l}+Eu_{l}'\bar{R}_{l}Eu_{l}\notag\\
  &+\left[\hspace{-1mm}
  \begin{array}{cc}
     x_{l+1} \\
    Ex_{l+1}              \\
  \end{array}
\hspace{-1mm}\right]'\lambda_{l}\Big\},
\end{align}

Note that \eqref{th4} is assumed to be true for $k=l+1$, i.e.,
\begin{align}\label{ldan}
\lambda_{l}&=\left[
  \begin{array}{cc}
     P_{l+1} & \bar{P}_{l+1}^{(1)} \\
    \bar{P}_{l+1}^{(2)}   & \bar{P}_{l+1}^{(3)}       \\
  \end{array}
 \right]\left[
  \begin{array}{cc}
     x_{l+1} \\
    Ex_{l+1}              \\
  \end{array}
 \right],
\end{align}
where $P_{l+1}$ follows the iteration \eqref{th41} and $\bar{P}_{l+1}^{(1)}$, $\bar{P}_{l+1}^{(2)}$, $\bar{P}_{l+1}^{(3)}$ is calculated as \eqref{pk1}-\eqref{pk3} with $N$ replaced by $l+1$, respectively, and $\bar{P}_{l+1}^{(1)}+\bar{P}_{l+1}^{(2)}+\bar{P}_{l+1}^{(3)}=\bar{P}_{l+1}$, where $\bar{P}_{l+1}$ is given as \eqref{th42}.

By substituting \eqref{ldan} into \eqref{jnnn} and using the system dynamics \eqref{ps1}, $J(l)$ can be calculated as
\begin{align}\label{jk0}
&~~J(l)\notag\\
&=E(x_{l}'Q_{l}x_{l}+Ex_{l}'\bar{Q}_{l}Ex_{l}+u_{l}'R_{l}u_{l}+Eu_{l}'\bar{R}_{l}Eu_{l}\notag\\
&+x_{l+1}'P_{l+1}x_{l+1}+Ex_{l+1}'\bar{P}_{l+1}Ex_{l+1})\notag\\
&=E\Big\{x_{l}'\left(Q_{l}+A_{l}'P_{l+1}A_{l}+\sigma^{2}C_{l}'P_{l+1}C_{l}\right)x_{l}\notag\\
&+Ex_{l}'\Big[\bar{Q}_{l}+A_{l}'P_{l+1}\bar{A}_{l}+\sigma^{2}C_{l}'P_{l+1}\bar{C}_{l}+\bar{A}_{l}'P_{l+1}A_{l}\notag\\
&~~~~+\sigma^{2}\bar{C}_{l}'P_{l+1}C_{l}+\bar{A}_{l}'P_{l+1}\bar{A}_{l}+\sigma^{2}\bar{C}_{l}'P_{l+1}\bar{C}_{l}\notag\\
&~~~~+(A_{l}+\bar{A}_{l})'\bar{P}_{l+1}(A_{l}+\bar{A}_{l})\Big]Ex_{l}\notag\\
&+x_{l}'\left(A_{l}'P_{l+1}B_{l}+\sigma^{2}C_{l}'P_{l+1}D_{l}\right)u_{l}\notag\\
&+u_{l}'\left(B_{l}'P_{l+1}A_{l}+\sigma^{2}D_{l}'P_{l+1}C_{l}\right)x_{l}\notag\\
&+Ex_{l}'\Big[A_{l}'P_{l+1}\bar{B}_{l}+\sigma^{2}C_{l}'P_{l+1}\bar{D}_{l}+\bar{A}_{l}'P_{l+1}B_{l}\notag\\
&~~~~+\sigma^{2}\bar{C}_{l}'P_{l+1}D_{l}+\bar{A}_{l}'P_{l+1}\bar{B}_{l}+\sigma^{2}\bar{C}_{l}'P_{l+1}\bar{D}_{l}\notag\\
&~~~~+(A_{l}+\bar{A}_{l})'\bar{P}_{l+1}(B_{l}+\bar{B}_{l})\Big]Eu_{l}\notag\\
&+Eu_{l}'\Big[B_{l}'P_{l+1}\bar{A}_{l}+\sigma^{2}D_{l}'P_{l+1}\bar{C}_{l}+\bar{B}_{l}'P_{l+1}A_{l}\notag\\
&~~~~+\sigma^{2}\bar{D}_{l}'P_{l+1}C_{l}+\bar{B}_{l}'P_{l+1}\bar{A}_{l}+\sigma^{2}\bar{D}_{l}'P_{l+1}\bar{C}_{l}\notag\\
&~~~~+(B_{l}+\bar{B}_{l})'\bar{P}_{l+1}(A_{l}+\bar{A}_{l})\Big]Ex_{l}\notag\\
&+u_{l}'\left(R_{l}+B_{l}'P_{l+1}B_{l}+\sigma^{2}D_{l}'P_{l+1}D_{l}\right)u_{l}\notag\\
&+Eu_{l}'\Big[B_{l}'P_{l+1}\bar{B}_{l}+\sigma^{2}D_{l}'P_{l+1}\bar{D}_{l}+\bar{B}_{l}'P_{l+1}B_{l}\notag\\
&~~~~+\sigma^{2}\bar{D}_{l}'P_{l+1}D_{l}+\bar{B}_{l}'P_{l+1}\bar{B}_{l}+\sigma^{2}\bar{D}_{l}'P_{l+1}\bar{D}_{l}\notag\\
&~~~~+\bar{R}_{l}+(B_{l}+\bar{B}_{l})'\bar{P}_{l+1}(B_{l}+\bar{B}_{l})\Big]Eu_{l}\Big\}\notag\\
&=E(x_{l}'P_{l}x_{l}+Ex_{l}'\bar{P}_{l}Ex_{l})\notag\\
&+E\Big\{[u_{l}-Eu_{l}-K_{l}(x_{l}-Ex_{l})]'\Upsilon_{l}^{(1)}\notag\\
  &~~~~~\times[u_{l}-Eu_{l}-K_{l}(x_{l}-Ex_{l})]\Big\}\notag\\
&+[Eu_{l}\hspace{-1mm}-\hspace{-1mm}(K_{l}\hspace{-1mm}+\hspace{-1mm}\bar{K}_{l})Ex_{l}]'
\Upsilon_{l}^{(2)} [Eu_{l}\hspace{-1mm}-\hspace{-1mm}(K_{l}\hspace{-1mm}+\hspace{-1mm}\bar{K}_{l})Ex_{l}],
\end{align}
where $\Upsilon_{l}^{(1)}$ and $\Upsilon_{l}^{(2)}$ are respectively given by \eqref{upsi1} and \eqref{upsi2} for $k=l$.

Equation \eqref{barjk} indicates that $x_{l}$ is the initial state in minimizing $J(l)$. Now we show $\Upsilon_{l}^{(1)}>0$ and $\Upsilon_{l}^{(2)}>0$. We choose $x_{l}=0$, then \eqref{jk0} becomes
\begin{align}\label{zhang1}
J(l)\hspace{-1mm}=\hspace{-1mm}E\left\{(u_{l}\hspace{-1mm}-\hspace{-1mm}Eu_{l})'
\Upsilon_{l}^{(1)}(u_{l}\hspace{-1mm}-\hspace{-1mm}Eu_{l})\hspace{-1mm}+\hspace{-1mm}Eu_{l}'\Upsilon_{l}^{(2)}Eu_{l}\right\}.
\end{align}

It follows from Assumption \ref{ass1} that the minimum of $J(l)$ satisfies $J^{*}(l)\geq 0$. By \eqref{zhang1}, it is obvious that $u_{l}=0$ is the optimal controller and the associated optimal cost function $J^{*}(l)=0$.  The uniqueness of the optimal control implies that for any $u_{l}\neq 0$, $J(l)$ must be strictly positive.  Thus, following the discussion of \eqref{barjn1} for $J(N)$, we have $\Upsilon_{l}^{(1)}>0$ and $\Upsilon_{l}^{(2)}>0$.

Since $\Upsilon_{l}^{(1)}>0$ and $\Upsilon_{l}^{(2)}>0$,  the optimal controller can be given from \eqref{nc2}-\eqref{nc3} as \eqref{th43} for $k=l$, and the optimal cost function is given  as \eqref{jnst} for $k=l$.

Now we will show that \eqref{th4} associated with \eqref{th41}-\eqref{th42} are true for $k=l$. Since \eqref{th4} is assumed to be true for $k=l+1$, i.e., $\lambda_{l}$ is given by \eqref{ldan}. By substituting  \eqref{ldan} into \eqref{th32} for $k=l$, and applying the same lines for \eqref{nc8}-\eqref{barpn}, it is easy to verify that \eqref{th4} is true with $P_{l}$ satisfying \eqref{th41} and $\bar{P}_{l}^{(1)}$, $\bar{P}_{l}^{(2)}$, $\bar{P}_{l}^{(3)}$ given as \eqref{pk1}-\eqref{pk3} with $N$ replaced by $l$, furthermore $\bar{P}_{l}^{(1)}+\bar{P}_{l}^{(2)}+\bar{P}_{l}^{(3)}=\bar{P}_{l}$, and  $\bar{P}_{l}$ obeys \eqref{th42} for $k=l$.

Therefore, the proof of necessity is complete by using induction method.

``Sufficiency": Under Assumption \ref{ass1}, suppose $\Upsilon_{k}^{(1)},$ and $\Upsilon_{k}^{(2)}$, $k=0,\cdots,N$ are strictly positive definite, we will show that {\em Problem 1}  is uniquely solvable.

$V_{N}(k,x_{k})$ is denoted as
\begin{equation}\label{vnn}\begin{split}
  V_{N}(k,x_{k})\triangleq E(x_{k}'P_{k}x_{k})+Ex_{k}'\bar{P}_{k}Ex_{k},
\end{split}\end{equation}
where $P_k$ and $\bar{P}_k$ satisfy  \eqref{th41} and \eqref{th42} respectively. It follows that
\begin{align}\label{vn}
  &~~V_{N}(k,x_{k})-V_{N}(k+1,x_{k+1})\notag\\
  &=E\Big\{x_{k}'P_{k}x_{k}+Ex_{k}'\bar{P}_{k}Ex_{k}\notag\\
  &-x_{k}'\left(A_{k}'P_{k+1}A_{k}+\sigma^{2}C_{k}'P_{k+1}C_{k}\right)x_{k}\notag\\
  &-Ex_{k}'\Big[A_{k}'P_{k+1}\bar{A}_{k}+\sigma^{2}C_{k}'P_{k+1}\bar{C}_{k}\notag\\
 &~~~~+\bar{A}_{k}'P_{k+1}A_{k}+\sigma^{2}\bar{C}_{k}'P_{k+1}C_{k}\notag\\
  &~~~~+\bar{A}_{k}'P_{k+1}\bar{A}_{k}+\sigma^{2}\bar{C}_{k}'P_{k+1}\bar{C}_{k}\notag\\
 &~~~~+(A_{k}+\bar{A}_{k})'\bar{P}_{k+1}(A_{k}+\bar{A}_{k})\Big]Ex_{k}\notag\\
 &-x_{k}'\left(A_{k}'P_{k+1}B_{k}+\sigma^{2}C_{k}'P_{k+1}D_{k}\right)u_{k}\notag\\
 &-u_{k}'\left(B_{k}'P_{k+1}A_{k}+\sigma^{2}D_{k}'P_{k+1}C_{k}\right)x_{k}\notag\\
 &-Ex_{k}'\Big[A_{k}'P_{k+1}\bar{B}_{k}+\sigma^{2}C_{k}'P_{k+1}\bar{D}_{k}+\bar{A}_{k}'P_{k+1}B_{k}\notag\\
 &~~~~+\sigma^{2}\bar{C}_{k}'P_{k+1}D_{k}+\bar{A}_{k}'P_{k+1}\bar{B}_{k}+\sigma^{2}\bar{C}_{k}'P_{k+1}\bar{D}_{k}\notag\\
 &~~~~+(A_{k}+\bar{A}_{k})'\bar{P}_{k+1}(B_{k}+\bar{B}_{k})\Big]Eu_{k}\notag\\
  &-Eu_{k}'\Big[B_{k}'P_{k+1}\bar{A}_{k}+\sigma^{2}D_{k}'P_{k+1}\bar{C}_{k}+\bar{B}_{k}'P_{k+1}A_{k}\notag\\
 &~~~~+\sigma^{2}\bar{D}_{k}'P_{k+1}C_{k}+\bar{B}_{k}'P_{k+1}\bar{A}_{k}+\sigma^{2}\bar{D}_{k}'P_{k+1}\bar{C}_{k}\notag\\
 &~~~~+(B_{k}+\bar{B}_{k})'\bar{P}_{k+1}(A_{k}+\bar{A}_{k})\Big]Ex_{k}\notag\\
  &-u_{k}'\left(B_{k}'P_{k+1}B_{k}+\sigma^{2}D_{k}'P_{k+1}D_{k}\right)u_{k}\notag\\
 &-Eu_{k}'\Big[B_{k}'P_{k+1}\bar{B}_{k}+\sigma^{2}D_{k}'P_{k+1}\bar{D}_{k}+\bar{B}_{k}'P_{k+1}B_{k}\notag\\
 &~~~~+\sigma^{2}\bar{D}_{k}'P_{k+1}D_{k}+\bar{B}_{k}'P_{k+1}\bar{B}_{k}+\sigma^{2}\bar{D}_{k}'P_{k+1}\bar{D}_{k}\notag\\
 &~~~~+(B_{k}+\bar{B}_{k})'\bar{P}_{k+1}(B_{k}+\bar{B}_{k})\Big]Eu_{k}\Big\}\notag\\
 &=E\Big\{x_{k}'\{Q_{k}-[M_{k}^{(1)}]'[\Upsilon_{k}^{(1)}]^{-1}M_{k}^{(1)}\}x_{k}\notag\\
 &+Ex_{k}'\Big\{\bar{Q}_{k}+[M_{k}^{(1)}]'[\Upsilon_{k}^{(1)}]^{-1}M_{k}^{(1)}\notag\\
 &~~~~~-[M_{k}^{(2)}]'[\Upsilon_{k}^{(2)}]^{-1}M_{k}^{(2)}\Big\}Ex_{k}\notag\\
 &-x_{k}'[M_{k}^{(1)}]'u_{k}-u_{k}'M_{k}^{(1)}x_{k}+u_{k}'R_{k}u_{k}+Eu_{k}'\bar{R}_{k}Eu_{k}\notag\\
 &-Ex_{k}'[M_{k}^{(2)}-M_{k}^{(1)}]'Eu_{k}-Eu_{k}'[M_{k}^{(2)}-M_{k}^{(1)}]Ex_{k}\notag\\
 &-u_{k}'\Upsilon_{k}^{(1)}u_{k}-Eu_{k}'[\Upsilon_{k}^{(2)}-\Upsilon_{k}^{(1)}]Eu_{k}\Big\}\notag\\
  &=E\{x_{k}'Qx_{k}+Ex_{k}'\bar{Q}Ex_{k}+u_{k}'Ru_{k}+Eu_{k}'\bar{R}Eu_{k}\}\notag\\
  &-E\Big\{[u_{k}-Eu_{k}-K_{k}(x_{k}-Ex_{k})]'\Upsilon_{k}^{(1)}\notag\\
  &~~~~~\times[u_{k}-Eu_{k}-K_{k}(x_{k}-Ex_{k})]\Big\}\notag\\
  &-[Eu_{k}\hspace{-1.2mm}-\hspace{-1.2mm}(K_{k}\hspace{-1.2mm}+\hspace{-1mm}\bar{K}_{k})Ex_{k}]'\Upsilon_{k}^{(\hspace{-.3mm}2\hspace{-.3mm})}
  [Eu_{k}\hspace{-1.2mm}-\hspace{-1.2mm}(K_{k}\hspace{-1.2mm}+\hspace{-1.2mm}\bar{K}_{k})Ex_{k}],
\end{align}
where $K_{k}$ and $\bar{K}_{k}$ are respectively as in \eqref{kk} and \eqref{kkbar}.  Adding from $k=0$ to $k=N$ on both sides of
\eqref{vn}, the cost function \eqref{ps2} can be rewritten as
\begin{align}\label{jna}
    J_{N}
  &=\sum_{k=0}^{N}E\Big\{\Big[u_{k}-Eu_{k}-K_{k}(x_{k}-Ex_{k})\Big]'\Upsilon_{k}^{(1)}(N)\notag\\
  &~~~~~~~~~~~\times\Big[u_{k}-Eu_{k}-K_{k}(x_{k}-Ex_{k})\Big]\Big\}\notag\\
  &+\sum_{k=0}^{N}\left[Eu_{k}-(K_{k}+\bar{K}_{k})Ex_{k}\right]'\Upsilon_{k}^{(2)}(N)\notag\\
  &~~~~~~~~~~~\times \left[Eu_{k}-(K_{k}+\bar{K}_{k})Ex_{k}\right]\notag\\
  &+E(x_{0}'P_{0}x_{0})+Ex_{0}'\bar{P}_{0}Ex_{0}.
\end{align}
 Notice  $\Upsilon_{k}^{(1)}>0$ and $\Upsilon_{k}^{(2)}>0$, we have
 $$J_N\geq E(x_{0}'P_{0}x_{0})+Ex_{0}'\bar{P}_{0}Ex_{0},$$
  thus the minimum of $J_{N}$ is given by \eqref{jnst}, i.e.,
\begin{align*}
  J^{*}_{N}=E(x_{0}'P_{0}x_{0})+Ex_{0}'\bar{P}_{0}Ex_{0}.
\end{align*}
In this case the controller will satisfy that
\begin{align}
  u_{k}-Eu_{k}-K_{k}(x_{k}-Ex_{k})&=0,\label{bbb1} \\
  Eu_{k}-(K_{k}+\bar{K}_{k})Ex_{k}&=0.\label{bbb2}
\end{align}
Hence, the optimal controller can be uniquely obtained from \eqref{bbb1}-\eqref{bbb2} as \eqref{th43}.

In conclusion, \emph{Problem 1} admits a unique solution. The proof is complete.
\end{proof}

\section{Proof of Lemma \ref{111}}

\begin{proof}
Since $K_{k}(N)=-[\Upsilon_{k}^{(1)}(N)]^{-1}M_{k}^{(1)}(N)$, then it holds from \eqref{th41} that
\begin{equation*}\begin{split}
  &~~~[M_{k}^{(1)}(N)]'[\Upsilon_{k}^{(1)}(N)]^{-1}M_{k}^{(1)}(N)\\
  &=-[M_{k}^{(1)}(N)]'K_{k}(N)-K_{k}(N)'M_{k}^{(1)}(N)\\
  &~~~~-K_{k}(N)'\Upsilon_{k}^{(1)}(N)K_{k}(N).
\end{split}\end{equation*}

Thus, $P_{k}(N)$ in \eqref{th41} can be calculated as
\begin{align}\label{pnn1}
 &~~~P_{k}(N)\notag\\
 &=\hspace{-1mm}Q\hspace{-1mm}+\hspace{-1mm}A'P_{k+1}(N)A\hspace{-1mm}+\hspace{-1mm}\sigma^{2}C'P_{k+1}(N)C\hspace{-1mm}+
 \hspace{-1mm}[M_{k}^{(1)}(N)]'K_{k}(\hspace{-0.5mm}N\hspace{-0.5mm})\notag\\
 &+K_{k}'(N)M_{k}^{(1)}(N)+K_{k}'(N)\Upsilon_{k}^{(1)}(N)K_{k}(N)\notag\\
 &=Q+K_{k}'(N)RK_{k}(N)\notag\\
 &+[A+BK_{k}(N)]'P_{k+1}(N)[A+BK_{k}(N)]\notag\\
 &+\sigma^{2}[C+DK_{k}(N)]'P_{k+1}(N)[C+DK_{k}(N)].
\end{align}

Notice from Assumption \ref{ass1} that $Q\geq 0$ and$P_{N+1}(N)=P_{N+1}=0$, \eqref{pnn1} indicates that $P_{N}(N)\geq 0$. Using induction method, assume $P_{k}(N)\geq 0$ for $l+1\leq k\leq N$, by \eqref{pnn1}, immediately we can obtain $P_{l}(N)\geq 0$.

Therefore, for any $0\leq k\leq N$, $P_{k}(N)\geq 0$.

Moreover, using similar derivation with \eqref{pnn1}, from \eqref{upsi1}-\eqref{h2} we have that
\begin{align*}
  &~~~[M_{k}^{(2)}(N)]'[\Upsilon_{k}^{(2)}(N)]^{-1}M_{k}^{(2)}(N)\\
  &=\hspace{-1mm}-\hspace{-0.8mm}[M_{k}^{(2)}(N)]'[K_{k}(N)\hspace{-1mm}+\hspace{-1mm}\bar{K}_{k}(N)]\hspace{-1mm}-\hspace{-1mm}[K_{k}(N)\hspace{-1mm}+\hspace{-1mm}\bar{K}_{k}(N)]'M_{k}^{(2)}\\
  &~~~-[K_{k}(N)+\bar{K}_{k}(N)]'\Upsilon_{k}^{(2)}(N)[K_{k}(N)+\bar{K}_{k}(N)].
\end{align*}

Thus, $P_{k}(N)+\bar{P}_{k}(N)$ can be calculated as
\begin{align}\label{pnn2}
&~~~P_{k}(N)+\bar{P}_{k}(N)\notag\\
&=Q+\bar{Q}+(A+\bar{A})'[P_{k+1}(N)+\bar{P}_{k+1}(N)](A+\bar{A})\notag\\
&\hspace{-1mm}+\sigma^{2}(C+\bar{C})'P_{k+1}(N)(C+\bar{C})\notag\\
&\hspace{-1mm}-[M_{k}^{(2)}(N)]'\Upsilon_{k}^{(2)}(N)[M_{k}^{(2)}(N)]\notag\\
&\hspace{-1mm}=\hspace{-1mm}Q+\bar{Q}+[K_{k}(N)\hspace{-1mm}+\hspace{-1mm}\bar{K}_{k}(N)]'
(R\hspace{-1mm}+\hspace{-1mm}\bar{R})[K_{k}(N)\hspace{-1mm}+\hspace{-1mm}\bar{K}_{k}(N)]\notag\\
 &\hspace{-1mm}+\hspace{-1mm}\big\{A\hspace{-1mm}+\hspace{-1mm}\bar{A}\hspace{-1mm}
 +\hspace{-1mm}(B\hspace{-1mm}+\hspace{-1mm}\bar{B})[K_{k}(N)\hspace{-1mm}+\hspace{-1mm}\bar{K}_{k}(N)]\big\}'
 [P_{k+1}(N)\hspace{-1mm}+\hspace{-1mm}\bar{P}_{k+1}(\hspace{-0.5mm}N\hspace{-0.5mm})]\notag\\
 &\hspace{-1mm}\times\hspace{-1mm}\big\{A\hspace{-1mm}+\hspace{-1mm}\bar{A}\hspace{-1mm}
 +\hspace{-1mm}(B\hspace{-1mm}+\hspace{-1mm}\bar{B})[K_{k}(N)\hspace{-1mm}+\hspace{-1mm}\bar{K}_{k}(N)]\big\}\notag\\
 &\hspace{-1mm}+\sigma^{2}\{C+\bar{C}+(D+\bar{D})[K_{k}(N)+\bar{K}_{k}(N)]\}'P_{k+1}(N)\notag\\
 &\hspace{-1mm}\times \{C+\bar{C}+(D+\bar{D})[K_{k}(N)+\bar{K}_{k}(N)]\}.
\end{align}

Since $Q+\bar{Q}\geq 0$ as in Assumption \ref{ass1}, and $P_{N+1}=\bar{P}_{N+1}=0$, then $P_{N+1}(N)+\bar{P}_{N+1}(N)=P_{N+1}+\bar{P}_{N+1}= 0$. Furthermore, using induction method as above, we conclude that $P_{k}(N)+\bar{P}_{k}(N)\geq 0$ for any $0\leq k\leq N$. The proof is complete.
\end{proof}

\section{Proof of Lemma \ref{lemma2}}

\begin{proof}
If follows from Lemma \ref{111} that $P_{k}(N)\geq 0$ and $P_{k}(N)+\bar{P}_{k}(N)\geq 0$ for all $N\geq 0$. Via a time-shift, we can obtain $P_{k}(N)=P_{0}(N-k)$. Therefore, what we need to show is that there exists $\bar{N}_{0}>0$ such that $P_{0}(\bar{N}_{0})>0$ and $P_{0}(\bar{N}_{0})+\bar{P}_{0}(\bar{N}_{0})>0$.

Suppose this is not true, i.e., for arbitrary $N>0$, $P_{0}(N)$ and $P_{0}(N)+\bar{P}_{0}(N)$ are both strictly semi-definite positive. Now we construct two sets as follows,
\begin{align}
  X_{N}^{(1)}&\triangleq \Big\{x^{(1)}:x^{(1)}\neq0, E\{[x^{(1)}]'P_{0}(N)x^{(1)}\}=0,\notag\\
  &~~~~~~Ex^{(1)}=0\Big\},\label{zz1}\\
  X_{N}^{(2)}&\triangleq \Big\{x^{(2)}: x^{(2)}\neq  0,[x^{(2)}]'[P_{0}(N)+\bar{P}_{0}(N)]x^{(2)}=0,\notag\\
  &~~~~~~x^{(2)}=Ex^{(2)} ~\text{is deterministic}\Big\}.\label{zz2}
\end{align}
From Lemma \ref{lemma01} and Remark \ref{rem1}, we know that $X_{N}^{(1)}$ and $X_{N}^{(2)}$ are not empty.

Recall from Theorem \ref{main}, to minimize the cost function \eqref{ps2} with the weighting matrices, coefficient matrices being time-invariant and final condition $P_{N+1}(N)=\bar{P}_{N+1}(N)=0$, the optimal controller is given by \eqref{th43}, and the optimal cost function is presented as \eqref{jnst}, i.e.,
\begin{align}\label{opti}
&~~~J_{N}^{*}\notag\\
&=\min\{\sum_{k=0}^{N}E[x_{k}'Qx_{k}\hspace{-1mm}+\hspace{-1mm}Ex_{k}'\bar{Q}Ex_{k}\hspace{-1mm}+\hspace{-1mm}u_{k}'Ru_{k}\hspace{-1mm}+\hspace{-1mm}Eu_{k}'\bar{R}Eu_{k}]\}\notag\\
&=\sum_{k=0}^{N}E[x_{k}^{*'}Qx_{k}^{*}+Ex_{k}^{*'}\bar{Q}Ex_{k}^{*}+u_{k}^{*'}Ru_{k}^{*}+Eu_{k}^{*'}\bar{R}Eu_{k}^{*}]\notag\\
&=E[x_{0}'P_{0}(N)x_{0}]+(Ex_{0})'\bar{P}_{0}(N)(Ex_{0})\notag\\
&=E[(x_{0}-Ex_{0})'P_{0}(N)(x_{0}-Ex_{0})]\notag\\
&~~~+(Ex_{0})'[P_{0}(N)+\bar{P}_{0}(N)](Ex_{0}).
\end{align}
In the above equation, $x_{k}^{*}$ and $u_{k}^{*}$ represent the optimal state trajectory and the optimal controller, respectively.

Since $J_{N}\leq J_{N+1}$, then for any initial state $x_{0}$, we have $J_{N}^{*}\leq J_{N+1}^{*}$, it holds from \eqref{opti} that
\begin{align}\label{opti2}
&~~~E[(x_{0}-Ex_{0})'P_{0}(N)(x_{0}-Ex_{0})]\notag\\
&~~~+(Ex_{0})'[P_{0}(N)+\bar{P}_{0}(N)](Ex_{0})\notag\\
&\leq E[(x_{0}-Ex_{0})'P_{0}(N+1)(x_{0}-Ex_{0})]\notag\\
&~~~+(Ex_{0})'[P_{0}(N+1)+\bar{P}_{0}(N+1)](Ex_{0}).
\end{align}

For any initial state $x_{0}\neq 0$ with $Ex_{0}=0$, \eqref{opti2} can be reduced to
\begin{equation*}
  E[x_{0}'P_{0}(N)x_{0}]\leq E[x_{0}'P_{0}(N+1)x_{0}],
\end{equation*}
i.e., $E\{x_{0}'[P_{0}(N)-P_{0}(N+1)]x_{0}\}\leq 0$. By Lemma \ref{lemma01} and Remark \ref{rem1}, therefore we can obtain
\begin{align}\label{pi1}
P_{0}(N)\leq P_{0}(N+1),
\end{align}
 which implies that $P_{0}(N)$ increases with respect to $N$.

On the other hand, for arbitrary initial state $x_{0}\neq 0$ with $x_0=Ex_{0}$, i.e., $x_{0}\in \mathcal{R}^{n}$ is arbitrary deterministic, equation \eqref{opti2} indicates that
\begin{equation*}
 x_{0}'[P_{0}(N)+\bar{P}_{0}(N)]x_{0}\hspace{-1mm}\leq\hspace{-1mm} x_{0}'[P_{0}(N+1)+\bar{P}_{0}(N+1)]x_{0}.
\end{equation*}
 Note that $x_{0}$ is arbitrary, then using Remark \ref{rem1}, we have
\begin{align}\label{pi2}
P_{0}(N)+\bar{P}_{0}(N)\leq P_{0}(N+1)+\bar{P}_{0}(N+1),
\end{align}
which implies that $P_{0}(N)+\bar{P}_{0}(N)$ increases with respect to $N$, too. Furthermore, the monotonically increasing of $P_{0}(N)$ and
$P_{0}(N)+\bar{P}_{0}(N)$ indicates that

\begin{itemize}
\item  If $E\{[x^{(1)}]'P_{0}(N+1)x^{(1)}\}=0$ holds, then we can conclude $E\{[x^{(1)}]'P_{0}(N)x^{(1)}=0\}$;
\item  If $[x^{(2)}]'[P_{0}(N+1)+\bar{P}_{0}(N+1)]x^{(2)}=0$, then we can obtain $[x^{(2)}]'[P_{0}(N)+\bar{P}_{0}(N)]x^{(2)}=0$.
\end{itemize}
i.e., $X_{N+1}^{(1)} \subset X_{N}^{(1)}$ and $X_{N+1}^{(2)} \subset X_{N}^{(2)}$.

 As $\{X_{N}^{(1)}\}$ and $\{X_{N}^{(2)}\}$ are both non-empty finite dimensional sets, thus
\[1\leq \cdots \leq dim(X_{2}^{(1)}) \leq dim(X_{1}^{(1)}) \leq dim(X_{0}^{(1)})\leq n,\]
and
\[1\leq \cdots \leq dim(X_{2}^{(2)}) \leq dim(X_{1}^{(2)}) \leq dim(X_{0}^{(2)})\leq n.\]
where $dim$ means the dimension of the set.

Hence, there exists positive integer $N_{1}$, such that for any $N>N_{1}$, we can obtain
\begin{align*}
dim(X_{N}^{(1)})=dim(X_{N_{1}}^{(1)}),~dim(X_{N}^{(2)})=dim(X_{N_{1}}^{(2)}),
\end{align*}
which leads to $X_{N}^{(1)}=X_{N_{1}}^{(1)}$, and $X_{N}^{(2)}=X_{N_{1}}^{(2)}$, i.e.,
\begin{align*}
\bigcap_{N\geq 0}X_{N}^{(1)}=X_{N_{1}}^{(1)}\neq 0,~\bigcap_{N\geq 0}X_{N}^{(2)}=X_{N_{1}}^{(2)}\neq 0.
\end{align*}
Therefore, there exists nonzero $x^{(1)}\in X_{N_{1}}^{(1)}$ and $x^{(2)}\in X_{N_{1}}^{(2)}$ satisfying
\begin{align}
E\{[x^{(1)}]'P_{0}(N)x^{(1)}\}&=0,\label{x1}\\
[x^{(2)}]'[P_{0}(N)+\bar{P}_{0}(N)]x^{(2)}&=0.\label{x2}
\end{align}

1) Let the initial state of system \eqref{ps10} be $x_0=x^{(1)}$, where $x^{(1)}$ is as defined in \eqref{zz1}, then from  \eqref{opti} and using \eqref{x1}, the optimal value of the cost function can be calculated as
\begin{align}\label{jnstar}
 J_N^{*}&=\sum_{k=0}^{N}E[x_{k}^{*'}Qx_{k}^{*}\hspace{-1mm}+\hspace{-1mm}Ex_{k}^{*'}\bar{Q}Ex_{k}^{*}\hspace{-1mm}+\hspace{-1mm}u_{k}^{*'}Ru_{k}^{*}\hspace{-1mm}+\hspace{-1mm}Eu_{k}^{*'}\bar{R}Eu_{k}^{*}]\notag\\
  &=E\{[x^{(1)}]'P_{0}(N)x^{(1)}\}=0,
\end{align}
where $Ex^{(1)}=0$ has been used in the last equality. Notice  that $R>0$, $R+\bar{R}>0$, $Q\geq 0$ and $Q+\bar{Q}\geq 0$, from
\eqref{jnstar}, we obtain that
\[u_{k}^{*}=0,~Eu_{k}^{*}=0,~0\leq k\leq N,\]
and
 \begin{align*}
 0&=E[x_{k}^{*'}Qx_{k}^{*}+Ex_{k}^{*'}\bar{Q}Ex_{k}^{*}],~~0\leq k\leq N,\\
 &=E[(x_{k}^{*}-Ex_{k}^{*})'Q(x_{k}^{*}-Ex_{k}^{*})+Ex_{k}^{*'}(Q+\bar{Q})Ex_{k}^{*}],
 \end{align*}
i.e., $Q^{1/2}(x_{k}^{*}-Ex_{k}^{*})=0$, and $(Q+\bar{Q})^{1/2}Ex_{k}^{*}=0$.

By Assumption \ref{ass3}, $(A,\bar{A},C,\bar{C},\mathcal{Q}^{1/2})$ is exactly observation, i.e., $\left[
  \begin{array}{cc}
   Q& 0\\
   0        & Q+\bar{Q}      \\
  \end{array}
\right]^{1/2}\left[
  \begin{array}{cc}
   \hspace{-1mm} x_{k}-Ex_{k}\hspace{-1mm}\\
   \hspace{-1mm} Ex_{k}     \hspace{-1mm}           \\
  \end{array}
\right]=0\Rightarrow x_{0}=0$, then we have $x^{(1)}=x_0=Ex_0=0$, which is a contradiction with $x^{(1)}\neq 0$.

Thus, there exists $\bar{N}_{0}>0$, such that $P_{0}(\bar{N}_{0})>0$.

2) Let the initial state of system \eqref{ps10} be $x_{0}=x^{(2)}$, where $x^{(2)}$ is given by \eqref{zz2}, then by using \eqref{opti} and \eqref{x2}, the minimum of cost function can be rewritten as
 \begin{equation*}\begin{split}
 J_{N}^{*}&=\sum_{k=0}^{N}E[x_{k}^{*'}Qx_{k}^{*}\hspace{-1mm}+\hspace{-1mm}Ex_{k}^{*'}\bar{Q}Ex_{k}^{*}\hspace{-1mm}+\hspace{-1mm}
 u_{k}^{*'}Ru_{k}^{*}\hspace{-1mm}+\hspace{-1mm}Eu_{k}^{*'}\bar{R}Eu_{k}^{*}]\\
  &=[x^{(2)}]'[P_{0}(N)+\bar{P}_{0}(N)]x^{(2)}=0.
\end{split}\end{equation*}

Using similar method with that in 1), by Assumption \ref{ass3}, we can conclude that $x^{(2)}=x_{0}=Ex_{0}=0$, which is a contradiction with $x^{(2)}\neq 0$.

In conclusion, there exists $\bar{N}_{0}>0$ such that $P_{0}(\bar{N}_{0})>0$ and $P_{0}(\bar{N}_{0})+\bar{P}_{0}(\bar{N}_{0})>0$. Via a time-shift, hence we have, for any $k\geq 0$, there exists a positive integer $N_{0}\geq 0$ such that $P_{k}(N_{0})>0$ and
$P_{k}(N_{0})+\bar{P}_{k}(N_{0})>0$.

The proof is complete.
\end{proof}
\section{Proof of Theorem \ref{theorem2}}
\begin{proof}
1) Firstly, from the proof of Lemma \ref{lemma2}, we know that $P_{0}(N)$ and $P_{0}(N)+\bar{P}_{0}(N)$ are monotonically increasing, i.e., for any $N>0$,
\begin{align*}
P_{0}(N)&\leq P_{0}(N+1),\\
P_{0}(N)+\bar{P}_{0}(N)&\leq P_{0}(N+1)+\bar{P}_{0}(N+1).
\end{align*}
Next we will show that $P_{0}(N)$ and $P_{0}(N)+\bar{P}_{0}(N)$ are bounded. Since system \eqref{ps10} is stabilizable in the mean square sense, there exists $u_{k}$ has the form
\begin{equation}\label{uk}
  u_{k}=Lx_{k}+\bar{L}Ex_{k},
\end{equation}
with constant matrices $L$ and $\bar{L}$ such that the closed-loop system \eqref{ps10} satisfies
\begin{equation}\label{asy}
  \lim_{k\rightarrow +\infty}E(x_{k}'x_{k})=0.
\end{equation}

As $(Ex_{k})'Ex_{k}+E(x_{k}-Ex_{k})'(x_{k}-Ex_{k})=E(x_{k}'x_{k})$, thus, equation \eqref{asy} implies $\lim_{k\rightarrow +\infty}(Ex_{k})'Ex_{k}=0$.

Substituting \eqref{uk} into \eqref{ps10}, we can obtain
\begin{align}
  &x_{k+1}=[(A+w_{k}C)+(B+w_{k}D)L]x_{k}\label{xkexk}\\
  &~~\hspace{-1mm}+\hspace{-1mm}[(B\hspace{-1mm}+\hspace{-1mm}w_{k}D)\bar{L}\hspace{-1mm}+\hspace{-1mm}
  (\bar{A}\hspace{-1mm}+\hspace{-1mm}w_{k}\bar{C})\hspace{-1mm}+\hspace{-1mm}(\bar{B}\hspace{-1mm}+\hspace{-1mm}w_{k}\bar{D})
  (L\hspace{-1mm}+\hspace{-1mm}\bar{L})]Ex_{k},\notag\\
 & Ex_{k+1}=[(A+\bar{A})+(B+\bar{B})(L+\bar{L})]Ex_{k}.\label{exk}
\end{align}

Denote $
X_{k}\triangleq \left[
  \begin{array}{cc}
     x_{k} \\
  Ex_{k}   \\
  \end{array}
 \right],$ and $
 \mathcal{X}_{k}\triangleq E[X_{k}X_{k}'].
 $

Following from \eqref{xkexk} and \eqref{exk}, it holds
\begin{align}\label{xk1}
X_{k+1}=\mathcal{A}X_{k},
\end{align}
where $\mathcal{A}=\left[
  \begin{array}{cc}
     A_{11} &  A_{12} \\
    0          &A_{22}      \\
  \end{array}
 \right]$, $A_{11}=(A+w_{k}C)+(B+w_{k}D)L$,
$A_{12}=(B\hspace{-1mm}+\hspace{-1mm}w_{k}D)\bar{L}\hspace{-1mm}+\hspace{-1mm}
  (\bar{A}\hspace{-1mm}+\hspace{-1mm}w_{k}\bar{C})\hspace{-1mm}+\hspace{-1mm}(\bar{B}\hspace{-1mm}+\hspace{-1mm}w_{k}\bar{D})
  (L\hspace{-1mm}+\hspace{-1mm}\bar{L}),$ and
$A_{22}=(A+\bar{A})+(B+\bar{B})(L+\bar{L}).$

The mean square stabilization of $\lim_{k\rightarrow+\infty}E(x_{k}'x_{k})\hspace{-1mm}=\hspace{-1mm}0$ implies $\lim_{k\rightarrow+\infty}\mathcal{X}_{k}\hspace{-1mm}=\hspace{-1mm}0$, thus, it follows from \cite{rami} that
\begin{equation*}
  \sum_{k=0}^{\infty}E(x_{k}'x_{k})< +\infty,~\text{and}~ \sum_{k=0}^{\infty}(Ex_{k})'(Ex_{k})< +\infty.
\end{equation*}

Therefore, there exists constant $c$ such that
\begin{align}\label{cnc}
\sum_{k=0}^{\infty}E(x_{k}'x_{k})\leq cE(x_{0}'x_{0}).
\end{align}

Since $Q\geq 0$, $Q+\bar{Q}\geq 0$, $R>0$ and $R+\bar{R}>0$, thus there exists constant $\lambda$ such that $\left[\hspace{-1mm}
  \begin{array}{cc}
     Q &  0 \\
    0         &  Q+\bar{Q}      \\
  \end{array}
\hspace{-1mm}\right]\leq \lambda I$ and $\left[\hspace{-1mm}
  \begin{array}{cc}
     L'RL &  0 \\
    0          & (L+\bar{L})'(R+\bar{R})(L+\bar{L})     \\
  \end{array}
\hspace{-1mm}\right]\leq \lambda I$, using \eqref{uk} and \eqref{cnc}, we obtain that
\begin{align}\label{jjj}
  J&=\sum_{k=0}^{\infty}E[x_{k}'Qx_{k}+u_{k}'Ru_{k}+Ex_{k}'\bar{Q}Ex_{k}+Eu_{k}'\bar{R}Eu_{k}]\notag\\
  &=\sum_{k=0}^{\infty}E\Big\{x_{k}'(Q+L'RL)x_{k}+Ex_{k}'\big[\bar{Q}+L'R\bar{L}+\bar{L}'RL\notag\\
  &+\bar{L}'R\bar{L}+(L+\bar{L})'\bar{R}(L+\bar{L})\big]Ex_{k}\Big\}\notag\\
  &=\sum_{k=0}^{\infty}E\Big\{ \left[\hspace{-1mm}
  \begin{array}{cc}
     x_{k}\hspace{-1mm}-\hspace{-1mm}Ex_{k} \\
    Ex_{k}             \\
  \end{array}
\hspace{-1mm}\right]'\left[\hspace{-1mm}
  \begin{array}{cc}
     Q &  0 \\
    0        &  Q+\bar{Q}       \\
  \end{array}
\hspace{-1mm}\right]\left[\hspace{-1mm}
  \begin{array}{cc}
     x_{k}\hspace{-1mm}-\hspace{-1mm}Ex_{k} \\
    Ex_{k}             \\
  \end{array}
\hspace{-1mm}\right]\Big\}\notag\\
&+\sum_{k=0}^{\infty}E\Big\{\left[\hspace{-1mm}
  \begin{array}{cc}
    x_{k}\hspace{-1mm}-\hspace{-1mm}Ex_{k} \\
    Ex_{k}             \\
  \end{array}
\hspace{-1mm}\right]'\left[
  \begin{array}{cc}
     L'RL &  0 \\
    0          & (L\hspace{-1mm}+\hspace{-1mm}\bar{L})'(R\hspace{-1mm}+\hspace{-1mm}\bar{R})(L\hspace{-1mm}+\hspace{-1mm}\bar{L})       \\
  \end{array}
\hspace{-1mm}\right]\notag\\
&~~~~~~~\times\left[\hspace{-1mm}
  \begin{array}{cc}
     x_{k}\hspace{-1mm}-\hspace{-1mm}Ex_{k} \\
    Ex_{k}            \\
  \end{array}
\hspace{-1mm}\right]
\Big\}\notag\\
&\leq 2\lambda\sum_{k=0}^{\infty}E[(Ex_{k})'Ex_{k}+(x_{k}-Ex_{k})'(x_{k}-Ex_{k})]\notag\\
&=2\lambda\sum_{k=0}^{\infty}E(x_{k}'x_{k})\leq 2\lambda cE(x_{0}'x_{0}).
\end{align}
On the other hand, by \eqref{jnst}, notice the fact that
$$ E[x_{0}'P_{0}(N)x_{0}]+(Ex_{0})'\bar{P}_{0}(N)(Ex_{0})=J_{N}^{*}\leq J, $$
thus, \eqref{jjj} yields
\begin{equation}\label{x0}
 E[x_{0}'P_{0}(\hspace{-0.5mm}N\hspace{-0.5mm})x_{0}]\hspace{-1mm}+\hspace{-1mm}(Ex_{0})'\bar{P}_{0}(\hspace{-0.5mm}N\hspace{-0.5mm})(Ex_{0})\hspace{-1mm} \leq \hspace{-1mm} 2\lambda cE(x_{0}'x_{0}).
\end{equation}
Now we let the state initial value be random vector with zero mean, i.e., $Ex_{0}=0$, it follows from \eqref{x0} that
\begin{equation*}
  E[x_{0}'P_{0}(N)x_{0}]\leq 2\lambda cE(x_{0}'x_{0}).
\end{equation*}
Since $x_0$ is arbitrary with $Ex_0=0$,  by Lemma \ref{lemma01} and Remark \ref{rem1}, we have
\begin{align*}
  P_{0}(N)\leq 2\lambda cI.
\end{align*}

Similarly, let the state initial value be arbitrary deterministic i.e., $x_{0}=Ex_{0}$, \eqref{x0} yields that
\begin{equation*}\begin{split}
 x_{0}'[P_{0}(N)+\bar{P}_{0}(N)]x_{0}=J_{N}^{*}\leq J\leq 2\lambda cx_{0}'x_{0},
\end{split}\end{equation*}
which implies
\begin{equation*}
  P_{0}(N)+\bar{P}_{0}(N)\leq 2\lambda cI.
\end{equation*}

Therefore, both $P_{0}(N)$ and $P_{0}(N)+\bar{P}_{0}(N)$ are bounded. Recall that $P_{0}(N)$ and $P_{0}(N)+\bar{P}_{0}(N)$ are monotonically increasing, we conclude that $P_{0}(N)$ and $P_{0}(N)+\bar{P}_{0}(N)$ are convergent, i.e., there exists $P$ and $\bar{P}$ such that
\begin{equation*}
  \lim_{N\rightarrow +\infty}P_{k}(N)=\lim_{N\rightarrow +\infty}P_{0}(N-k)=P,
\end{equation*}
\begin{equation*}
  \lim_{N\rightarrow +\infty}\bar{P}_{k}(N)=\lim_{N\rightarrow +\infty}\bar{P}_{0}(N-k)=\bar{P}.
\end{equation*}

Furthermore, in view of \eqref{upsi1}-\eqref{h2}, we know that $\Upsilon_{k}^{(1)}(N)$, $M_{k}^{(1)}(N)$, $\Upsilon_{k}^{(2)}(N)$ and $M_{k}^{(2)}(N)$ are convergent, i.e.,
\begin{align}
\lim_{N\rightarrow+ \infty}\Upsilon_{k}^{(1)}(N)&=\Upsilon^{(1)}\geq R>0,\label{u1}\\
\lim_{N\rightarrow +\infty}M_{k}^{(1)}(N)&=M^{(1)},\label{hhh1}\\
\lim_{N\rightarrow +\infty}\Upsilon_{k}^{(2)}(N)&=\Upsilon^{(2)}\geq R+\bar{R}>0,\label{u2}\\
\lim_{N\rightarrow +\infty}M_{k}^{(2)}(N)&=M^{(2)}.\label{hhh2}
\end{align}
where $\Upsilon^{(1)},M^{(1)},\Upsilon^{(2)},M^{(2)}$ are given by \eqref{up1}-\eqref{hh2}. Taking limitation on both sides of \eqref{th41} and \eqref{th42}, we know that $P$ and $\bar{P}$ satisfy the coupled ARE \eqref{are1}-\eqref{are2}.

2) From Lemma \ref{lemma2}, for any $k\geq 0$, there exists $N_{0}>0$ such that, $P_{k}(N_{0})>0$ and
$P_{k}(N_{0})+\bar{P}_{k}(N_{0})>0$, hence we have
\begin{align*}
  &P=\lim_{N\rightarrow +\infty}P_{k}(N)\geq P_{k}(N_{0})>0,\\
  &P\hspace{-1mm}+\hspace{-1mm}\bar{P}\hspace{-1mm}=\hspace{-1mm}\lim_{N\rightarrow +\infty}[P_{k}(N)\hspace{-1mm}+\hspace{-1mm}\bar{P}_{k}(N)]\geq P_{k}(N_{0})\hspace{-1mm}+\hspace{-1mm}\bar{P}_{k}(N_{0})>0.
\end{align*}
This ends the proof.
\end{proof}

\section{Proof of Theorem \ref{succeed}}

\begin{proof}
\emph{``Sufficiency"}: Under Assumptions \ref{ass2} and \ref{ass3}, we suppose that $P$ and $\bar{P}$ are the solution of \eqref{are1}-\eqref{are2} satisfying $P>0$ and $P+\bar{P}>0$, we will show \eqref{control} stabilizes \eqref{ps10} in mean square sense.

Similar to \eqref{vnn}, we define the Lyapunov function candidate $V(k,x_{k})$ as
\begin{align}\label{lya}
  V(k,x_{k})&\triangleq E(x_{k}'Px_{k})+Ex_{k}'\bar{P}Ex_{k}.
\end{align}
 Apparently we have
\begin{align}\label{vnk}
V(k,x_{k})&=E[(x_{k}\hspace{-1mm}-\hspace{-1mm}Ex_{k})'P(x_{k}\hspace{-1mm}-\hspace{-1mm}Ex_{k})\hspace{-1mm}+
\hspace{-1mm}Ex_{k}'(P\hspace{-1mm}+\hspace{-1mm}\bar{P})Ex_{k}]\notag\\
&\geq 0.
\end{align}

 We claim that $V(k,x_{k})$ monotonically decreases. Actually, following the derivation of \eqref{vn}, we have
\begin{align}\label{lya1}
   &~~V(k,x_{k})-V(k+1,x_{k+1})\notag\\
  &=E[x_{k}'Qx_{k}+Ex_{k}'\bar{Q}Ex_{k}+u_{k}'Ru_{k}+Eu_{k}'\bar{R}Eu_{k}]\notag\\
  &~-E\{[u_{k}-Eu_{k}-K(x_{k}-Ex_{k})]'\Upsilon^{(1)}\notag\\
  &~~~~~~~~~~~~~~\times[u_{k}-Eu_{k}-K(x_{k}-Ex_{k})]\}\notag\\
  &-[Eu_{k}-(K+\bar{K})Ex_{k}]'\Upsilon^{(2)} [Eu_{k}-(K+\bar{K})Ex_{k}]\notag\\
  &=E[x_{k}'Qx_{k}+Ex_{k}'\bar{Q}Ex_{k}+u_{k}'Ru_{k}+Eu_{k}'\bar{R}Eu_{k}]\notag\\
  &\geq 0,~~k\geq 0,
\end{align}
where $u_{k}=Kx_{k}+\bar{K}Ex_{k}$ is used in the last identity. The last inequality implies that $V(k,x_{k})$ decreases with respect to $k$, also from \eqref{vnk} we know that $V(k,x_{k})\geq 0$, thus $V(k,x_{k})$ is convergent.

Let $l$ be any positive integer, by adding from $k=l$ to $k=l+N$ on both sides of \eqref{lya1}, we obtain that
\begin{align}\label{lmi1}
  &\sum_{k=l}^{l+N}E[x_{k}'Qx_{k}\hspace{-1mm}+\hspace{-1mm}Ex_{k}'\bar{Q}Ex_{k}\hspace{-1mm}+\hspace{-1mm}u_{k}'Ru_{k}\hspace{-1mm}+\hspace{-1mm}
  Eu_{k}'\bar{R}Eu_{k}]\notag\\
  &=[V(l,x_{l})-V(l+N+1,x_{l+N+1})].
\end{align}

Since $V(k,x_{k})$ is convergent, then by taking limitation of $l$ on both sides of \eqref{lmi1}, it holds
\begin{align}\label{lmi}
  &\lim_{l\rightarrow+\infty}\sum_{k=l}^{l+N}E[x_{k}'Qx_{k}\hspace{-1mm}+\hspace{-1mm}Ex_{k}'\bar{Q}Ex_{k}\hspace{-1mm}+\hspace{-1mm}u_{k}'Ru_{k}\hspace{-1mm}+\hspace{-1mm}
  Eu_{k}'\bar{R}Eu_{k}]\notag\\
  &=\lim_{l\rightarrow+\infty}[V(l,x_{l})-V(l+N+1,x_{l+N+1})]=0.
\end{align}

Recall from \eqref{opti} that
\begin{align}
  &~~~~J_{N}=\sum_{k=0}^{N}E[x_{k}'Qx_{k}\hspace{-1mm}+\hspace{-1mm}Ex_{k}'\bar{Q}Ex_{k}\hspace{-1mm}+\hspace{-1mm}u_{k}'Ru_{k}\hspace{-1mm}+\hspace{-1mm}
  Eu_{k}'\bar{R}Eu_{k}]\notag\\
  &\geq J_{N}^{*}=E[x_{0}'P_{0}(N)x_{0}]+Ex_{0}'\bar{P}_{0}(N)Ex_{0}.\label{jn}
\end{align}

Thus, taking limitation on both sides of \eqref{jn}, via a time-shift of $l$ and using \eqref{lmi}, it yields that
\begin{align}\label{ly1}
  &0\hspace{-1mm}=\hspace{-1mm}\lim_{l\rightarrow+\infty}\sum_{k=l}^{l+N}E[x_{k}'Qx_{k}\hspace{-1mm}+\hspace{-1mm}Ex_{k}'\bar{Q}Ex_{k}\hspace{-1mm}+\hspace{-1mm}u_{k}'Ru_{k}\hspace{-1mm}+\hspace{-1mm}
  Eu_{k}'\bar{R}Eu_{k}]\notag\\
  &\geq\lim_{l\rightarrow+\infty} E\left[x_{l}'P_{l}(l+N)x_{l}+Ex_{l}'\bar{P}_{l}(l+N)Ex_{l}\right]\notag\\
&=\lim_{l\rightarrow+\infty} E\Big\{(x_{l}-Ex_{l})'P_{l}(l+N)(x_{l}-Ex_{l})\notag\\
&~~~~~+Ex_{l}'[P_{l}(l+N)+\bar{P}_{l}(l+N)]Ex_{l}\Big\}\notag\\
&=\lim_{l\rightarrow+\infty} E\Big\{(x_{l}-Ex_{l})'P_{0}(N)(x_{l}-Ex_{l})\notag\\
&~~~~~+Ex_{l}'[P_{0}(N)+\bar{P}_{0}(N)]Ex_{l}\Big\}\geq 0.
\end{align}

Hence, it follows from \eqref{ly1} that
\begin{align}
\lim_{l\rightarrow+\infty}E[(x_{l}-Ex_{l})'P_{0}(N)(x_{l}-Ex_{l})]&=0,\label{conc1}\\
\lim_{l\rightarrow+\infty}Ex_{l}'[P_{0}(N)+\bar{P}_{0}(N)]Ex_{l}&=0.\label{conc2}
\end{align}

By Lemma \ref{lemma2}, we know that there exists $N_{0}\geq 0$ such that $P_{0}(N)>0$ and $P_{0}(N)+\bar{P}_{0}(N)>0$ for any $N>N_{0}$.
Thus from \eqref{conc1} and \eqref{conc2}, we have
\begin{align}\label{geq}
\lim_{l\rightarrow+\infty}\hspace{-2mm}E[(x_{l}\hspace{-1mm}-\hspace{-1mm}Ex_{l})'(x_{l}\hspace{-1mm}-\hspace{-1mm}Ex_{l})]\hspace{-1mm}=\hspace{-1mm}0,
\lim_{l\rightarrow+\infty}\hspace{-2mm}Ex_{l}'Ex_{l}\hspace{-1mm}=\hspace{-1mm}0,
\end{align}
which indicates that $\lim_{l\rightarrow+\infty}E(x_{l}'x_{l})=0.$

In conclusion, \eqref{control} stabilizes \eqref{ps10} in the mean square sense.

Next we will show that controller \eqref{control} minimizes the cost function \eqref{ps200}. For \eqref{lya1}, adding from $k=0$ to $k=N$, we have
\begin{align}\label{kon}
  &\sum_{k=0}^{N}E[x_{k}'Qx_{k}\hspace{-1mm}+\hspace{-1mm}Ex_{k}'\bar{Q}Ex_{k}\hspace{-1mm}+\hspace{-1mm}u_{k}'Ru_{k}\hspace{-1mm}+\hspace{-1mm}
  Eu_{k}'\bar{R}Eu_{k}]\notag\\
  &=V(0,x_{0})-V(N+1,x_{N+1})\notag\\
  &+\sum_{k=0}^{N}E\Big\{[u_{k}-Eu_{k}-K(x_{k}-Ex_{k})]'\Upsilon^{(1)}\notag\\
  &~~~~~~~~~~~~~~\times[u_{k}-Eu_{k}-K(x_{k}-Ex_{k})]\Big\}\notag\\
  &\hspace{-1mm}+\hspace{-1mm}\sum_{k=0}^{N}\hspace{-0.3mm}[Eu_{k}\hspace{-1mm}-\hspace{-1mm}(K\hspace{-1mm}+\hspace{-1mm}\bar{K})Ex_{k}]'\Upsilon^{(\hspace{-0.3mm}2\hspace{-0.3mm})}
  [Eu_{k}\hspace{-1mm}-\hspace{-1mm}(K\hspace{-1mm}+\hspace{-1mm}\bar{K})Ex_{k}].
\end{align}

Moreover, following from \eqref{lya} and \eqref{geq}, we have that
\begin{align*}
  0&\hspace{-1mm}\leq \hspace{-1mm}\lim_{k\rightarrow+\infty} V(k,x_{k})\hspace{-1mm}=\hspace{-1mm}\lim_{k\rightarrow+\infty} E\{x_{k}'Px_{k}+Ex_{k}'\bar{P}Ex_{k}\}=0.
\end{align*}

Thus, taking limitation of $N\rightarrow+\infty$ on both sides of \eqref{kon} and noting \eqref{ps200}, we have
\begin{align}\label{kon1}
  &J=\sum_{k=0}^{\infty}E[x_{k}'Qx_{k}\hspace{-1mm}+\hspace{-1mm}Ex_{k}'\bar{Q}Ex_{k}\hspace{-1mm}+\hspace{-1mm}u_{k}'Ru_{k}\hspace{-1mm}+\hspace{-1mm}
  Eu_{k}'\bar{R}Eu_{k}]\notag\\
  &=E[(x_{0}-Ex_{0})'P(x_{0}-Ex_{0})]+Ex_{0}'(P+\bar{P})Ex_{0}\notag\\
  &+\sum_{k=0}^{\infty}E\Big\{[u_{k}-Eu_{k}-K(x_{k}-Ex_{k})]'\Upsilon^{(1)}\notag\\
  &~~~~~~~~~~~~~~\times[u_{k}-Eu_{k}-K(x_{k}-Ex_{k})]\Big\}\notag\\
  &\hspace{-1mm}+\hspace{-1mm}\sum_{k=0}^{\infty}[Eu_{k}\hspace{-1mm}-\hspace{-1mm}(K\hspace{-1mm}+\hspace{-1mm}\bar{K})Ex_{k}]'
  \Upsilon^{(\hspace{-0.4mm}2\hspace{-0.4mm})}
  [Eu_{k}\hspace{-1mm}-\hspace{-1mm}(K\hspace{-1mm}+\hspace{-1mm}\bar{K})Ex_{k}].
\end{align}

Note that $\Upsilon^{(1)}>0$ and $\Upsilon^{(2)}>0$, following the discussion in the sufficiency proof of Theorem \ref{main}, thus, the cost function \eqref{ps200} can be minimized by controller \eqref{control}. Furthermore, directly from \eqref{kon1}, the optimal cost function can be given as \eqref{cost}.

\emph{``Necessity"}: Under Assumptions \ref{ass2} and \ref{ass3}, if \eqref{ps10} is stablizable in mean square sense, we will show that the coupled ARE \eqref{are1}-\eqref{are2} has unique solution $P$ and $P+\bar{P}$ satisfying $P>0$ and $P+\bar{P}>0$. The existence of the solution to \eqref{are1}-\eqref{are2} satisfying $P>0$ and $P+\bar{P}>0$ has been verified in Theorem \ref{theorem2}. The uniqueness of the solution remains to be shown.

 Let $S$ and $\bar{S}$ be another solution of \eqref{are1}-\eqref{are2} satisfying $S>0$ and $S+\bar{S}>0$, i.e.,
\begin{align}
  S&=Q+A'SA+\sigma^{2}C'SC-\hspace{-1mm}[T^{(1)}]'[\Delta^{(1)}]^{-1}T^{(1)},\label{are3}\\
  \bar{S}&=\bar{Q}+A'S\bar{A}+\sigma^{2}C'S\bar{C}+\bar{A}'SA+\sigma^{2}\bar{C}'SC\notag\\
  &~~+\bar{A}'S\bar{A}+\sigma^{2}\bar{C}'S\bar{C}+(A+\bar{A})'\bar{S}(A+\bar{A})\notag\\
  &~~+[T^{(1)}]'[\Delta^{(1)}]^{-1}T^{(1)}-[T^{(2)}]'[\Delta^{(2)}]^{-1}T^{(2)},\label{are4}
\end{align}
where
\begin{align*}
\Delta^{(1)}&=R+B'SB+\sigma^{2}D'SD,\\
T^{(1)}&=B'SA+\sigma^{2}D'SC,\\
\Delta^{(2)}&=R+\bar{R}+(B+\bar{B})'(S+\bar{S})(B+\bar{B})\\
&~~~~~+\sigma^{2}(D+\bar{D})'S(D+\bar{D}),\\
T^{(2)}&=(B+\bar{B})'(S+\bar{S})(A+\bar{A})\\
&~~~~~+\sigma^{2}(D+\bar{D})'S(C+\bar{C}).
\end{align*}

Notice that the optimal cost function has been proved to be \eqref{cost}, i.e.,
\begin{align}\label{cost2}
  J^{*}&=E(x_{0}'Px_{0})+Ex_{0}'\bar{P}Ex_{0}\notag\\
  &=E(x_{0}'Sx_{0})+Ex_{0}'\bar{S}Ex_{0}.
\end{align}

For any initial state $x_{0}$ satisfying $x_{0}\neq 0$ and $Ex_{0}=0$, equation \eqref{cost2} implies that
\begin{equation*}
  E[x_{0}'(P-S)x_{0}]=0,
\end{equation*}
By Lemma \ref{lemma01} and Remark \ref{rem1}, we can conclude that $P=S$.

Moreover, if $x_{0}=Ex_{0}$ is arbitrary deterministic initial state, it follows from \eqref{cost2} that
\begin{equation*}
  x_{0}'(P+\bar{P}-S-\bar{S})x_{0}=0,
\end{equation*}
which indicates $P+\bar{P}=S+\bar{S}$.

Hence we have $S=P$ and $\bar{S}=\bar{P}$, i.e., the uniqueness has been proven. The proof is complete.
\end{proof}

\section{Proof of Theorem \ref{succeed2}}

\begin{proof}
``Necessity:" Under Assumption \ref{ass2} and \ref{ass4}, suppose mean-field system \eqref{ps10} is stabilizable in mean square sense, we will show that the coupled ARE \eqref{are1}-\eqref{are2} has a unique solution $P$ and $\bar{P}$ with $P\geq 0$ and $P+\bar{P}\geq 0$.

Actually, from \eqref{opti}-\eqref{pi2} in the proof of Lemma \ref{lemma2}, we know that $P_{0}(N)$ and $P_{0}(N)+\bar{P}_{0}(N)$ are monotonically increasing, then following the lines of \eqref{uk}-\eqref{x0}, the boundedness of $P_{0}(N)$ and $P_{0}(N)+\bar{P}_{0}(N)$ can be obtained. Hence, $P_{0}(N)$ and $P_{0}(N)+\bar{P}_{0}(N)$ are convergent. Then there exists $P$ and $\bar{P}$ such that
\begin{align*}
  \lim_{N\rightarrow+\infty}P_{k}(N)=\lim_{N\rightarrow+\infty}P_{0}(N-k)&=P,\\
  \lim_{N\rightarrow+\infty}\bar{P}_{k}(N)=\lim_{N\rightarrow+\infty}\bar{P}_{0}(N-k)&=\bar{P}.
\end{align*}
From Lemma \ref{111}, we know that $P_{k}(N)\geq 0$ and $P_{k}(N)+\bar{P}_{k}(N)\geq 0$, thus we have $P\geq 0$ and $P+\bar{P}\geq 0$. Furthermore, in view of \eqref{upsi1}-\eqref{h2}, $\Upsilon^{(1)},\Upsilon^{(2)},M^{(1)},M^{(2)}$ in \eqref{up1}-\eqref{hh2} can be obtained. Taking limitation on both sides of \eqref{th41} and \eqref{th42}, we know that $P$ and $\bar{P}$ satisfy the coupled ARE \eqref{are1} and \eqref{are2}. Under Assumption \ref{ass2}, Lemma \ref{lemma3} yields that \emph{Problem 1} has a unique solution, then following the steps of \eqref{are3}-\eqref{cost2} in Theorem \ref{succeed}, the uniqueness of $P$ and $\bar{P}$ can be obtained. Finally, taking limitation on both sides of \eqref{th43} and \eqref{jnst}, the unique optimal controller can be given as \eqref{control}, and optimal cost function is presented by \eqref{cost}. The necessity proof is complete.

``Sufficiency:" Under Assumption \ref{ass2} and \ref{ass4}, if $P$ and $\bar{P}$ are the unique solution to \eqref{are1}-\eqref{are2} satisfying $P\geq 0$ and $P+\bar{P}\geq 0$, we will show that \eqref{control} stabilizes system \eqref{ps10} in mean square sense.

Following from \eqref{pnn1}-\eqref{pnn2}, the coupled ARE \eqref{are1}-\eqref{are2} can be rewritten as follows:
\begin{align}
  P&=Q+K'RK+(A+BK)'P(A+BK)\notag\\
  &+\sigma^{2}(C+DK)'P(C+DK),\label{ly01}\\
  P+\bar{P}&=Q+\bar{Q}+(K+\bar{K})'(R+\bar{R})(K+\bar{K})\notag\\
  &+[A+\bar{A}+(B+\bar{B})(K+\bar{K})]'(P+\bar{P})\notag\\
  &\times[A+\bar{A}+(B+\bar{B})(K+\bar{K})]\notag\\
  &+\sigma^{2}[C+\bar{C}+(D+\bar{D})(K+\bar{K})]'P\notag\\
  &\times[C+\bar{C}+(D+\bar{D})(K+\bar{K})],\label{ly2}
\end{align}
in which $K$ and $\bar{K}$ are respectively given as \eqref{K} and \eqref{KK}.

Recalling that the Lyapunov function candidate is denoted as in \eqref{lya} and using optimal controller \eqref{control}, we rewrite \eqref{lya1} as
\begin{align}\label{lya2}
&V(k,x_{k})-V(k+1,x_{k+1})\notag\\
&=E\{x_{k}'(Q+K'RK)x_{k}+Ex_{k}'[\bar{Q}+\bar{K}'RK+K'R\bar{K}\notag\\
&+\bar{K}'R\bar{K}+(K+\bar{K})'\bar{R}(K+\bar{K})]Ex_{k}\}\notag\\
&=E\{(x_{k}-Ex_{k})'(Q+K'RK)(x_{k}-Ex_{k})+Ex_{k}'[Q\notag\\
&+\bar{Q}+(K+\bar{K})'(R+\bar{R})(K+\bar{K})]Ex_{k}\}\notag\\
&=E(\mathbb{X}_{k}'\tilde{\mathcal{Q}}\mathbb{X}_{k})\geq 0.
\end{align}
where $\tilde{\mathcal{Q}}\hspace{-1mm}=\hspace{-1mm}\left[\hspace{-2mm}
  \begin{array}{cc}
   Q\hspace{-1mm}+\hspace{-1mm}K'RK& 0\\
   0        & Q\hspace{-1mm}+\hspace{-1mm}\bar{Q}\hspace{-1mm}+\hspace{-1mm}
   (K\hspace{-1mm}+\hspace{-1mm}\bar{K})'(R\hspace{-1mm}+\hspace{-1mm}\bar{R})(K\hspace{-1mm}+\hspace{-1mm}\bar{K})      \\
  \end{array}
\hspace{-2mm}\right]\hspace{-1mm}\geq\hspace{-1mm} 0$, and $\mathbb{X}_{k}=\left[
  \begin{array}{cc}
   \hspace{-1mm} x_{k}-Ex_{k}\hspace{-1mm}\\
   \hspace{-1mm} Ex_{k}     \hspace{-1mm}           \\
  \end{array}
\right]$.

Taking summation on both sides of \eqref{lya2} from $0$ to $N$ for any $N>0$, we have that
\begin{align}\label{lya3}
&~~\sum_{k=0}^{N}E(\mathbb{X}_{k}'\tilde{\mathcal{Q}}\mathbb{X}_{k})=V(0,x_{0})-V(N+1,x_{N+1})\notag\\
&=E(x_{0}'Px_{0})+(Ex_{0})'\bar{P}Ex_{0}\notag\\
&-[E(x_{N+1}'Px_{N+1})+(Ex_{N+1})'\bar{P}Ex_{N+1}]\notag\\
&=E(\mathbb{X}_{0}'\mathbb{P}\mathbb{X}_{0})-E(\mathbb{X}_{N+1}'\mathbb{P}\mathbb{X}_{N+1}),
\end{align}
in which $\mathbb{P}=\left[
  \begin{array}{cc}
   P& 0\\
   0        & P+\bar{P}      \\
  \end{array}
\right]$.

Using the symbols denoted above, mean-field system \eqref{ps10} with controller \eqref{control} can be rewritten as
\begin{align}\label{fb}
\mathbb{X}_{k+1}&=\tilde{\mathbb{A}}\mathbb{X}_{k}+\tilde{\mathbb{C}}\mathbb{X}_{k}w_{k},
\end{align}
where $\tilde{\mathbb{A}}=\left[\hspace{-2mm}
  \begin{array}{cc}
   A\hspace{-1mm}+\hspace{-1mm}BK& 0\\
   0        & A\hspace{-1mm}+\hspace{-1mm}\bar{A}\hspace{-1mm}+\hspace{-1mm}(B\hspace{-1mm}+\hspace{-1mm}\bar{B})
   (K\hspace{-1mm}+\hspace{-1mm}\bar{K})      \\
  \end{array}
\hspace{-2mm}\right]$ and $\tilde{\mathbb{C}}=\left[\hspace{-2mm}
  \begin{array}{cc}
   C\hspace{-1mm}+\hspace{-1mm}DK& C\hspace{-1mm}+\hspace{-1mm}\bar{C}\hspace{-1mm}+\hspace{-1mm}(D\hspace{-1mm}+\hspace{-1mm}\bar{D})
   (K\hspace{-1mm}+\hspace{-1mm}\bar{K})\\
   0        & 0\\
  \end{array}
\hspace{-2mm}\right]$.
Thus, the stabilization of system \eqref{ps10} with controller \eqref{control} is equivalent to the stability of system \eqref{fb}, i.e., $(\tilde{\mathbb{A}},\tilde{\mathbb{C}})$ for short.

Following the proof of \emph{Theorem 4} and \emph{Proposition 1} in \cite{zhangw}, we know that the exactly detectability of system \eqref{mf}, i.e., $(A,\bar{A},C,\bar{C},\mathcal{Q}^{1/2})$, implies that the following system is exactly detectable
\begin{equation}\label{mf01}
 \left\{ \begin{array}{ll}
 \mathbb{X}_{k+1}=\tilde{\mathbb{A}}\mathbb{X}_{k}+\tilde{\mathbb{C}}\mathbb{X}_{k}w_{k},\\
\tilde{Y}_{k}=\tilde{\mathcal{Q}}^{1/2}\mathbb{X}_{k}.
\end{array} \right.
\end{equation}
i.e., for any $N\geq 0$,
\begin{equation*}
  \tilde{Y}_{k}= 0, ~\forall~ 0\leq k\leq N~\Rightarrow~\lim_{k\rightarrow+\infty}E(\mathbb{X}_{k}'\mathbb{X}_{k})=0.
\end{equation*}

Now we will show that the initial state $\mathbb{X}_{0}$ is an unobservable state of system \eqref{mf01}, i.e., $(\tilde{\mathbb{A}},\tilde{\mathbb{C}},\tilde{\mathcal{Q}}^{1/2})$ for simplicity, if and only if $\mathbb{X}_{0}$ satisfies $E(\mathbb{X}_{0}'\mathbb{P}\mathbb{X}_{0})=0$.

In fact, if $\mathbb{X}_{0}$ satisfies $E(\mathbb{X}_{0}'\mathbb{P}\mathbb{X}_{0})=0$, from \eqref{lya3} we have
\begin{align}\label{lll}
0\leq \sum_{k=0}^{N}E(\mathbb{X}_{k}'\tilde{\mathcal{Q}}\mathbb{X}_{k})=-E(\mathbb{X}_{N+1}'\mathbb{P}\mathbb{X}_{N+1})\leq 0,
\end{align}
i.e., $\sum_{k=0}^{N}E(\mathbb{X}_{k}'\tilde{\mathcal{Q}}\mathbb{X}_{k})=0$. Thus, we can obtain
\begin{equation*}
  \sum_{k=0}^{N}E(Y_{k}'Y_{k})=\sum_{k=0}^{N}E(\mathbb{X}_{k}'\tilde{\mathcal{Q}}\mathbb{X}_{k})=0,
\end{equation*}
which means for any $k\geq 0$, $\tilde{Y}_{k}=\tilde{\mathcal{Q}}^{1/2}\mathbb{X}_{k}=0$. Hence, $\mathbb{X}_{0}$ is an unobservable state of system $(\tilde{\mathbb{A}},\tilde{\mathbb{C}},\tilde{\mathcal{Q}}^{1/2})$.

On the contrary, if we choose $\mathbb{X}_{0}$ as an unobservable state of $(\tilde{\mathbb{A}},\tilde{\mathbb{C}},\tilde{\mathcal{Q}}^{1/2})$, i.e., $\tilde{Y}_{k}=\tilde{\mathcal{Q}}^{1/2}\mathbb{X}_{k}\equiv 0$, $k\geq 0$. Noting that  $(\tilde{\mathbb{A}},\tilde{\mathbb{C}},\tilde{\mathcal{Q}}^{1/2})$ is exactly detectable, it holds $\lim_{N\rightarrow +\infty}E(\mathbb{X}_{N+1}'\mathbb{P}\mathbb{X}_{N+1})=0$. Thus, from \eqref{lya3} we can obtain that
\begin{equation}\label{130}
  E(\mathbb{X}_{0}'\mathbb{P}\mathbb{X}_{0})\hspace{-1mm}=\hspace{-1mm}
  \sum_{k=0}^{\infty}E(\mathbb{X}_{k}'\tilde{\mathcal{Q}}\mathbb{X}_{k})\hspace{-1mm}=
  \hspace{-1mm}\sum_{k=0}^{\infty}E(\tilde{Y}_{k}'\tilde{Y}_{k})\hspace{-1mm}=\hspace{-1mm}0.
\end{equation}

Therefore, we have shown that $\mathbb{X}_{0}$ is an unobservable state if and only if $\mathbb{X}_{0}$ satisfies $E(\mathbb{X}_{0}'\mathbb{P}\mathbb{X}_{0})=0$.

Next we will show system \eqref{ps10} is stabilizable in mean square sense in two different cases.

1) $\mathbb{P}>0$, i.e., $P>0$ and $P+\bar{P}>0$.

In this case, $E(\mathbb{X}_{0}'\mathbb{P}\mathbb{X}_{0})=0$ implies that $\mathbb{X}_{0}=0$, i.e., $x_{0}=Ex_{0}=0$. Following the discussions as above we know that system $(\tilde{\mathbb{A}},\tilde{\mathbb{C}},\tilde{\mathcal{Q}}^{1/2})$ is exactly observable. Thus it follows from Theorem \ref{succeed} that mean-field system \eqref{ps10} is stabilizable in mean square sense.

2) $\mathbb{P}\geq 0$.

Firstly, it is noticed from \eqref{ly01} and \eqref{ly2} that $\mathbb{P}$ satisfies the following Lyapunov equation:
\begin{equation}\label{ly4}
  \mathbb{P}=\tilde{\mathcal{Q}}+\tilde{\mathbb{A}}'\mathbb{P}\tilde{\mathbb{A}}+\sigma^{2}[\tilde{\mathbb{C}}^{(1)}]'\mathbb{P}\tilde{\mathbb{C}}^{(1)}
  +\sigma^{2}[\tilde{\mathbb{C}}^{(2)}]'\mathbb{P}\tilde{\mathbb{C}}^{(2)},
\end{equation}
where $\tilde{\mathbb{C}}^{(1)}\hspace{-1mm}=\hspace{-1mm}\left[\hspace{-2mm}
  \begin{array}{cc}
   C\hspace{-1mm}+\hspace{-1mm}DK\hspace{-2mm}&\hspace{-2mm} 0\\
   0        \hspace{-2mm}&\hspace{-2mm} 0\\
  \end{array}
\hspace{-2mm}\right]$, $\tilde{\mathbb{C}}^{(2)}\hspace{-1mm}=\hspace{-1mm}\left[\hspace{-2mm}
  \begin{array}{cc}
   0\hspace{-2mm}&\hspace{-2mm} C\hspace{-1mm}+\hspace{-1mm}\bar{C}\hspace{-1mm}+\hspace{-1mm}(D\hspace{-1mm}+\hspace{-1mm}\bar{D})
   (K\hspace{-1mm}+\hspace{-1mm}\bar{K})\\
   0        \hspace{-2mm}&\hspace{-2mm} 0\\
  \end{array}
\hspace{-2mm}\right]$ and $\tilde{\mathbb{C}}^{(1)}+\tilde{\mathbb{C}}^{(2)}=\tilde{\mathbb{C}}$.

Since $\mathbb{P}\geq 0$, thus there exists orthogonal matrix $U$ with $U'=U^{-1}$ such that
\begin{align}\label{upu}
U'\mathbb{P}U=\left[
  \begin{array}{cc}
   0& 0\\
   0        & \mathbb{P}_{2}     \\
  \end{array}
\right], \mathbb{P}_{2}>0.
\end{align}
Obviously from \eqref{ly4} we can obtain that
\begin{align}\label{ly5}
  U'\mathbb{P}U&=U'\tilde{\mathcal{Q}}U+U'\tilde{\mathbb{A}}'U\cdot U'\mathbb{P}U\cdot U'\tilde{\mathbb{A}}U\notag\\
  &+\sigma^{2}U'[\tilde{\mathbb{C}}^{(1)}]'U\cdot U'\mathbb{P}U\cdot U'\tilde{\mathbb{C}}^{(1)}U\notag\\
  &+\sigma^{2}U'[\tilde{\mathbb{C}}^{(2)}]'U\cdot U'\mathbb{P}U\cdot U'\tilde{\mathbb{C}}^{(2)}U.
\end{align}

Assume $U'\tilde{\mathbb{A}}U=\left[\hspace{-2mm}
  \begin{array}{cc}
   \tilde{\mathbb{A}}_{11}\hspace{-2mm}&\hspace{-2mm} \tilde{\mathbb{A}}_{12}\\
   \tilde{\mathbb{A}}_{21}\hspace{-2mm}&\hspace{-2mm} \tilde{\mathbb{A}}_{22}    \\
  \end{array}
\hspace{-2mm}\right]$, $U'\tilde{\mathcal{Q}}U=\left[\hspace{-2mm}
  \begin{array}{cc}
   \tilde{\mathcal{Q}}_{1}\hspace{-2mm}&\hspace{-2mm} \tilde{\mathcal{Q}}_{12}\\
   \tilde{\mathcal{Q}}_{21}\hspace{-2mm}&\hspace{-2mm} \tilde{\mathcal{Q}}_{2}    \\
  \end{array}
\hspace{-2mm}\right]$, $U'\tilde{\mathbb{C}}^{(1)}U=\left[\hspace{-2mm}
  \begin{array}{cc}
   \tilde{\mathbb{C}}_{11}^{(1)}\hspace{-2mm}&\hspace{-2mm} \tilde{\mathbb{C}}_{12}^{(1)}\\
   \tilde{\mathbb{C}}_{21}^{(1)}\hspace{-2mm}&\hspace{-2mm} \tilde{\mathbb{C}}_{22}^{(1)}    \\
  \end{array}
\hspace{-2mm}\right]$ and $U'\tilde{\mathbb{C}}^{(2)}U=\left[\hspace{-2mm}
  \begin{array}{cc}
   \tilde{\mathbb{C}}_{11}^{(2)}\hspace{-2mm}&\hspace{-2mm} \tilde{\mathbb{C}}_{12}^{(2)}\\
   \tilde{\mathbb{C}}_{21}^{(2)}\hspace{-2mm}&\hspace{-2mm} \tilde{\mathbb{C}}_{22}^{(2)}    \\
  \end{array}
\hspace{-2mm}\right]$, we have that
\begin{align*}U'\tilde{\mathbb{A}}'U\hspace{-1mm}\cdot\hspace{-1mm} U'\mathbb{P}U\hspace{-1mm}\cdot\hspace{-1mm} U'\tilde{\mathbb{A}}U&=\left[
  \begin{array}{cc}
   \tilde{\mathbb{A}}_{21}'\mathbb{P}_{2}\tilde{\mathbb{A}}_{21}& \tilde{\mathbb{A}}_{21}'\mathbb{P}_{2}\tilde{\mathbb{A}}_{22}\\
   \tilde{\mathbb{A}}_{22}'\mathbb{P}_{2}\tilde{\mathbb{A}}_{21}& \tilde{\mathbb{A}}_{22}'\mathbb{P}_{2}\tilde{\mathbb{A}}_{22}    \\
  \end{array}
\right],\\
U'\{\tilde{\mathbb{C}}^{(1)}\}'U\hspace{-1mm}\cdot\hspace{-1mm} U'\mathbb{P}U\hspace{-1mm}\cdot\hspace{-1mm} U'\tilde{\mathbb{C}}^{(1)}U&\hspace{-1mm}=\hspace{-1mm}\left[\hspace{-2mm}
  \begin{array}{cc}
   \{\tilde{\mathbb{C}}_{21}^{(1)}\}'\mathbb{P}_{2}\tilde{\mathbb{C}}_{21}^{(1)}\hspace{-3mm}&\hspace{-3mm} \{\tilde{\mathbb{C}}_{21}^{(1)}\}'\mathbb{P}_{2}\tilde{\mathbb{C}}_{22}^{(1)}\hspace{-1mm}\\
   \{\tilde{\mathbb{C}}_{22}^{(1)}\}'\mathbb{P}_{2}\tilde{\mathbb{C}}_{21}^{(1)}\hspace{-3mm}&\hspace{-3mm} \{\tilde{\mathbb{C}}_{22}^{(1)}\}'\mathbb{P}_{2}\tilde{\mathbb{C}}_{22}^{(1)}\hspace{-1mm}
  \end{array}
\hspace{-2mm}\right]\\
U'\{\tilde{\mathbb{C}}^{(2)}\}'U\hspace{-1mm}\cdot\hspace{-1mm} U'\mathbb{P}U\hspace{-1mm}\cdot\hspace{-1mm} U'\tilde{\mathbb{C}}^{(2)}U&\hspace{-1mm}=\hspace{-1mm}\left[\hspace{-2mm}
  \begin{array}{cc}
   \{\tilde{\mathbb{C}}_{21}^{(2)}\}'\mathbb{P}_{2}\tilde{\mathbb{C}}_{21}^{(2)}\hspace{-3mm}&\hspace{-3mm} \{\tilde{\mathbb{C}}_{21}^{(2)}\}'\mathbb{P}_{2}\tilde{\mathbb{C}}_{22}^{(2)}\hspace{-1mm}\\
   \{\tilde{\mathbb{C}}_{22}^{(2)}\}'\mathbb{P}_{2}\tilde{\mathbb{C}}_{21}^{(2)}\hspace{-3mm}&\hspace{-3mm} \{\tilde{\mathbb{C}}_{22}^{(2)}\}'\mathbb{P}_{2}\tilde{\mathbb{C}}_{22}^{(2)}\hspace{-1mm}
  \end{array}
\hspace{-2mm}\right]\end{align*}
Thus, by comparing each block element on both sides of \eqref{ly5} and noting $\mathbb{P}_{2}>0$, we have that $\tilde{\mathbb{A}}_{21}=0$, $\tilde{\mathbb{C}}_{21}^{(1)}=\tilde{\mathbb{C}}_{21}^{(2)}=0$ and $\tilde{\mathcal{Q}}_{1}=\tilde{\mathcal{Q}}_{12}=\tilde{\mathcal{Q}}_{21}=0$, i.e.,
\begin{align}\label{ly7}
U'\tilde{\mathbb{A}}U\hspace{-1mm}=\hspace{-1mm}\left[\hspace{-2mm}
  \begin{array}{cc}
   \tilde{\mathbb{A}}_{11}\hspace{-2mm}&\hspace{-2mm} \tilde{\mathbb{A}}_{12}\\
   0\hspace{-2mm}&\hspace{-2mm} \tilde{\mathbb{A}}_{22}    \\
  \end{array}
\hspace{-2mm}\right]\hspace{-1mm}, U'\tilde{\mathbb{C}}U\hspace{-1mm}=\hspace{-1mm}\left[\hspace{-2mm}
  \begin{array}{cc}
   \tilde{\mathbb{C}}_{11}\hspace{-2mm}&\hspace{-2mm} \tilde{\mathbb{C}}_{12}\\
   0\hspace{-2mm}&\hspace{-2mm} \tilde{\mathbb{C}}_{22}    \\
  \end{array}
\hspace{-2mm}\right]\hspace{-1mm},U'\tilde{\mathcal{Q}}U\hspace{-1mm}=\hspace{-1mm}\left[\hspace{-2mm}
  \begin{array}{cc}
   0\hspace{-2mm}&\hspace{-2mm} 0\\
   0 \hspace{-2mm}&\hspace{-2mm} \tilde{\mathcal{Q}}_{2}     \\
  \end{array}
\hspace{-2mm}\right],
\end{align}
where $\tilde{\mathcal{Q}}_{2}\geq 0$, $ \tilde{\mathbb{C}}_{11}= \tilde{\mathbb{C}}_{11}^{(1)}+ \tilde{\mathbb{C}}_{11}^{(2)}$, $ \tilde{\mathbb{C}}_{12}= \tilde{\mathbb{C}}_{12}^{(1)}+ \tilde{\mathbb{C}}_{12}^{(2)}$ and $ \tilde{\mathbb{C}}_{22}= \tilde{\mathbb{C}}_{22}^{(1)}+ \tilde{\mathbb{C}}_{22}^{(2)}$.

Substituting \eqref{upu} and \eqref{ly7} into \eqref{ly5} yields that
\begin{equation}\label{ly6}
 \mathbb{P}_{2}=\tilde{\mathcal{Q}}_{2}+\tilde{\mathbb{A}}_{22}' \mathbb{P}_{2}\tilde{\mathbb{A}}_{22}+\sigma^{2}\{\tilde{\mathbb{C}}_{22}^{(1)}\}' \mathbb{P}_{2}\tilde{\mathbb{C}}_{22}^{(1)}+
 \sigma^{2}\{\tilde{\mathbb{C}}_{22}^{(2)}\}' \mathbb{P}_{2}\tilde{\mathbb{C}}_{22}^{(2)}.
\end{equation}

Define $U'\mathbb{X}_{k}=\bar{\mathbb{X}}_{k}=\left[\hspace{-2mm}
  \begin{array}{cc}
    \bar{\mathbb{X}}_{k}^{(1)}\hspace{-1mm}\\
    \bar{\mathbb{X}}_{k}^{(2)}    \hspace{-1mm}           \\
  \end{array}
\hspace{-2mm}\right]$, where the dimension of $\bar{\mathbb{X}}_{k}^{(2)} $ is the same as the rank of $\mathbb{P}_{2}$. Thus, from \eqref{fb} we have
\begin{align*}
U'\mathbb{X}_{k+1}&=U'\tilde{\mathbb{A}}UU'\mathbb{X}_{k}+U'\tilde{\mathbb{C}}UU'\mathbb{X}_{k}w_{k},\end{align*} i.e.,
\begin{align}
\bar{\mathbb{X}}_{k+1}^{(\hspace{-0.3mm}1\hspace{-0.3mm})}&=\tilde{\mathbb{A}}_{11}\bar{\mathbb{X}}_{k}^{(\hspace{-0.3mm}1\hspace{-0.3mm})}
\hspace{-1mm}+\hspace{-1mm}\tilde{\mathbb{A}}_{12}\bar{\mathbb{X}}_{k}^{(\hspace{-0.3mm}2\hspace{-0.3mm})}
\hspace{-1mm}+\hspace{-1mm}(\tilde{\mathbb{C}}_{11}\bar{\mathbb{X}}_{k}^{(\hspace{-0.3mm}1\hspace{-0.3mm})}\hspace{-1mm}+\hspace{-1mm}\tilde{\mathbb{C}}_{12}\bar{\mathbb{X}}_{k}^{(\hspace{-0.3mm}2\hspace{-0.3mm})})w_{k},\label{lly1}\\
\bar{\mathbb{X}}_{k+1}^{(2)}&=\tilde{\mathbb{A}}_{22}\bar{\mathbb{X}}_{k}^{(2)}+\tilde{\mathbb{C}}_{22}\bar{\mathbb{X}}_{k}^{(2)}w_{k}.\label{lly2}
\end{align}

Next we will show the stability of $(\tilde{\mathbb{A}}_{22},\tilde{\mathbb{C}}_{22})$.

Actually, recall from \eqref{lya3} and \eqref{ly7}, we have that
\begin{align}\label{lya30}
&~~\sum_{k=0}^{N}E[(\bar{\mathbb{X}}_{k}^{(2)})'\tilde{\mathcal{Q}}_{2}\bar{\mathbb{X}}_{k}^{(2)}]=\sum_{k=0}^{N}E(\mathbb{X}_{k}'\tilde{\mathcal{Q}}\mathbb{X}_{k})\notag\\
&=E(\mathbb{X}_{0}'\mathbb{P}\mathbb{X}_{0})-E(\mathbb{X}_{N+1}'\mathbb{P}\mathbb{X}_{N+1})\notag\\
&=E[(\bar{\mathbb{X}}_{0}^{(2)})'\mathbb{P}_{2}\bar{\mathbb{X}}_{0}^{(2)}]-E[(\bar{\mathbb{X}}_{N+1}^{(2)})'\mathbb{P}_{2}\bar{\mathbb{X}}_{N+1}^{(2)}].
\end{align}
Similar to the discussions from \eqref{lll} to \eqref{130}, we conclude $\bar{\mathbb{X}}_{0}^{(2)}$ is an unobservable state of $(\tilde{\mathbb{A}}_{22},\tilde{\mathbb{C}}_{22},\tilde{\mathcal{Q}}_{2}^{1/2})$ if and only if $\bar{\mathbb{X}}_{0}^{(2)}$ obeys $E[(\bar{\mathbb{X}}_{0}^{(2)})'\mathbb{P}_{2}\bar{\mathbb{X}}_{0}^{(2)}]=0$. Since $\mathbb{P}_{2}>0$, thus $(\tilde{\mathbb{A}}_{22},\tilde{\mathbb{C}}_{22},\tilde{\mathcal{Q}}_{2}^{1/2})$ is exactly observable as discussed in 1). Therefore, following from Theorem \ref{succeed}, we know that
\begin{equation}\label{l2l}
  \lim_{k\rightarrow +\infty}E(\bar{\mathbb{X}}_{k}^{(2)})'\bar{\mathbb{X}}_{k}^{(2)}=0,
\end{equation}
i.e., $(\tilde{\mathbb{A}}_{22},\tilde{\mathbb{C}}_{22})$ is stable in mean square sense.

Thirdly, the stability of $(\tilde{\mathbb{A}}_{11},\tilde{\mathbb{C}}_{11})$ will be shown as below. We might as well choose $\bar{\mathbb{X}}_{0}^{(2)}=0$, then from \eqref{lly2} we have $\bar{\mathbb{X}}_{k}^{(2)}=0$ for any $k\geq 0$. In this case, \eqref{lly1} becomes
\begin{equation}\label{zz}
  \mathbb{Z}_{k+1}=\tilde{\mathbb{A}}_{11}\mathbb{Z}_{k}+\tilde{\mathbb{C}}_{11}\mathbb{Z}_{k}w_{k},
\end{equation}
where $\mathbb{Z}_{k}$ is the value of $\bar{\mathbb{X}}_{k}^{(1)}$ with $\bar{\mathbb{X}}_{k}^{(2)}=0$. Thus, for arbitrary initial state $\mathbb{Z}_{0}=\bar{\mathbb{X}}_{0}^{(1)}$, we have
\begin{equation}\label{lly3}
  E[\tilde{Y}_{k}'\tilde{Y}_{k}]=E[\mathbb{X}_{k}'\tilde{\mathcal{Q}}\mathbb{X}_{k}]=E[(\bar{\mathbb{X}}_{k}^{(2)})'\tilde{\mathcal{Q}}_{2}\bar{\mathbb{X}}_{k}^{(2)}]\equiv 0.
\end{equation}
From the exactly detectability of $(\tilde{\mathbb{A}},\tilde{\mathbb{C}},\tilde{\mathcal{Q}}^{1/2})$, it holds
\begin{align}\label{xxx}
\lim_{k\rightarrow +\infty}\hspace{-2mm}E(\bar{\mathbb{X}}_{k}'\bar{\mathbb{X}}_{k})\hspace{-1mm}=\hspace{-1mm}\lim_{k\rightarrow +\infty}\hspace{-2mm}E(\bar{\mathbb{X}}_{k}'U'U\bar{\mathbb{X}}_{k})\hspace{-1mm}=\hspace{-1mm}\lim_{k\rightarrow +\infty}\hspace{-2mm}E(\mathbb{X}_{k}'\mathbb{X}_{k})\hspace{-1mm}=\hspace{-1mm}0.
\end{align}
Therefore, in the case of $\bar{\mathbb{X}}_{0}^{(2)}=0$, \eqref{xxx} indicates that
\begin{align}\label{l1l}
 &~~ \lim_{k\rightarrow +\infty}\hspace{-2mm}E(\mathbb{Z}_{k}'\mathbb{Z}_{k})\hspace{-1mm}=\hspace{-1mm}
 \lim_{k\rightarrow +\infty}\hspace{-2mm} E[(\bar{\mathbb{X}}_{k}^{(1)})'\bar{\mathbb{X}}_{k}^{(1)}]\\
 &\hspace{-1mm}=\hspace{-2mm}\lim_{k\rightarrow +\infty}\hspace{-2mm} \{E[(\bar{\mathbb{X}}_{k}^{(\hspace{-0.3mm}1\hspace{-0.3mm})})'\bar{\mathbb{X}}_{k}^{(\hspace{-0.3mm}1\hspace{-0.3mm})}]
 \hspace{-1mm}+\hspace{-1mm}E[(\bar{\mathbb{X}}_{k}^{(\hspace{-0.3mm}2\hspace{-0.3mm})})'\bar{\mathbb{X}}_{k}^{(\hspace{-0.3mm}2\hspace{-0.3mm})}]\}
\hspace{-1mm}=\hspace{-2mm}\lim_{k\rightarrow +\infty}\hspace{-2mm}E(\bar{\mathbb{X}}_{k}'\bar{\mathbb{X}}_{k})\hspace{-1mm}=\hspace{-1mm}0.\notag
\end{align}
i.e., $(\tilde{\mathbb{A}}_{11},\tilde{\mathbb{C}}_{11})$ is mean square stable.

Finally we will show that system \eqref{ps10} is stabilizable in mean square sense. In fact, we denote $\tilde{\mathcal{A}}=\left[\hspace{-2mm}
  \begin{array}{cc}
   \tilde{\mathbb{A}}_{11}\hspace{-2mm}&\hspace{-2mm} 0\\
   0\hspace{-2mm}& \hspace{-2mm}\tilde{\mathbb{A}}_{22}    \\
  \end{array}
\hspace{-2mm}\right]$, $\tilde{\mathcal{C}}=\left[\hspace{-2mm}
  \begin{array}{cc}
   \tilde{\mathbb{C}}_{11}\hspace{-2mm}&\hspace{-2mm} 0\\
   0\hspace{-2mm}&\hspace{-2mm} \tilde{\mathbb{C}}_{22}    \\
  \end{array}
\hspace{-2mm}\right]$. Hence, \eqref{lly1}-\eqref{lly2} can be reformulated as
\begin{align}\label{133}
\bar{\mathbb{X}}_{k+1}\hspace{-1mm}=\hspace{-1mm}\{\tilde{\mathcal{A}}\bar{\mathbb{X}}_{k}+\left[\hspace{-1mm}
  \begin{array}{cc}
   \tilde{\mathbb{A}}_{12}\\
   0\\
  \end{array}
\hspace{-1mm}\right]\mathbb{U}_{k}\}\hspace{-1mm}+\hspace{-1mm}\{\tilde{\mathcal{C}}\bar{\mathbb{X}}_{k}\hspace{-1mm}+\hspace{-1mm}\left[\hspace{-1mm}
  \begin{array}{cc}
   \tilde{\mathbb{C}}_{12}\\
   0\\
  \end{array}
\hspace{-1mm}\right]\mathbb{U}_{k}\}w_{k},
\end{align}
where $\mathbb{U}_{k}$ is as the solution to equation \eqref{lly2} with initial condition $\mathbb{U}_{0}=\mathbb{X}_{0}^{(2)}$. The stability of $(\tilde{\mathbb{A}}_{11},\tilde{\mathbb{C}}_{11})$ and $(\tilde{\mathbb{A}}_{22},\tilde{\mathbb{C}}_{22})$ as proved above indicates that $(\tilde{\mathcal{A}},\tilde{\mathcal{C}})$ is stable in mean square sense. Obviously from \eqref{l2l} it holds $\lim_{k\rightarrow +\infty}E(\mathbb{U}_{k}'\mathbb{U}_{k})=0$ and $\sum_{k=0}^{\infty}E(\mathbb{U}_{k}'\mathbb{U}_{k})<+\infty$. By using \emph{Proposition 2.8} and \emph{Remark 2.9} in \cite{abb}, we know that there exists constant $c_{0}$ such that
\begin{align}\label{xu}
\sum_{k=0}^{\infty}E(\bar{\mathbb{X}}_{k}'\bar{\mathbb{X}}_{k})<c_{0}\sum_{k=0}^{\infty}E(\mathbb{U}_{k}'\mathbb{U}_{k})<+\infty.
\end{align}
Hence, $\lim_{k\rightarrow +\infty}E(\bar{\mathbb{X}}_{k}'\bar{\mathbb{X}}_{k})=0$ can be obtained from \eqref{xu}. Furthermore, it is noted from \eqref{xxx} that
\begin{align*}
&\lim_{k\rightarrow +\infty}\hspace{-2mm}E(x_{k}'x_{k})\hspace{-1mm}=\hspace{-1mm}\lim_{k\rightarrow +\infty}\hspace{-2mm}[(x_{k}\hspace{-1mm}-\hspace{-1mm}Ex_{k})'(x_{k}\hspace{-1mm}-\hspace{-1mm}Ex_{k})\hspace{-1mm}+\hspace{-1mm}Ex_{k}'Ex_{k}]\notag\\
&=\lim_{k\rightarrow +\infty}\hspace{-2mm}E(\mathbb{X}_{k}'\mathbb{X}_{k})\hspace{-1mm}=\hspace{-1mm}\lim_{k\rightarrow +\infty}\hspace{-2mm}E(\bar{\mathbb{X}}_{k}'\bar{\mathbb{X}}_{k})\hspace{-1mm}=\hspace{-1mm}0.
\end{align*}
Note that system $(\tilde{\mathbb{A}},\tilde{\mathbb{C}})$ given in \eqref{fb} is exactly mean-field system \eqref{ps10} with controller \eqref{control}. In conclusion, mean-field system \eqref{ps10} can be stabilizable in the mean square sense. The proof is complete.
\end{proof}

\ifCLASSOPTIONcaptionsoff
  \newpage
\fi

\end{document}